\documentclass[reqno,10pt]{amsart}
\usepackage{amsmath}
\usepackage{amssymb}
\usepackage{amsfonts}
\usepackage{graphicx}
\usepackage[abs]{overpic}
\usepackage{caption}
\usepackage{wrapfig}
\usepackage{float}
\usepackage{graphicx}
\usepackage{mathrsfs}
\usepackage{accents}
\usepackage{amsthm}
\usepackage{hyperref}
\usepackage{mathtools}

\usepackage{tikz-cd}
\usepackage[colorinlistoftodos]{todonotes}
\usetikzlibrary{matrix,decorations.text}
\usepackage{adjustbox}

\usepackage{enumitem}

\usepackage{nccmath}

\def\Xint#1{\mathchoice
{\XXint\displaystyle\textstyle{#1}}%
{\XXint\textstyle\scriptstyle{#1}}%
{\XXint\scriptstyle\scriptscriptstyle{#1}}%
{\XXint\scriptscriptstyle\scriptscriptstyle{#1}}%
\!\int}
\def\XXint#1#2#3{{\setbox0=\hbox{$#1{#2#3}{\int}$ }
\vcenter{\hbox{$#2#3$ }}\kern-.6\wd0}}

\def\dashint{\Xint-}

\setlist[enumerate,1]{font=\normalfont}
\setlist[itemize,1]{font=\normalfont}

\newlist{thmlist}{enumerate}{1}
\setlist[thmlist]{label=(\roman{thmlisti}),
	ref=(\roman{thmlisti}),font=\normalfont,
	noitemsep}

\newtheorem{theorem}{Theorem}[section]
\newtheorem{lemma}[theorem]{Lemma}
\newtheorem{proposition}[theorem]{Proposition}

\newtheorem{definition}[theorem]{Definition}

\newtheorem{rem}[theorem]{Remark}

\newcommand{\N}{\Bbb N}

\newcommand{\R}{\Bbb R}
\newcommand{\C}{\Bbb C}

\DeclareMathOperator*{\argmin}{arg\,min}

\def\Id{\mathbf{Id}}
\def\id{\mathbf{id}}

\def\eps{\varepsilon}

\def\dist{\operatorname{dist}}
\def\Xint#1{\mathchoice
   {\XXint\displaystyle\textstyle{#1}}%
   {\XXint\textstyle\scriptstyle{#1}}%
   {\XXint\scriptstyle\scriptscriptstyle{#1}}%
   {\XXint\scriptscriptstyle\scriptscriptstyle{#1}}%
   \!\int}
\def\XXint#1#2#3{{\setbox0=\hbox{$#1{#2#3}{\int}$}
     \vcenter{\hbox{$#2#3$}}\kern-.5\wd0}}

\def\dashint{\Xint-}
\newcommand\thickbar[1]{\accentset{\rule{.6em}{.8pt}}{#1}}

\usepackage{color}

\newcommand{\sym}{{\rm sym}}

\newcommand{\LLL}{\color{black}} 
\newcommand{\FFF}{\color{black}} 
\newcommand{\ZZZ}{\color{black}} 
 
\newcommand{\BBB}{\color{black}} 
\newcommand{\EEE}{\color{black}}

 \newcommand{\PPP}{\color{black}}

\newcommand{\SSS}{\color{black}}

\newcommand{\MMM}{\color{black}} 
\newcommand{\QQQ}{\color{black}} 
\newcommand{\AAA}{\color{black}} 
\numberwithin{equation}{section}

\title[One-dimensional viscoelastic von K\'{a}rm\'{a}n theories]{One-dimensional viscoelastic von K\'{a}rm\'{a}n theories derived from nonlinear  thin-walled beams}

\author{Manuel Friedrich}
\author{Lennart Machill}


\address[Manuel Friedrich]{%
  Department of Mathematics,  
  Friedrich-Alexander Universit\"at Erlangen-N\"urnberg,  
  Cauerstr.~11, D-91058 Erlangen, Germany, 
  \& Mathematics M\"{u}nster, 
  University of M\"{u}nster, 
  Einsteinstr.~62, D-48149 M\"{u}nster, Germany
}
\email{manuel.friedrich@fau.de}

\address[Lennart Machill]{Applied Mathematics,  
Universit\"{a}t M\"{u}nster, Einsteinstr. 62, D-48149 M\"{u}nster, Germany}
\email{lennart.machill@uni-muenster.de}

\begin{document}

\maketitle

\begin{abstract}
We derive  an effective one-dimensional limit from a three-dimensional Kelvin-Voigt model for viscoelastic thin-walled beams, in which the elastic and the viscous stress tensor comply \FFF with \EEE a frame-indifference principle.  The limiting system of equations  comprises stretching, bending, and twisting both in the elastic and the viscous stress. It coincides with the model already identified  via \cite{MFMKDimension} and \cite{MFLMDimension2D1D} by a successive dimension reduction, first from 3D to a 2D theory for  von K\'arm\'an plates and then from  2D to a 1D theory for ribbons. In the present paper, we complement the previous analysis by showing that the limit can also be obtained by  sending the height and width of the beam to zero \FFF simultaneously. \EEE Our arguments rely on the static $\Gamma$-convergence in \cite{Freddi2013}, on the abstract theory of metric gradient flows \cite{AGS}, and on evolutionary $\Gamma$-convergence \MMM \cite{S1}. \EEE
\end{abstract}

\section{Introduction}
Many three-dimensional models in continuum mechanics are nonlinear and nonconvex, resulting in  difficult numerical approximations and high computational \FFF costs. \EEE The derivation of simplified effective  theories still preserving the main features of the original systems plays therefore a significant role in current research. Prominent examples in that direction are  variational problems in dimension reduction where  a \FFF rigorous \EEE relationship between the full three-dimensional model and its lower-dimensional counterpart is \AAA achieved by \EEE means of $\Gamma$-convergence \cite{DalMaso:93}. 
Despite the long history of the subject in elasticity  (see \cite{Antmann:04, ciarlet2} for surveys), the theory flourished in the last twenty years triggered by the availability of \FFF the rigidity estimate in \EEE \cite{FrieseckeJamesMueller:02}.
 In the present paper, we continue the study of \cite{MFMKDimension, MFLMDimension2D1D} and \QQQ perform a dimension reduction \AAA for \EEE an evolutionary problem \EEE in the setting of \FFF viscoelastic \EEE materials.

In the purely elastic framework, there exists an extensive literature, among which we only mention the most relevant for our analysis. To this end, suppose that the reference configuration of the material is represented by a thin set $\Omega_{h,\delta} = (-\frac{l}{2},\frac{l}{2}) \times (-\frac{h}{2},\frac{h}{2})\times (-\frac{\delta}{2},\frac{\delta}{2}) $ with \FFF length $l$, \AAA width $h$, and height  $\delta$.  \EEE
After the rigorous justification of bending theory \cite{FrieseckeJamesMueller:02}, a complete hierarchy of plate models has been derived in  \cite{hierarchy, lecumberry} in the limit of vanishing  \AAA height $\delta \to 0$, \EEE particularly including \EEE the von K\'arm\'an theory. 
\FFF Starting \EEE from the latter, in \cite{Freddi2018} a further dimension reduction has been performed by sending the  \AAA width $h$ to zero, identifying \EEE an effective one-dimensional model for elastic ribbons. 
\FFF More generally, this theory also appears as $\Gamma$-limit from three-dimensional nonlinear elasticity in \cite{Freddi2012, Freddi2013}, in which a hierarchy of one-dimensional models \AAA was \EEE derived by considering the simultaneous limit $\delta \ll h \to 0$. \EEE
These studies differ from effective rod models \cite{ABP, Mora, Mora2} which are identified under  the assumption $h \sim \delta \AAA \to 0\EEE$. We mention that the above $\Gamma$-convergence approach can be complemented by convergence of equilibria, i.e., \AAA can be obtain \EEE effective limits of the three-dimensional momentum balance 
\begin{align}\label{eq:elasticequation}
-{\rm div}\ \partial_F W(\nabla w) = f  \EEE \ \ \  \text{ in $  \Omega_{h,\delta} \EEE $},
\end{align}
 see \QQQ e.g.\ \EEE \cite{mora-scardia} \AAA for $h \sim 1, \delta\to 0$ \EEE and  \cite{davoli2} for $h \sim \delta \to 0$. \EEE  Here, $f\colon \Omega_{h,\delta} \to \R$ is a volume density of external forces,  $\nabla w$ denotes the deformation gradient, $\QQQ\partial_F \EEE W:=\partial W/\partial F$ is the first {{Piola-Kirchhoff stress tensor}}, where $W\colon \R^{3 \times 3} \AAA \to [0,\infty]\EEE$ is a suitable elastic energy density\QQQ, \EEE and $F \in \R^{3 \times 3}$ is the place holder of $\nabla w$. \AAA The density is supposed to satisfy \EEE  usual assumptions in nonlinear elasticity, in particular frame indifference in the sense $W(F)=W(QF)$ for $Q\in{\rm SO}(3)$ and $F\in\R^{3\times 3}$, which implies that $W$ depends on the right Cauchy-Green strain tensor  $C:=F^\top F$.

The goal of the present paper is to pass to an effective one-dimensional description in a system of PDEs for \FFF nonlinear \EEE \QQQ \emph{viscoelastic} \EEE thin-walled beams, corresponding to the limits \FFF $\delta \ll h \to 0$. \EEE This complements the $\Gamma$-convergence result by {\sc Freddi,  Mora,  and Paroni}  \cite{Freddi2013}, and can be considered as the completion of our previous results in  \cite{MFMKDimension} and \cite{MFLMDimension2D1D}, which were the evolutionary analogs of \cite{lecumberry} and \cite{Freddi2018}, respectively.

We now describe our setting in more detail. We consider a quasistatic nonlinear model  for nonsimple viscoelastic materials in the Kelvin's-Voigt's rheology \FFF without inertia\EEE, which obeys the system of equations 
\begin{align}\label{eq:viscoel-nonsimple}
-{\rm div}\Big( \partial_F W(\nabla w)  -   \zeta_{h,\delta}  {\rm div} (\partial_ZP(\nabla^2 w)) + \partial_{\dot{F}}R(\nabla w,\partial_t \nabla w)  \Big) =  f \BBB  \EEE \ \ \  \text{ in $ [0,T] \times \EEE \Omega_{h,\delta}$}
\end{align}
for  some $\zeta_{h,\delta} >0$. In contrast to \eqref{eq:elasticequation}, the deformation mapping $w\colon[0,T]\times \Omega_{h,\delta}\to\R^3$  additionally depends on the time $t\in[0,T]$ with $T>0$.   The viscous stress $\partial_{\dot F} R$ can be derived from  a (pseudo)potential $R\colon  \R^{3 \times 3} \times \R^{3 \times 3} \to [0,\infty) \EEE $, playing an analogous role to the density $W$,  where $\dot F \in \R^{3 \times 3}$ is the placeholder of $\partial_t \nabla w$. As observed by {\sc Antman} \cite{Antmann}, the viscous stress tensor must comply with a time-continuous frame-indifference principle  meaning that $R(F,\dot F)=\tilde R(C,\dot C)$ for a suitable function $\tilde R$, where  $\dot C$ denotes the  time derivative of the right Cauchy-Green strain tensor $C$. In the following, we assume that \AAA $\partial_{\dot F}R$ \EEE is linear in $\dot F$, implying $R$ to be quadratic in $\dot F$. \AAA Note, however, that still a nonlinearity arises in  \EEE the viscous stress, due to its frame-indifference principle. The system \FFF \eqref{eq:viscoel-nonsimple} \EEE is further complemented with appropriate initial and boundary conditions, see \eqref{assumption:clampedboundary} \AAA below. \EEE   

Eventually, the remaining term, the so-called hyperstress, is induced by an additional term in the mechanical energy given by a convex and frame indifferent density $P \colon \R^{3\times 3\times 3} \to [0,+\infty)$ depending on the second gradient of $w$. In this sense, we treat a model for     second-grade  \EEE materials, originally introduced by {\sc Toupin}  \cite{Toupin:62,Toupin:64} to enhance compactness and regularity properties of problems in  mathematical elasticity. In particular, this approach \QQQ currently \EEE seems to be unavoidable to overcome issues in connection with time-continuous frame indifference and to prove the existence of weak solutions in a finite strain setting, see   \cite{MFMK, MielkeRoubicek}. (We refer, e.g., to \cite{demoulini, Lewick}  for some existence results with other \FFF solution concepts \EEE not needing second gradients.)  In a similar spirit, this \FFF idea \EEE has been essential in extensions to thermoviscoelasticity \cite{RBMFMK, MielkeRoubicek},   and non-interpenetration constraints \cite{Kroemer}, as well as  in  the derivation of linearized models \cite{MFMK} and a viscoelastic plate model of von K\'{a}rm\'{a}n type \cite{MFMKDimension}. We also refer to \cite{capriz, dunn}\QQQ,  \EEE where thermodynamical consistency of such models has been shown.

In the present contribution, we consider the limiting passage \FFF $\delta \ll h \to 0$, \EEE extending the purely elastic result in \cite{Freddi2013} to the  viscoelastic setting. More precisely, in \AAA Theorem \ref{maintheorem3}(iii),  \EEE we show that weak solutions to  \eqref{eq:viscoel-nonsimple} converge in a suitable sense  to a solution of 
\begin{align} \label{eq: equation-simp-intro11}
\begin{cases}
\FFF 0 \EEE = -  \mfrac{\rm d}{{\rm d}x_1} \bigg(C_W^0\Big(\xi_1^\prime+\LLL\frac{r}{2}\vert \xi_3^\prime\vert^2\EEE\Big) + C_R^0(\partial_t \xi_1^\prime+ \LLL r \EEE \xi_3^\prime \partial_t \xi_3^\prime )   \bigg) , &\vspace{0.1cm}\\
\FFF 0 \EEE = \mfrac{1}{12}   \mfrac{{\rm d}^2}{{\rm d}x_1^2}  \Big(C_W^0 \xi_2^{\prime\prime} +C_R^0 \partial_t \xi_2^{\prime\prime}   \Big) ,  &\vspace{0.1cm}\\
\FFF f^{1D} \EEE = - \LLL r \MMM \mfrac{{\rm d}}{{\rm d}x_1} \bigg(\Big(C_W^0\Big(\xi_1^{\prime} +\LLL\frac{r}{2}\vert \xi_3^\prime\vert^2\EEE \Big) + C_R^0(\partial_t \xi_1^{\prime} + \LLL r \EEE \xi_3^\prime \partial_t \xi_3^\prime) \Big)  \xi_3^\prime \bigg)  &\vspace{0.1cm}\\
\ \ \ \ \ \ \ \ \ \ \ +
\frac{1}{24} \mfrac{{\rm d}^2}{{\rm d}x_1^2}  \Big(\partial_1 Q_W^1(\xi_3^{\prime\prime},\theta^\prime)+\partial_1Q_R^1(\partial_t\xi_3^{\prime\prime},\partial_t \theta^\prime )\Big), &\vspace{0.1cm}\\
0 =  \mfrac{{\rm d}}{{\rm d}x_1} \Big(\partial_2 Q_W^1(\xi_3^{\prime\prime},\theta^\prime)+\partial_2 Q_R^1(\partial_t \xi_3^{\prime\prime},\partial_t \theta^\prime)\Big)   &\ \ \ \    \text{in } [0,T] \EEE  \times (-\frac{l}{2},\frac{l}{2})\QQQ, \EEE
\end{cases}
\end{align}
where $l>0$ is the length of the beam, and  the constants $C_W^0>0$ and $C_R^0>0$ as well as the quadratic forms $Q_W^1$ and $Q_R^1$ are related to $W$ and $R$, respectively. Moreover, $r$ is a parameter, specifying the relation of $\delta$ and the energy scaling, see \eqref{assumption:energyscalingthickness1} for details. The \QQQ functions \EEE $\xi_1$ and $\xi_2$ denote  an axial and  an orthogonal in-plane displacement, respectively, whereas $\xi_3$ denotes the  out-of-plane displacement and $\theta$ is a twist function. Finally, we suppose that $f$ in \eqref{eq:viscoel-nonsimple} only acts in \QQQ the \EEE$x_3$-direction and  $f^{1D}$ is its effective limit.

Our result is directly related to  \cite{MFMKDimension, MFLMDimension2D1D}, in the sense that in \cite{MFMKDimension}, the limit \AAA $\delta \to 0$ \EEE is considered to derive a von K\'arm\'an theory for viscoelastic plates, and subsequently in  \cite{MFLMDimension2D1D}\QQQ, \EEE a further dimension reduction \AAA $h \to 0$ \EEE is performed  to \AAA derive \eqref{eq: equation-simp-intro11} \EEE for $r=1$, representing a  one-dimensional model for viscoelastic ribbons. The result in the present contribution shows that the same limit can be obtained by \QQQ a \EEE simultaneous instead of \QQQ a \EEE successive  limiting passage \QQQ as \EEE $h,\delta \to 0$. This is a nontrivial issue as it is well known that there are multiple effective theories for one-dimensional objects such as beams and  rods depending \FFF on the \EEE ratio of the thickness in different directions. Note that we extend the result \cite{Freddi2013} also in the sense that we prescribe clamped boundary conditions which \FFF render \EEE  the analysis of geometric rigidity properties more delicate.

Both in the \FFF three- and one-dimensional \EEE setting, weak solutions to \eqref{eq:viscoel-nonsimple} and \eqref{eq: equation-simp-intro11} can be approximated by time-discrete solutions for a fixed time step $\tau$, and the limiting passage $h,\delta \to 0$ can be combined with the time-discrete approximation $\tau \to 0$. We obtain a corresponding commutativity result (Theorem \ref{maintheorem3}(ii),(iii)) which is illustrated in  Figure \ref{diagram}. 

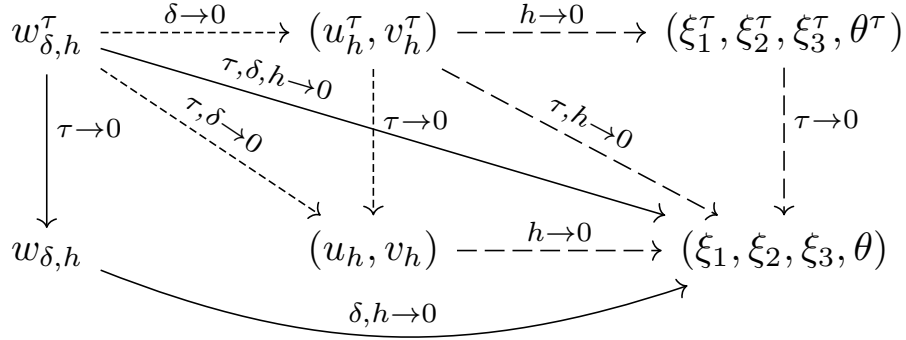
\begin{figure}[H]
	\centering
	\adjustbox{scale=1.5}{\begin{tikzcd}
		{w^\tau_{\delta,h}} \arrow[rr, "\delta \to 0",  dash pattern=on 2pt off 1pt] \arrow[rrdd, "{\tau,\delta \to 0}",sloped,  dash pattern=on 2pt off 1pt] \arrow[rrrrdd, "{\color{black}{\tau, \delta, h \to 0}}",black,sloped,pos= 0.3] \arrow[dd, "\tau \to 0",pos = 0.35,black] &  & {\QQQ(u^\tau_h,v^\tau_h)\EEE} \arrow[rrdd, "{\tau,h \to 0}",sloped,pos= 0.5 ,  dash pattern=on 5pt off 2pt] \arrow[rr, "h \to 0", dash pattern=on 5pt off 2pt] \arrow[dd, "\tau \to 0",pos = 0.35 ,   dash pattern=on 2pt off 1pt] &  & {\QQQ(\xi_1^\tau,\xi_2^\tau,\xi_3^\tau,\theta^\tau)\EEE} \arrow[dd, "\tau \to 0",pos = 0.35,  dash pattern=on 5pt off 2pt] \\
		&  &                                                                                   &  &                           \\
		{w_{\delta,h}}    \arrow[rrrr, "{ \delta , h \to 0 }", bend right=20,black]     &  & {(u_h,v_h)} \arrow[rr, "h \to 0", dash pattern=on 5pt off 2pt]                                                          &  &    {(\xi_1,\xi_2,\xi_3,\theta)}                          
		\end{tikzcd}
		}
	\caption{Illustration of the commutativity result, where the columns  correspond to \EEE the dimension,  and the rows indicate the (time-)discretized and continuous problems, respectively. Moreover, $\tau$ denotes the timestep, and $h$ and $\delta$ correspond to the  \AAA width and the height \EEE of the \AAA body, \EEE respectively. Whereas the dark arrows are considered in Proposition~\ref{maintheorem1} are Theorem~\ref{maintheorem3}, the dashed arrows \QQQ have \EEE already been addressed in \cite{MFMKDimension, MFMKJV, MFLMDimension2D1D}. \EEE} \label{diagram}

\end{figure}

\EEE

We briefly explain the scheme for the three-dimensional problem and highlight the choice of the dissipation. Given an initial value $w^0$, the first natural idea would consist in solving inductively the minimization problem 
\begin{align}\label{eq:time-discreteschemeintro}
w^n = \argmin\limits_w \left(\tau \mathcal{R}\left(w^{n-1}, \tfrac{w- w^{n-1}}{\tau}\right)+ E(w)\right)
\end{align}
for every $n\in \N$, where $E$ is the mechanical  energy defined by 
\begin{align*}
E(w) = \int_{\AAA\Omega_{h,\delta}\EEE} \Big( W(\nabla w(x)) + \zeta_{h,\delta} P(\nabla^2 w(x)) -   w(x) \cdot \EEE f (x)  \Big) \, {\rm d}x
\end{align*}
and $\mathcal{R}$ is the dissipation functional given by $\mathcal{R} (w,\partial_t w ) = \int_{\Omega_{h,\delta}} R (\nabla w(x), \partial_t\nabla  w(x) ) \, {\rm d}x$. Then, as $\tau \to 0$, limits of suitably defined interpolations of the time-discrete solutions converge to weak solutions of  \eqref{eq:viscoel-nonsimple}. Following the discussion in \cite[Section 2.2]{MOS}, we propose a slightly modified scheme by considering the minimization  problems
\begin{align}\label{eq: disc2}
w^n \in \argmin\limits_w \frac{1}{2\tau}\mathcal{D}^2(w,w^{n-1})+ E(w),
\end{align}
where $\mathcal{D}$ is a dissipation distance whose square is given by $\mathcal{D}^2(w_1,w_2) = \int_\Omega D(\nabla w_1, \nabla w_2)^2$\FFF. \EEE  Assuming that $\mathcal{D}$ is connected to $R$ via the relation $R(F,\dot{F}) := \lim_{\eps \to 0} \frac{1}{2\eps^2} D^2(F+\eps\dot{F},F)$  ensures that both minimization problems only differ from each other by lower order terms effectively leading to the same  \AAA system \eqref{eq:viscoel-nonsimple}. \EEE A main reason why we prefer to  consider  \eqref{eq: disc2} in place of \eqref{eq:time-discreteschemeintro} is the fact that, due to the separate frame indifference of  dissipation distances (see \eqref{eq: assumptions-D}(v) below for details), the \FFF functional \EEE minimized in \eqref{eq: disc2} is frame indifferent in contrast \eqref{eq:time-discreteschemeintro}, i.e., this important feature of the model is already satisfied on the time-discrete level.  

This approach also allows us to identify time-continuous limits as gradient flows in metric spaces \cite{AGS}, both in the 3D and the 1D setting.   Whereas for the one-dimensional system the existence of such solutions, so-called \emph{curves of maximal slope},    has  already been established in \cite{MFLMDimension2D1D}, the corresponding result for the three-dimensional thin material is new and a byproduct of our analysis, see Proposition~\ref{maintheorem1}. These curves can be related to weak solutions of the systems \eqref{eq:viscoel-nonsimple} and \eqref{eq: equation-simp-intro11}, see \cite[Theorem~2.1(iii)]{MFMK} and  \cite[Theorem 2.2(ii)]{MFLMDimension2D1D}, respectively, but provide additional information as a certain energy-dissipation-balance is satisfied, see \eqref{maximalslope} below for details. This balance is not just of independent interest but at the core of our approach to relate the 3D \QQQ to \EEE the 1D \AAA model \EEE by resorting to the theory of convergence of gradient flows introduced in \cite{Ortner, S1, S2}.  In using this theory, the challenge \AAA is \EEE 
that besides $\Gamma$-convergence additional conditions are  needed to ensure convergence of gradient flows. We refer to \cite[Introduction]{MFMKDimension} for a \FFF detailed \EEE account \FFF of \EEE the relevant issues and the main proof strategy in the context of \AAA dimension-reduction \EEE problems.

Let us highlight \QQQ again \EEE that the approach of  nonsimple materials is used due to the presence of viscous
effects. Indeed, \AAA without second gradients, \EEE already the existence of  time-discrete solutions \eqref{eq: disc2} cannot be guaranteed. This is due to the fact that the structure of $\mathcal{D}$ appears to be incompatible with quasiconvexity (see \cite{MOS} for a detailed discussion), and thus  weak lower semicontinuity cannot be expected. Moreover,  in a similar spirit to \cite{RBMFMK, MielkeRoubicek, Kroemer}, it is essential to show that the topology induced by $\mathcal{D}$ is equivalent to the weak $H^1$-topology in order to obtain a priori bounds on the strain rate, cf.\ Lemma \ref{th: metric space}(v). This deeply relies on a \emph{generalized Korn inequality} (see Theorem \ref{pompe}) which requires certain properties  of the deformation gradient guaranteed by higher regularity of the spatial gradients.


In the following, we suppose that the  material is homogeneous, i.e., neither the elastic stored energy density nor  the dissipation \FFF depends \EEE on  \AAA the \EEE  material point.   Moreover, for technical reasons, we will essentially  restrict our analysis to materials with zero Poisson's ratio, such as cork, see \eqref{quadraticforms} and \eqref{eq: new-orth}, and also Remark~\ref{rem: VDE} for a possible generalization.   Such an assumption, also present in other works (see e.g.~\cite{BK,MFMKDimension,MFLMDimension2D1D}), simplifies the analysis as it excludes nontrivial relaxation effects in the passage from the 3D to the 1D model.

Let us also mention the related issue of deriving effective theories for problems with inertia but without viscosity.  \AAA This \EEE has been the \EEE subject of \cite{Abels3, AMM} and  \cite{Abels1, Abels2} in the context of plate and rod models, respectively.  Combining inertial and viscosity effects for thin structures, also in connection with temperature \cite{RBMFMK, MielkeRoubicek}, will be \FFF the \EEE subject of future research.

The plan of the paper is as follows. In Section 2, we introduce the three- and one-dimensional models in more detail and state our main results.
Whereas Section 3 is devoted to the collection of results concerning the theory of gradient flows in metric spaces, the abstract theory is adapted to our model in Sections 4--7. First, in Section \ref{sec:rigidity} we adapt the rigidity \AAA estimates \EEE of \cite{Freddi2013} to our setting with clamped boundary conditions.  Section \ref{sec: 3d-1d} is devoted to \FFF the \EEE properties of the three- \FFF and \EEE one-dimensional system. In Section \FFF6, \EEE we discuss the main convergence results to apply the abstract theory  \cite{S1}. In particular, we show  the lower semicontinuity of the local slopes (see Theorem~\ref{theorem: lsc-slope}), which is the key difficulty of our paper. Eventually, the \FFF proofs \EEE of the main results are contained  in Section \ref{sec: mianmain}. Some elementary lemmata \AAA about \EEE the energies and the dissipations are postponed to Appendix \ref{sec:Appendix}. We close the introduction with some basic notation.

\subsection*{Notation} 

In what follows, we use standard notation for Lebesgue spaces, $L^p(\Omega)$, which are measurable maps on $\Omega\subset\R^d$, \MMM $d=1,2,3$, \EEE integrable with the $p$-th power (if $1\le p<+\infty$) or essentially bounded (if $p=+\infty$).    Sobolev spaces, \QQQ written $W^{k,p}(\Omega)$\EEE, denote the linear spaces of  maps  which, together with their weak derivatives up to the order $k\in\N$, belong to $L^p(\Omega)$. 
   \BBB Moreover, for a function $\hat v \in W^{k,p}(\Omega)$ the set $W^{k,p}_{\hat v}(\Omega)$ contains maps from $W^{k,p}(\Omega)$ having boundary conditions (in the sense of traces) up to the   $(k-1)$-th  order with respect to $\hat v$. \EEE
If the target space is a Banach space $E \neq \R$, we use the usual notion of Bochner-Sobolev spaces, written $W^{k,p}(\Omega;E)$.  For more details on Sobolev spaces and their \FFF duals, \EEE we refer to \cite{AdamsFournier:05}. Whereas $\nabla$ and $\nabla^2$ denote the spatial gradient and Hessian, respectively, the symbol $^\prime$ is used for derivatives of functions depending solely on one spatial \AAA variable. \EEE  Further, $\partial_t$ indicates a time derivative and $\delta_{ij}$ denotes the Kronecker delta function. Finally, $\vert A \vert$ \QQQ stands for \EEE\FFF the Frobenius norm \EEE of a matrix $A \in \R^{3\times 3}$, \QQQ and ${\rm sym}(A) = \frac{1}{2}(A^\top + A)$ and ${\rm skew}(A) = \frac{1}{2}(A - A^\top)$ indicate the symmetric and skew-symmetric part, respectively. \EEE

 \section{The model  and main results}\label{section:2}
 \subsection{The three-dimensional model}
 In this subsection, we describe the model and discuss the variational setting. \LLL Following the discussion in \cite[Section 2.2]{MOS} and \cite[Section 2]{MFMK}, we model \eqref{eq:viscoel-nonsimple} as a metric gradient flow. For this purpose, we need to specify three main \AAA ingredients: \EEE the state space that contains admissible deformations of the material, \FFF the \EEE elastic energy that drives the evolution, and \FFF the \EEE dissipation mechanism represented by a distance.
 \EEE We consider \MMM thin-walled beams with rectangular cross-section, i.e., the reference configuration of the material is   a cuboid of the form \EEE
 \begin{align*}
 	\Omega_h := \LLL I \EEE \times \omega_h := (-\tfrac{l}{2},\tfrac{l}{2}) \times\{ (z_2,z_3) : \vert z_2 \vert < h/2, \vert z_3 \vert < \delta_h/2 \}  \subset \R^3.
 \end{align*}
 Here, $(\delta_h)_h$ is a \MMM null sequence \EEE satisfying $\lim_{h\to 0} \tfrac{\delta_h}{h} = 0$, i.e., the width (corresponding to the $x_2$-coordinate) tends to \MMM zero \EEE much slower than the height (described by the $x_3$-coordinate). 
 
 \subsection*{Elastic energy}
 We define the \emph{elastic energy}  per unit cross-section associated with a deformation $w\colon \Omega_h \to \R^3$ by
 \begin{align*}
 	E(w) = \tfrac{1}{h\delta_h}\int_{\Omega_h} W(\nabla w(x)) \, \MMM  {\rm d}x \EEE + \tfrac{\zeta_h}{h\delta_h}\int_{\Omega_h} P(\nabla^2 w(x)) \, \MMM  {\rm d}x \EEE - \tfrac{1}{h\delta_h}\int_{\Omega_h} f^{3D}_h(x) \, \FFF w_3(x) \, \MMM  {\rm d}x. \EEE
 \end{align*}
  Here, $W\colon \R^{3 \times 3} \to [0,\infty]$ denotes  a single well, frame-indifferent stored energy density with the usual assumptions in nonlinear elasticity. \QQQ More precisely,  we \EEE  suppose  that there exists $c>0$ such that
 \begin{align}\label{assumptions-W}
 \begin{split}
 {\rm (i)}& \ \ W \text{ \LLL is continuous and \BBB $C^3$ \EEE in a neighborhood of $SO(3)$},\\
 {\rm (ii)}& \ \ W \text{ \LLL is frame indifferent, i.e., } W(QF) = W(F) \text{ for all } F \in \R^{3 \times 3}, Q \in SO(3),\\
 {\rm (iii)}& \ \ W(F) \ge c\dist^2(F,SO(3)) \QQQ \text{ for all }F\in \R^{3\times 3}, \quad \EEE  W(F) = 0 \text{ iff } F \in SO(3),
 \end{split}
 \end{align}
 where $SO(3) = \lbrace Q\in \R^{3 \times 3}: Q^\top Q = \Id, \, \det Q=1 \rbrace$. Besides the elastic energy density $W$ depending on the deformation gradient, we also  consider a \emph{strain gradient energy term} $P$ depending on the Hessian $\nabla^2 w$, adopting the concept of 2nd-grade nonsimple materials\LLL , see \cite{Toupin:62}. More specifically, for $p>3$, let $P\colon \R^{3\times 3 \times 3} \to [0,\SSS +\infty)\EEE$ satisfy
 \begin{align}\label{assumptions-P}
 \begin{split}
 {\rm (i)}& \ \ \text{frame indifference, i.e., } P(QZ) = P(Z) \text{ for all } Z \in \R^{3 \times 3 \times 3}, Q \in SO(3),\\
 {\rm (ii)}& \ \ \text{convexity and $C^1$-regularity},\\
 {\rm (iii)}& \ \ \text{a growth condition, namely} \ \    c_1 |Z|^p \le P(Z) \le c_2 |Z|^p, \\&   \ \ \ \  \ \ \ \ \ \   \ \ \ \  \ \ \ \ \ \  \ \ \ \  \ \ \ \ \ \  \ \ \ \  \ \ \ \ \ \ |\partial_{Z} P(Z)|  \le c_2 |Z|^{p-1}  \text{ for all $Z \in \R^{3 \times 3 \times 3}$ }
 \end{split}
 \end{align}
 for $0<c_1<c_2$. The contribution to the model of the latter is measured by $(\zeta_h)_h$ indicating a null sequence.  Finally, \LLL $f^{3D}_h \in L^2(\Omega_h)$ \MMM denotes \EEE a \LLL force acting in the $x_3$-direction. \AAA For simplicity, we assume \EEE $ f^{3D}_{h}$ to be independent of $x_2$ and $x_3$ and write $ f^{3D}_{h}\colon I \to \R$ with a slight abuse of notation. \FFF
  The theory also holds for more general forces which may
 depend on $x_2$, see \cite[Section 4]{Freddi2013}. \EEE  We postpone a more precise definition of $\zeta_h$ and $f^{3D}_h$ to \eqref{assumption:penalizationscale} and \eqref{eq: forces}.
%

 %
 %
 
 \subsection*{Dissipation mechanism:} 
 Consider now  time-dependent deformations $w\colon[0,T]\times \Omega_h \to \R^3$. In contrast to elasticity, viscosity is not only related to the strain $\nabla w$ but also to the strain rate $\partial_t \nabla  w$. \SSS It can be expressed in terms of $R(\nabla w, \partial_t \nabla w)$ for the \emph{dissipation potential} $R\colon \R^{3 \times 3} \times \R^{3 \times 3} \to [0,\infty)$ given in \eqref{eq:viscoel-nonsimple}. Due to our formulation as a metric gradient flow, we consider a corresponding distance as follows:   we introduce  \EEE $\mathcal{D}$, defined by $$\mathcal{D}(w_1,w_2) = \Big(\int_\Omega D(\nabla w_1(x), \nabla w_2(x))^2 {\rm d}x \Big)^{1/2}$$ for $w_1,w_2\colon \Omega_h \to \R^3$.
Here, for some $c>0$ we assume that the density  $D\colon GL_+(3) \times  GL_+(3) \to [0,\infty)$  satisfies \QQQ for all $F_1,F_2\in GL_+(3) := \lbrace F \in \R^{3 \times 3}: \det F>0 \rbrace$\EEE
 \begin{align}\label{eq: assumptions-D}
 {\rm (i)} & \ \ D(F_1,F_2)> 0 \text{ if } F_1^\top F_1 \neq F_2^\top F_2,\notag \\
 {\rm (ii)} & \ \ D(F_1,F_2) = D(F_2,F_1),\notag\\
 {\rm (iii)} & \ \ D(F_1,F_3) \le D(F_1,F_2) + D(F_2,F_3),\notag \\
 {\rm (iv)} & \ \ \text{$D(\cdot,\cdot)$ is $C^3$ in a neighborhood of $SO(3) \times SO(3)$},
 \\
 {\rm (v)}& \ \ \text{Separate frame indifference, i.e., } D(Q_1F_1,Q_2F_2) = D(F_1,F_2)  \text{ for all } Q_1,Q_2 \in SO(3), \notag\\ 
\SSS {\rm (vi)} & \ \ \partial^2_{F_1^2} D^2(\Id, \Id) [G,G] \ge c\,|\QQQ \sym(G) \EEE|^2 \ \ \forall\,  G \in \R^{3 \times 3} \notag. \EEE
 \end{align} 
\QQQ We \EEE point out that the   Hessian of  $D^2$ in direction of $F_1$ at $(F_1,F_2)$, denoted by $\partial^2_{F_1^2} D^2(F_1,F_2)$, is a \QQQ fourth-order \EEE tensor. Note that (i) implies that $D$ is a true distance when restricted to {{positive definite}} \AAA symmetric \EEE matrices.    \SSS Moreover, (vi) is a natural condition as $\R^{3 \times 3}_{\rm skew} = \lbrace A  \in \R^{3 \times 3}: A=-A^\top \rbrace$ is necessarily contained in the kernel of $\partial^2_{F_1^2} D^2(\Id, \Id)$ by (v), see Lemma \ref{lemma: ele}.   The relation of  $D$ and the dissipation potential $R$  is given by 
 \begin{align}\label{intro:R}
 	R(F,\dot{F}) := \lim_{\eps \to 0} \frac{1}{2\eps^2} D^2(F+\eps\dot{F},F) = \frac{1}{4} \partial^2_{F_1^2} D^2(F,F) [\dot{F},\dot{F}]
 \end{align}
 for $F \in GL_+(3)$ and  $\dot{F} \in \R^{3 \times 3}$, whenever $\partial^2_{F_1^2} D^2(F,F)$ exists.   The second equality  \SSS follows by  \EEE a Taylor expansion. In addition, we point out that (v) guarantees that $R$ satisfies frame indifference in the sense
 \begin{align*}
 R(F,\dot{F}) = R(QF,Q(\dot{F} + AF))  \ \ \  \forall  Q \in SO(3),\, A \in \R^{3 \times 3}_{\rm skew}
 \end{align*}
 for all $F \in GL_+(3)$ and $\dot{F} \in \R^{3 \times 3}$, see \cite[Lemma 2.1]{MOS}. This corresponds to a time-dependent version of frame indifference, and we refer to \cite{Antmann, MOS} for a \SSS thorough \EEE discussion.
\SSS A further consequence of (v),(vi) is that  $R$ depends only the right Cauchy-Green strain tensor $C:= F^{\top}F$ and its time derivative $\dot C= \dot{F}^{\top}F+F^{\top}\dot F$, that it is quadratic in $\dot C$, and that 
 \begin{align}\label{eq:quadraticlowerbound}
 	 R(F,\dot F) \geq c \vert \dot C \vert ^2,
 \end{align}
 see Lemma \ref{lemma: ele}. This condition corresponds to \cite[assumption (2.30e)]{MielkeRoubicek} and is essential there in order to derive good compactness properties for establishing solutions in the three-dimensional setting. \EEE  One possible example of $D$ satisfying 
 \eqref{eq: assumptions-D} might be $D(F_1,F_2)= |F_1^\top F_1-F_2^\top F_2|$.   This leads  to $R(F,\dot F)=|\BBB  F^\top \dot F + \dot F^\top F \EEE |^2/2$ which is a standard choice. \EEE  For further examples  we refer  to \cite[Section~2.3]{MOS}.

 \EEE

 \subsection*{Rescaling}\label{sec:rescaling}
 For the limiting passage, it is more convenient to work on \AAA the \EEE rescaled domain $\Omega := I \times (-1/2,1/2) \times (-1/2, 1/2)$ that does not depend on $h$. \MMM For \LLL notational \MMM convenience,  we denote by \EEE $S:=\LLL I\EEE\times (-1/2,1/2)$ the scaled $x_1$-$x_2$ cross-section  and \MMM by \EEE $\omega:=(-1/2,1/2) \times (-1/2,1/2)$ the scaled $x_2$-$x_3$ cross-section. \LLL For a deformation $w\colon \Omega_h \to \R^3$, we let $p_h\colon \Omega \to \Omega_h$, $p_h(x_1,x_2,x_3) = (x_1, hx_2,\delta_hx_3)$ be the projection of $\Omega$ onto $\Omega_h$ and introduce the scaled deformation mapping $y\colon \Omega \to \R^3$ \AAA by \EEE $y=w\circ p_h$. \EEE \LLL For \MMM a \EEE smooth \LLL function \EEE $y\colon \Omega \to \R^3$, we \AAA define the scaled gradient   as \EEE $\nabla_h y = ({y,}_1, \tfrac{{y,}_2}{h}, \tfrac{{y,}_3}{\delta_h})$, \FFF where the subscript indicates the directional derivative along the $i$-th unit vector. \EEE Moreover, $\nabla_h^2$ denotes the scaled Hessian and is defined by 
 $$(\nabla_h^2y)_{ijk} := h^{-\delta_{2j}-\delta_{2k}}\delta_h^{-\delta_{3j} - \delta_{3k}} (\nabla^2y)_{ijk} \quad \text{for $i,j,k \in \{1,2,3\} \quad$ and $ \quad(\nabla^2 y)_{ijk}:= (\nabla^2 y_i)_{jk}$}.
 $$
Let $(\eps_h)_h$ \MMM be \EEE a null sequence representing the energy scaling. Then, by a change of variables, we see that the rescaled energy, defined by $\phi_h(y) = \tfrac{1}{\eps_h^2}E(y\circ p_h^{-1})$, satisfies
 \begin{align}\label{assumption:energyscaling}
 	\phi_h(y) = \frac{1}{\eps_h^2} \int_\Omega W(\nabla_h y(x)) \, \MMM  {\rm d}x \EEE +\frac{\zeta_h}{\eps_h^2} \int_\Omega P(\nabla^2_h y(x)) \, \MMM  {\rm d}x \EEE - \frac{1}{\eps_h^2} \int_{\Omega}  f^{3D}_h(x_1)\, \FFF y_3(x) \, \MMM  {\rm d}x \EEE
 \end{align}
 for all $y \in W^{2,p}(\Omega;\R^3)$.  Concerning the viscosity part of the model, we \EEE define the scaled dissipation distance as
 \begin{align}\label{assumption:dissipationsclae}
 	\mathcal{D}_h(y_0,y_1) := \left( \frac{1}{\eps_h^2} \int_\Omega D^2(\nabla_h y_0, \nabla_h y_1) \, {\rm d}x \right)^{1/2}
 \end{align}
 \SSS    for all $y_0, y_1 \in W^{2,p}(\Omega;\R^3)$. \EEE
 
 \subsection*{Clamped boundary conditions} 
 
 In contrast to \AAA the \EEE model studied in \EEE \cite{Freddi2013}, we consider a problem  with Dirichlet boundary conditions on $ \Gamma :=\{-\tfrac{l}{2},\tfrac{l}{2}\} \times \MMM \omega\EEE$. Given functions $\hat \xi_1 \in W^{2,p}(\MMM I \EEE )$, $\hat \xi_2, \hat \xi_3 \in W^{3,p}(I)$, $p>3$,  we introduce the set of admissible configurations by
 \begin{align}\label{assumption:clampedboundary}
 	\mathscr{S}^{3D}_h := \Bigg\{ y^{h} \in W^{2,p}(\Omega;\R^3) : y^{h} = \begin{pmatrix}
 		x_1 \\ h x_2 \\ \delta_h x_3 \\
 	\end{pmatrix} + \eps_h \begin{pmatrix}
 		\hat \xi_1 - x_2 \hat \xi_2' - x_3 \hat \xi_3'   \\ \hat \xi_2/h  \\ \hat \xi_3/ \delta_h \\
 	\end{pmatrix} \quad {\rm on } \quad \Gamma \Bigg\}.
 \end{align}
\AAA Note that the different components of the boundary conditions are coupled.  Such a structure  \EEE is a typical  in \AAA dimension-reduction \EEE problems to ensure compatibility of the recovery sequence. We refer e.g.\ to \MMM \cite{MFMKDimension, MFLMDimension2D1D, lecumberry}.

 \subsection*{Rescaled equations of nonlinear viscoelasticity}

\SSS As a preparation for the  formulation of \EEE the rescaled equations of nonlinear viscoelasticity, we  \SSS introduce the    scaled (distributional) divergence. First, for  $g \in L^1(\Omega;\R^3)$  \EEE we define   ${\rm div}_h g$  by  ${\rm div}_h g = \partial_1 g_1+ \frac{1}{h}\partial_2 g_2 + \frac{1}{\delta_h}\partial_3 g_3$. Then, for $i,j \EEE \in \lbrace 1,2, 3\rbrace$, we denote by $(\partial_ZP(\nabla^2_h y))_{ij*}$ the vector-valued function  $((\partial_ZP(\nabla^2_h y))_{ijk})_{k=1,2,3}$, \SSS and let   
 \begin{align*}
\big({\rm div}_h (\partial_ZP(\nabla^2_h y))\big)_{ij} =   {\rm div}_h (\partial_ZP(\nabla^2_h y))_{ij*}, \ \ \ \  i,j \EEE \in \lbrace 1,2, 3\rbrace
 \end{align*}   
 for $y \in \mathscr{S}_h^{3D}$. \SSS Rescaling  of \eqref{eq:viscoel-nonsimple} leads to the system of   equations \EEE
 \begin{align}\label{nonlinear equation}
 	\begin{cases} -  {\rm div}_h \Big( \partial_FW(\nabla_h y) - \MMM \zeta_h \SSS {\rm div}_h (\partial_ZP(\nabla^2_h y)) \EEE + \partial_{\dot{F}}R(\nabla_h y,\partial_t \nabla_h y)  \Big) =   \FFF f^{3D}_h e_3 & \text{in } [0,\infty) \times \Omega, \\
 		y(0,\cdot) = y^h_0 & \text{in } \Omega ,\\
 		y(t,\cdot) \in \mathscr{S}_h^{3D} &\text{for } t\in [0,\infty)
 	\end{cases}
 \end{align}
 for some \MMM initial datum \EEE $y^h_0 \in \mathscr{S}_h^{3D}$, where $\partial_FW(\nabla_h y)-\zeta_h\SSS  {\rm div}_h (\partial_ZP(\nabla^2_h y)) \EEE$ denotes the    \emph{first Piola-Kirchhoff stress tensor} and $\partial_{\dot{F}}R(\nabla_h y,\partial_t \nabla_h y)$ the \emph{viscous stress} with $R$ as introduced in \eqref{intro:R}. \FFF Moreover, $e_3$ indicates the normal vector pointing in the \QQQ $x_3$-direction. 
  \AAA We also implicitly assume zero Neumann boundary conditions for the stress and the hyperstress on $\partial\Omega \setminus \Gamma$, and on the lateral boundary there arise additional Neumann conditions from the second deformation gradient. We do not include the conditions here but refer to \cite{Kroemer} for details. \EEE  

 \subsection*{Existence of solutions to the 3D-model} 
 \MMM To guarantee existence of weak solutions to \eqref{nonlinear equation}, \EEE we introduce \MMM an \EEE approximation scheme solving suitable time-incremental minimization problems. We consider a fixed time step \AAA $\tau >0$, and \EEE set $Y^0_{h,\tau} = y_0^h$. \EEE   Whenever $Y^0_{h,\tau}, \ldots, Y^{n-1}_{h,\tau}$ are known, $Y^n_{h,\tau}$ is defined as (if existent)
 \begin{align}\label{incremental}
 	Y^n_{h,\tau} = {\rm argmin}_{y \in\mathscr{S}_h^{3D}  } \Phi_h(\tau, Y^{n-1}_{h,\tau}; y), \ \ \ \Phi_h(\tau,y_0; y_1):= \BBB \phi_h(y_1)  \EEE +  \frac{1}{2\tau} \mathcal{D}_h^2(y_0,y_1), 
 \end{align}
 where $\phi_h$ and $\mathcal{D}_h$ are defined in \eqref{assumption:energyscaling} and \eqref{assumption:dissipationsclae}. Suppose that, for a choice of $\tau$, a sequence $(Y^n_{h,\tau})_{n \in \N}$ solving  \eqref{incremental} exists. We define the  piecewise constant interpolation by
 \begin{align}\label{ds}
 	\tilde{Y}_{h,\tau} \BBB(0,\cdot) \EEE = Y^0_{h,\tau}, \ \ \ \tilde{Y}_{h,\tau} \BBB (t,\cdot) \EEE = Y^n_{h,\tau}  \ \text{for} \ t \in ( (n-1)\tau,n\tau], \ n\ge 1. \EEE
 \end{align}
 In the following,  $\tilde{Y}_{h,\tau}$  will be called a \BBB \emph{time-discrete solution}. We often drop the
 $x$-dependence and write $ \tilde{Y}_{h,\tau}(t)$ for a time-discrete solution at time $t$.

 Our first  result addresses the existence of solutions to the 3D problem. We employ an abstract convergence result concerning metric gradient flows, more precisely for curves of maximal slope and their approximation via the minimizing movement scheme. The relevant \SSS notions  about curves of maximal slope \EEE are recalled in Section \ref{sec: auxi-proofs}. \QQQ In particular, the local slopes of $\phi_h$ with respect to $\mathcal{D}_h$ are denoted by \QQQ $|\partial {\phi}_h|_{{\mathcal{D}}_h}$, see Definition~\ref{main def2}. \EEE
  For notational convenience, we define the sublevel set $\mathscr{S}_{h,M}^{3D}:=\{y\in\mathscr{S}_h^{3D} : \phi_h(y)\leq M\}$ \QQQ for $M>0$\EEE, representing deformations with \SSS uniformly \EEE bounded energy.  
%
 \begin{proposition}[Solutions in the 3D setting]\label{maintheorem1}
	Let $M>0$, $y_0^h \in \mathscr{S}^{3D}_{h,M}$ \QQQ and $(\tau_k)_k$ be a null sequence.
		\begin{itemize}
		\item[(i)] {\rm (Time-discrete solutions)} Then, there exists a solution of the minimization problems in \eqref{incremental}.  
		\item[(ii)]  {\rm (Continuous solutions)} Let $(\tilde{Y}_{h,\tau_k})_k$ be a sequence of time-discrete solutions as given in \eqref{ds}. Then, for \EEE $h>0$ sufficiently small only depending on $M$, \QQQ there exists a subsequence of $(\tau_k)_k$ (not relabeled) and a function $y^h \EEE \in L^\infty([0,+\infty);    \mathscr{S}_{h,M}^{3D} \MMM) \cap W^{1,2}([0,+\infty);H^1(\Omega;\R^3))$ \QQQ satisfying $y^h(0) = y_0^h$ \EEE such that \QQQ  $\tilde{Y}_{h,\tau_k}(t) \rightharpoonup y^h(t)$  \EEE weakly in $W^{2,p}(\Omega;\R^3)$ \QQQ for every $t\geq 0$ as $k\to \infty$, and $y^h$  is a  \EEE curve of maximal \AAA slope  \EEE for ${\phi}_h$  with respect to  $|\partial {\phi}_h|_{{\mathcal{D}}_h}$. 
	\end{itemize}
\end{proposition}

 In particular, in \cite[Theorem 2.1]{MFMK} it has been shown that curves of maximal \AAA slope \LLL are \MMM weak solutions of the system \eqref{nonlinear equation}. Let us stress that the existence of weak solutions has already been settled successfully  \SSS in \cite{Kroemer, MielkeRoubicek}. \EEE  Still, our approach sheds light on this issue from a slightly different perspective as our result is formulated in a metric setting, \AAA see below \eqref{eq: disc2} for the motivation of our scheme. The result is proved in Section~\ref{sec: mianmain}.  \EEE

 \subsection{Compactness and limiting variables\EEE}

\AAA We \QQQ suppose \EEE that the limit
 \begin{align}\label{assumption:energyscalingthickness1}
 r:= \lim\limits_{h \to 0} \frac{\eps_h}{\delta_h^2} \AAA \in [0,\infty) \EEE
 \end{align}
 exists. This assumption corresponds to the supercritical regime studied in  \cite{Freddi2013} \MMM in a purely elastic framework. \AAA  The case $r=\infty$ also studied in \cite{Freddi2013} is  not covered by our theory and indeed more delicate \EEE due to a nonlinear constraint in the limiting model.  \AAA To perform the  passage to the  limit,  \EEE we need to specify the \emph{penalization parameter} $\zeta_h$, the \FFF force \EEE $f^{3D}_h$ in \eqref{assumption:energyscaling}, and the topology of the convergence.
  Fixing $\alpha<\MMM 1\EEE$, we suppose that
 \begin{align}\label{assumption:penalizationscale}
\liminf_{h \to 0}\zeta_h \eps_h^{-2} (\eps_h/\delta_h)^{\alpha p} >0 , \ \ \ \   \limsup_{h \to 0} \zeta_h \eps_h^{-2} (\eps_h/\delta_h)^{p}  = 0   .
 \end{align}
 Whereas the liminf condition will help to derive suitable compactness properties, the limsup condition ensures compatibility with the recovery sequence,  see \eqref{eq:secondordervanishes}.  Further, we assume \MMM that \EEE
 \begin{align}\label{eq: forces}
 \FFF  \frac{1}{\eps_h\delta_h}{f}^{3D}_h \rightharpoonup f^{1D} \ \ \text{weakly in } L^2(I) \EEE
 \end{align}
 for \FFF a function $f^{1D} \in L^2(I)$.  Moreover, \EEE we define the sequence of displacements $ u^h \colon \Omega \to \R^3$ by
 \begin{align}
 	u_1^h := \frac{ y_1^h -x_1}{\eps_h}, \quad
 	u_2^h := \frac{ y_2^h - hx_2}{\eps_h / h}, \quad \text{and} \quad 
 	u_3^h := \frac{ y_3^h - \delta_h x_3}{\eps_h/\delta_h}\label{def:udisplacement}
 \end{align}
 and define the function $\theta^h\colon \MMM I \EEE \to \R$ by
 \begin{align}\label{def:theta}
 	\theta^h := \frac{1}{I_0}\frac{1}{\eps_h} \int_\omega \Big( \frac{\delta_h}{h} x_2 y_3^h - x_3  y_2^h \Big) {\rm d}x_2 {\rm d}x_3,
 \end{align}
 where $I_0:= \int_\omega (x_2^2 + x_3^2) {\rm d}x_2 {\rm d}x_3 = \tfrac{1}{6}$.  Whereas $ u_1^h$, $u_2^h$, and $u_3^h$ correspond to \MMM (scaled) deviations of the deformation from the identity, \EEE the function $ \theta^h$ can be interpreted as \MMM a \EEE twist. The following proposition \AAA identifies limits of $(u^h)_h$ and $(\theta^h)_h$ \EEE by a compactness argument, also defining the topology of the convergence. The limiting variables corresponding to displacements lie in the space of \emph{Bernoulli-Navier displacements}  \EEE
 \begin{align}\label{def:Bernoulli-Navier}
 	\mathcal{A}^{BN}_{\hat \xi_1, \hat \xi_2, \hat \xi_3}  = \big\{ u \in W^{1,2}(\Omega;\R^3) : & \text{ there exist } \xi_1 \in W^{1,2}_{\hat \xi_1}(I),\, \xi_2 \in W^{2,2}_{\hat \xi_2}(I) \, \text{and } \xi_3 \in W^{2,2}_{\hat \xi_3}(I) \nonumber\\
 	& \text{ such that } u_1 = \xi_1 - x_2 \xi_2' - x_3 \xi_3', \, u_2 = \xi_2, \, u_3 = \xi_3	\big\},
 \end{align} 
where $\hat \xi_i$, $i =1,2,3$, were defined in \eqref{assumption:clampedboundary}.   On the other hand, the limit of the twists $\theta^h$ is a $W^{1,2}_0(I)$-function, and we \AAA thus \EEE introduce the space \EEE
 \begin{align}\label{eq: s1D}
 	\mathscr{S}^{1D} := \mathcal{A}^{BN}_{\hat \xi_1, \hat \xi_2, \hat \xi_3}   \times W^{1,2}_{0}(I).
 \end{align}\LLL
 \QQQ Finally, we are in the position to state the compactness result. \EEE
 \begin{proposition}[Compactness]\label{lemma:compactness}
 	Consider a sequence $(y^h)_h$ with $y^h \in \mathscr{S}^{3D}_{h,M}$ for all $h$. Then, there exists $(u,\theta) \in \mathscr{S}^{1D}$ such that \QQQ up to subsequences (not relabeled) \EEE
 	\begin{itemize}
		\item[(i)] $  u^h \rightharpoonup u$ in $W^{1,2}(\Omega; \R^3)$ and $u^h_3 \to u_3$ \EEE in $W^{1,2}(\Omega)$.
 		\item[(ii)] $ \theta^h \rightharpoonup \theta$ in $W^{1,2}(I)$.
 	\end{itemize}
 \end{proposition}
\FFF The proposition will be proved in Section~\ref{sec: gammastatic}. \EEE Later, we will see that the compactness also holds in an \AAA evolutionary \EEE    setting,  \SSS see Theorem~\ref{maintheorem3}(iii). \EEE


 \smallskip    
 
 \subsection{Quadratic forms and compatibility conditions}\label{sec:quadraticforms}

 As a preparation for the formulation of the one-dimensional model, we introduce effective quadratic forms related to $W$ and $D$. We define $Q^3_W\colon \R^{3\times3} \to \R$ and $Q^3_R\colon \R^{3\times3} \to \R$ by
 \begin{align}\label{eq:quadraticformsnotred}
 	Q^3_W(A):= \partial^2_{F^2} W(\Id)[A,A] \qquad\text{and}\qquad Q^3_D(A):=  \tfrac{1}{2} \partial^2_{F_1^2} D^2(\Id,\Id)[A,A] = 2 R(\Id,A)
 \end{align}
 for $A \in \R^{3 \times 3}$. \AAA By Taylor \EEE expansion and Polar decomposition in combination with frame indifference (see \eqref{assumptions-W} and \eqref{eq: assumptions-D}) one can observe that the quadratic forms \MMM $Q^3_W$, \LLL $Q^3_D$  only depend \EEE on the symmetric part ${\rm sym}(A)$ of  $A \in \R^{3\times 3}$ and \QQQ that \EEE they are positive definite on $\R^{3\times3}_\sym := \{A \in \R^{3\times 3}: A = A^{\top}\}$. \AAA We \EEE  define reduced quadratic forms by minimizing over stretches in the \QQQ $x_2$- and $x_3$-direction. \EEE More precisely, we let
 \begin{align}\label{def:quadraticforms}
 	Q^1_S(q_{11},q_{12}) := &\min \left\{Q^3_S(A): A\in \R^{3\times 3}_\sym, \, a_{1j}= q_{1j} \, \text{ for } j =1,2 \right\}  &&\text{and} \nonumber\\
 	Q^0_S(q_{11}) :=  &\min \left\{Q^1_S(q_{11},z):z\in \R \right\} \qquad \qquad\qquad \qquad &&\text{for }  S \EEE = W,D. 
 \end{align}\EEE
 \AAA The \EEE quadratic forms \MMM $Q_S^i$, $i=1,3$, \EEE induce \QQQ fourth-order \MMM and \FFF second-order \EEE tensors, respectively, \EEE  denoted by $\C_W^i$ and $\C_D^i$ for $i=1,3$.
 In a similar spirit to  \cite{MFMKDimension, MFLMDimension2D1D}, \EEE  we require some compatibility conditions of the quadratic forms to perform a rigorous evolutionary dimension reduction. \MMM This is crucial \EEE as we need to construct mutual recovery sequences, compatible \FFF with \EEE the elastic energy and the viscous dissipation at the same time, see Theorem \ref{theorem: lsc-slope} \MMM below. \FFF We assume that we can decompose $Q_S^3$ in the following way: \AAA there exist \EEE quadratic forms $Q_S^*$ and constants $C_S^*>0$ such that for all $A= (a_{ij})_{i,j=1,2,3} \in \R^{3\times 3}_\sym$ it holds that
 \begin{align}
 Q_S^3(A) 	= Q_S^1(a_{11},  a_{12} ) + Q_S^*(\tilde A),\qquad Q_S^1(a_{11},a_{12}) = Q_S^0(a_{11}) + C_S^* a_{12}^2, \qquad \text{for } S=W,D\EEE \tag{\textbf{H}} \label{quadraticforms}
 \end{align}
 where $\tilde A = (\tilde a_{ij})_{i,j=1,2,3}\in \R^{3\times 3}_\sym$ satisfies $\tilde a_{1j}=0$ for $j =1,2$, and $ \tilde a_{km}= a_{km} $ for $(k,m) \notin\{ (1,1),(1,2), \AAA (2,1) \EEE \}$.
\EEE This induces a restriction from a modeling point of view since it essentially corresponds to materials with Poisson ratio zero, such as cork. \FFF More precisely, the assumption also covers materials with a nonzero Poisson ratio \QQQ with respect to the \EEE $x_2$- and $x_3$-direction, e.g.\ certain orthotropic materials with density ($\lambda_2,\lambda_3,\mu \ge 0$)
\begin{align}\label{eq: new-orth} 
Q^3_W(A) =  (\lambda_2 a_{22} + \lambda_3 a_{33})^2  +  \mu|A|^2.
\end{align}
 \EEE A \MMM possible generalization (not included in the following for simplicity) \EEE is to assume that $Q_S^3$ are $h$-dependent, denoted by $Q_{S,h}^3$, such that $Q_{S,h}^3 = Q_S^3 + \MMM {\rm o}(1)\EEE \hat{Q}_{S}$ \EEE for $h \to 0$, \EEE where $Q_S^3$ satisfies \eqref{quadraticforms} and $\hat{Q}_S$ is any positive definite quadratic form. Another sound option for a compatibility condition consists in vanishing dissipation effects, briefly discussed in Remark \ref{rem: VDE} below.

 \subsection{The one-dimensional model}
To formulate the gradient flow  for the one-dimensional theory, we again need to assign a metric and an energy to a suitable space. As the limit variables  have already been identified with $\mathscr{S}^{1D}$ in Proposition~\ref{lemma:compactness}, \SSS it remains \EEE to find a metric $\mathcal{D}_0$ and an energy $\phi_0$. Following the abstract theory in \cite{S2}, the natural candidate for the energy is the $\Gamma$-limit of $\phi_h$ from the static theory\QQQ, see Theorem~\ref{th: Gamma}, \EEE and thus we define
 \begin{align}\label{def:enegeryphi-1D}
 	\phi_0(u,\theta):= \frac{1}{2} \int_I Q_W^0 \EEE \Big(\xi_1' + \frac{r}{2} \vert \xi_3^{\prime}\vert^2\Big) \,  {\rm d}x_1 + \frac{1}{24} \int_I \big(Q_W^0 \EEE (\xi_2'') + Q_W^1(\xi_3^{\prime\prime}, \theta^\prime) \big) \, {\rm d}x_1 - \FFF \int_I f^{1D} \xi_3 \,{\rm d}x_1  
 \end{align}
\FFF for $(u,\theta)\in \mathscr{S}^{1D}$, \EEE where \QQQ $u$ is identified with $(\xi_1,\xi_2,\xi_3)$ via \eqref{def:Bernoulli-Navier} and \EEE $f^{1D}$ \FFF is \EEE defined in \eqref{eq: forces}.   \QQQ Due \EEE to the similar structure of the  metric $\mathcal{D}_h$ and the energy $\phi_h$, see \eqref{assumption:energyscaling} and \eqref{assumption:dissipationsclae}, it turns out that \QQQ the \EEE appropriate choice for the metric is \EEE
 \begin{align}
 	\mathcal{D}_0((u,\theta),(\tilde u, \tilde \theta)) :=&\bigg(\int_I Q_R^0\Big(\xi_1' - \tilde \xi_1' + \frac{r}{2}(\vert \xi_3^\prime\vert^2 -\vert \tilde \xi_3^{\prime}\vert^2)\Big) {\rm d}x_1\nonumber\\
 	&+ \frac{1}{12}\int_I \big( Q_R^0(\xi_2''- \tilde \xi_2'') + Q_R^1(\xi_3^{\prime\prime}- \tilde \xi_3^{\prime\prime}, \theta^\prime - \tilde \theta^{\prime}) \big) {\rm d}x_1\bigg)^{1/2}\label{def:metric}
 \end{align}
 for $(u,\theta),(\tilde u, \tilde \theta)\in \mathscr{S}^{1D} $, \SSS where \QQQ $(\xi_1,\xi_2,\xi_3)$ and $(\tilde \xi_1,\tilde \xi_2,\tilde \xi_3)$ correspond to $u$ and $\tilde u$, respectively, see \eqref{def:Bernoulli-Navier}. \EEE
 
 The \EEE geometrical interpretation of the variables $(u,\theta) \in \mathscr{S}^{1D}$ \MMM is as follows: \EEE $u_2$ and $u_3$ correspond to orthogonal displacements in the \QQQ $x_2$- and $x_3$-direction. \EEE The axial displacement $u_1$ in the $x_1$-direction is additionally \MMM influenced \EEE by linear \FFF perturbations \EEE of $u_2'$ and $u_3'$, see \eqref{def:Bernoulli-Navier}. \MMM Eventually, \EEE $\theta$ corresponds to a twist. 
 Note that in \cite{MFLMDimension2D1D} the variable  $\xi_3$ is denoted by $w$. \MMM In \cite[Theorem 2.2]{MFLMDimension2D1D} it has been shown that curves of maximal slope for $\phi_0$ with respect to $\mathcal{D}_0$ give rise to weak solutions in the sense of \cite[(2.13)]{MFLMDimension2D1D} of the system \AAA given in \eqref{eq: equation-simp-intro11}. \EEE    Whereas the cases $r>0$ \SSS considered in \cite{MFLMDimension2D1D} \EEE lead to the one-dimensional von K\'arm\'an theory, $r=0$ corresponds to a linear system of equations. Note that in that case, $\phi_0$ and $\mathcal{D}_0$ become \AAA purely \EEE quadratic, and thus  convex.   Then, existence (and even uniqueness) of weak solutions follow from well-known theory, see e.g. \cite[Section 2.4]{AGS}.

 \subsection{Main  \MMM convergence result\EEE}\label{sec:mainconvergenceresult}

 %
 
 
 To relate the three-dimensional with the one-dimensional model, we will use the abstract theory of gradient flows \cite{AGS} and evolutionary $\Gamma$-convergence \cite{Mielke, S1, S2}.  \QQQ In particular, we recall that the \EEE local slopes are denoted by $|\partial {\phi}_h|_{{\mathcal{D}}_h}$ and $|\partial {\phi}_0|_{{\mathcal{D}}_0}$, respectively, see Definition \ref{main def2},  \SSS and we again refer \QQQ to \EEE Section \ref{sec: auxi-proofs} for more details. \EEE    We introduce \QQQ the \EEE topology of convergence as follows: given \EEE  a deformation $y^h\in \mathscr{S}^{3D}_{h}$, we define mappings $\pi_h\colon \mathscr{S}^{3D}_{h} \to \mathscr{S}$ by $\pi_h(y^h) = (u^h,\theta^h)$ for $\mathscr{S} = W^{1,2}(\Omega;\R^3) \times W^{1,2}(I)$, where $u^h$ and $\theta^h$ are defined in \eqref{def:udisplacement} and \eqref{def:theta}. \SSS We write   $y^h \stackrel{\pi\sigma}{\to} (u,\theta)$ for the convergence found in   Proposition~\ref{lemma:compactness}. \EEE (The symbol $\pi\sigma$ is  used because of  the abstract convergence result, see Section \ref{sec: auxi-proofs}.)  \MMM We also write  \EEE   $y^h \stackrel{\pi\rho}{\to} (u,\theta)$ if the convergence   of $u_1^h$ and $u_2^h$   holds with respect to the strong in place of the weak topology.  We remark that the limiting variables $(u,\theta)$ are contained in the space $\mathscr{S}^{1D} \subset \mathscr{S}$ defined in \eqref{eq: s1D}. \EEE
 Given $(u_0,\theta_0) \in \mathscr{S}^{1D}$, we introduce the family \MMM of sequences of  admissible initial data  \EEE as $
 \mathcal{B}(u_0,\theta_0) = \big\{ (y^h_0)_h: \ y^h_0 \in \mathscr{S}^{3D}_h, \ y^h_0 \stackrel{\pi\sigma}{\to} (u_0,\theta_0), \ \phi_h(y^h_0) \to \MMM {\phi}_0(u_0,\theta_0) \EEE \big\}$.

 %
 
 \begin{theorem}[Relation between three-dimensional and one-dimensional system]\label{maintheorem3}
 	Let $(u_0,\theta_0) \in \mathscr{S}^{1D}$ be an initial datum   and suppose that \eqref{quadraticforms} holds.\EEE \\
 	\begin{itemize}
 		\item[(i)] {\rm (Well-posedness of initial datum)} \SSS There holds $	\mathcal{B}(u_0,\theta_0) \neq \emptyset$ \MMM for all $(u_0,\theta_0)\in \mathscr{S}^{1D}$.  \EEE
 		\item[(ii)] {\rm (Convergence of discrete solutions)}  Consider a sequence $(y^h_0)_h \in \mathcal{B}(u_0,\theta_0)$,  a null sequence $(\tau_h)_h$, and a sequence of time-discrete solutions $\tilde{Y}_{h,\tau_h}$ as in \eqref{ds}  with  $\tilde{Y}_{h,\tau_h}(0)=y^h_0$. \\ 
 		\noindent Then, there exists a curve of maximal slope $(u,\theta)\colon [0,\infty) \to \mathscr{S}^{1D}$ for ${\phi}_0$  with respect to  $|\partial {\phi}_0|_{{\mathcal{D}}_0}$ satisfying $(u(0),\theta(0)) = (u_0,\theta_0)$ such that up to a subsequence (not relabeled) \QQQ it holds that \EEE 
 		\begin{align}\label{eq:ABC1}
 		\tilde{Y}_{h,\tau_h}(t) \stackrel{\pi\rho}{\to} (u(t),\theta(t)) \ \ \ \ \ \text{for all} \ t \in [0,\infty) \ \  \text{ as $h \to 0$}.
 		\end{align}
 		\item[(iii)] {\rm{(Convergence of continuous solutions)}} 
 		\SSS Consider a sequence $(y^h_0)_h \in \mathcal{B}(u_0,\theta_0)$. \EEE
 		Let $(y^h)_h$ be a sequence of curves of maximal \AAA slope \EEE for $\phi_{h}$ with respect to $\vert \partial \phi_{h} \vert_{\mathcal{D}_{h}}$ satisfying $y^h(0) = y^h_0$. Then, there exists a curve of maximal slope $(u,\theta)\colon [0,\infty) \to \mathscr{S}^{1D}$ for ${\phi}_0$  with respect to  $|\partial {\phi}_0|_{{\mathcal{D}}_0}$ satisfying $(u(0),\theta(0)) = (u_0,\theta_0)$ such that up to a subsequence (not relabeled) \QQQ it holds that \EEE 
 		\begin{align}\label{eq:ABC2}
 			y^h(t) \stackrel{\pi\rho}{\to} (u(t),\theta(t)) \ \ \ \ \ \text{for all} \ t \in [0,\infty) \ \  \text{ as $h \to 0$}.
 			\end{align}
 	\end{itemize}
 \end{theorem}
 \AAA The result is proved in Section \ref{sec: mianmain}. \EEE Note that Proposition~\ref{maintheorem1} provides the existence of discrete solutions, but also the existence of curves of maximal \AAA slope \EEE  in the three-dimensional setting (which are weak solutions of \eqref{nonlinear equation} due to \cite[Theorem 2.1]{MFMK}). \EEE \QQQ Further, we \EEE mention once again that in \cite{MFLMDimension2D1D} it has been shown that \LLL the curves from (ii) and (iii) \EEE can be identified as weak solutions to the one-dimensional \AAA system \eqref{eq: equation-simp-intro11}. \EEE Moreover, we stress that \SSS compared to  Proposition~\ref{lemma:compactness} \EEE  the convergence of the displacements $u_1^h$ and $u_2^h$ also holds in the strong $W^{1,2}$-sense.
 \QQQ In addition, \EEE we point out that the assumpion~\eqref{quadraticforms} is only needed in the limiting passage and not for a fixed $h$ in Proposition~\ref{maintheorem1}. \EEE 
 \EEE From now on we set $\QQQ f_h^{3D} \EEE \equiv 0 $ for convenience. The general case indeed follows \QQQ by \EEE minor modifications, which are standard. \FFF For details we refer the reader to Lemma~\ref{lem:compactnessfornonzeroforces}. \EEE

\section{Metric gradient flows}\label{sec: auxi-proofs}

\subsection{Definitions}

In this section, we  recall the relevant definitions about curves of maximal slope and present  abstract theorems concerning the convergence of time-discrete solutions and continuous solutions to curves of maximal slope.   
We consider a   complete metric space $(\mathscr{S},\mathcal{D})$. We say a curve $y\colon (a,b) \to \mathscr{S}$ is \emph{absolutely continuous} with respect to $\mathcal{D}$ if there exists $m \in L^1(a,b)$ such that
\begin{align}\label{def:absolutecontinuity}
\mathcal{D}(y(s),y(t)) \le \int_s^t m(r) \,{\rm d}r \EEE \ \ \   \text{for all} \ a \le s \le t \le b.
\end{align}
The smallest function $m$ with this property, denoted by $|y'|_{\mathcal{D}}$, is called \emph{metric derivative} of  $y$  and satisfies  for a.e.\ $t \in (a,b)$   (see \cite[Theorem 1.1.2]{AGS} for the existence proof)
\begin{align}\label{def:metricderivative}
|y'|_{\mathcal{D}}(t) := \lim_{s \to t} \frac{\mathcal{D}(y(s),y(t))}{|s-t|}.
\end{align}
We  define the notion of a \emph{curve of maximal slope}. We only give the basic definition here and refer to \cite[Section 1.2, 1.3]{AGS} for motivations and more details.  By  $h^+:=\max(h,0)$ we denote the positive part of a function  $h$.

\begin{definition}[Upper gradients, slopes, curves of maximal slope]\label{main def2} 
	We consider a   complete metric space $(\mathscr{S},\mathcal{D})$ with a functional $\phi\colon \mathscr{S} \to (-\infty,+\infty]$.

	{\rm(i)} A function $g\colon \mathscr{S} \to [0,\infty]$ is called a strong upper gradient for $\phi$ if for every absolutely continuous curve $ y\colon  (a,b) \to \mathscr{S}$ the function $g \circ y$ is Borel and 
	$$|\phi(y(t)) - \phi(y(s))| \le \int_s^t g( y(r)) |y'|_{\mathcal{D}}(r)\,\EEE {\rm d}r \EEE \  \ \  \text{for all} \ a< s \le t < b.$$
	
	{\rm(ii)} For each $y \in \mathscr{S}$ the local slope of $\phi$ at $y$ is defined by 
	$$|\partial \phi|_{\mathcal{D}}(y): = \limsup_{z \to y} \frac{(\phi(y) - \phi(z))^+}{\mathcal{D}(y,z)}.$$

	{\rm(iii)} An absolutely continuous curve $y\colon (a,b) \to \mathscr{S}$ is called a curve of maximal slope for $\phi$ with respect to the strong upper gradient $g$ if for a.e.\ $t \in (a,b)$
	$$\frac{\rm d}{ {\rm d} t} \phi(y(t)) \le - \frac{1}{2}|y'|^2_{\mathcal{D}}(t) - \frac{1}{2}g^2(y(t)).$$
\end{definition}

\subsection{Curves of maximal slope as limits of time-discrete solutions}

In the following, we consider a sequence of complete metric spaces $(\mathscr{S}_k, \mathcal{D}_k)_k$, as well as a limiting complete metric space $(\mathscr{S}_0,\mathcal{D}_0)$. Moreover, let $(\phi_k)_k$ be a sequence of functionals with $\phi_k\colon \mathscr{S}_k \to \FFF(-\infty,\infty] \EEE$ and $\phi_0\colon \mathscr{S}_0 \to \FFF (-\infty,\infty] \EEE$. 

\FFF Fixing $k\in \N$, we now \AAA describe \EEE the construction of time-discrete solutions for the energy $\phi_k$ and the metric $\mathcal{D}_k$, already mentioned in Section~\ref{section:2}, but in a \AAA general \EEE metric setting to give a precise description of the abstract theory.
Consider a fixed time step $\tau >0$ and suppose that an initial datum $Y^0_{k,\tau}$ is given. Whenever $Y_{k,\tau}^0, \ldots, Y^{n-1}_{k,\tau}$ are known, $Y^n_{k,\tau}$ is defined as (if existent)
\begin{align*}
Y_{k,\tau}^n = {\rm argmin}_{v \in \mathscr{S}_k} \mathbf{\Phi_k}(\tau,Y^{n-1}_{k,\tau}; v), \ \ \ \mathbf{\Phi_k}(\tau,u; v):=  \frac{1}{2\tau} \mathcal{D}_k(v,u)^2 + \phi_k(v). 
\end{align*}
 Then, \EEE we define the  piecewise constant interpolation by
\begin{align*}
\tilde{Y}_{k,\tau}(0) = Y^0_{k,\tau}, \ \ \ \tilde{Y}_{k,\tau}(t) = Y^n_{k,\tau}  \ \text{for} \ t \in ( (n-1)\tau,n\tau], \ n\ge 1.  
\end{align*}
We call  $\tilde{Y}_{k,\tau}$  a \emph{time-discrete solution}. 

\MMM Our \EEE goal is to study the limit of time-discrete solutions as $k \to \infty$. To this end, we need to introduce a suitable topology for the convergence.  We assume that for each $k \in \N$ there exists a map $\pi_k\colon \mathscr{S}_k \to \mathscr{S}$ for a suitable set $\mathscr{S} \supseteq \mathscr{S}_0$. Given a sequence $(z_k)_k$, $z_k \in \mathscr{S}_k$, and $z \in \mathscr{S}$,  we say
\begin{align*}
z_k \stackrel{\pi\sigma}{\to} z \  \ \ \ \  \text{if} \ \ \ \pi_k(z_k) \stackrel{\sigma}{\to} z,
\end{align*}
\FFF where $\sigma$ is a topology on $\mathscr{S}$ satisfying the following conditions: \EEE
we suppose \MMM  that \EEE   
\begin{align}\label{compatibility}
\begin{split}
z_k \stackrel{\pi\sigma}{\to} z, &\ \  \bar{z}_k \stackrel{\pi\sigma}{\to} \bar{z}  \ \ \  \Rightarrow \ \ \ \liminf_{k \to \infty} \mathcal{D}_k(z_k,\bar{z}_k) \ge  \mathcal{D}_0(z,\bar{z})
\end{split}
\end{align}
for all $z, \bar z \in \mathscr{S}_0$.
Moreover, we assume that for every sequence $(z_k)_k$, $z_k \in \mathscr{S}_k$, and $N \in \N$ we have
\begin{align}\label{basic assumptions2}
\begin{split}
\phi_k(z_k) \leq N \quad \Rightarrow \ \ \ z_k \stackrel{\pi\sigma}{\to} z \in \mathscr{S}_0 \quad \text{(up to a subsequence)}.
\end{split}
\end{align}
\FFF Further, \EEE we suppose lower semicontinuity of the energies and the slopes in the following sense: for all $z \in \mathscr{S}_0$ and $(z_k)_k$, $z_k \in \mathscr{S}_k$, we have
\begin{align}\label{eq: implication}
\begin{split}
z_k \stackrel{\pi\sigma}{\to}  z \ \ \ \ \  \Rightarrow \ \ \ \ \  \liminf_{k \to \infty} |\partial \phi_{k}|_{\mathcal{D}_{k}} (z_{k}) \ge |\partial \phi_0|_{\mathcal{D}_0} (z), \ \ \ \ \  \liminf_{k \to \infty} \phi_{k}(z_{k}) \ge \phi_0(z).
\end{split}
\end{align}
\MMM For the relation  of time-discrete solutions  and \EEE curves of maximal slope we \FFF will \EEE use the following result\FFF.\EEE
\begin{theorem}\label{th:abstract convergence 2}
	Suppose that   \eqref{compatibility}--\eqref{eq: implication} hold. Moreover, assume that    $|\partial \phi_0|_{\mathcal{D}_0}$ is a  strong upper gradient for $ \phi_0 $.  Consider a  null sequence $(\tau_k)_k$. Let   $(Y^0_{k,\tau_k})_k$ with $Y^0_{k,\tau_k} \in \mathscr{S}_k$  and $z_0 \in \mathscr{S}_0$ be initial data satisfying 
	
	\begin{align*}
	& \ \  Y^0_{k,\tau_k} \stackrel{\pi\sigma}{\to} z_0 , \ \ \ \ \  \phi_k(Y^0_{k,\tau_k}) \to \phi_0(z_0).
	\end{align*}
	Then, for each sequence of discrete solutions $(\tilde{Y}_{k,\tau_k})_k$  starting from $(Y^0_{k,\tau_k})_k$ there exists a limiting function $z\colon [0,+\infty) \to \mathscr{S}_0$ such that up to a  subsequence \QQQ (not relabeled) \EEE
	$$\tilde{Y}_{k,\tau_k}(t) \stackrel{\pi\sigma}{\to} z(t), \ \ \ \ \ \phi_k(\tilde{Y}_{k, \tau_k}(t)) \to \phi_0(z(t)) \ \ \  \ \ \ \ \ \forall t \ge 0$$
	as $k \to \infty$, and $z$ is a curve of maximal slope for $\phi_0$ with respect to $|\partial \phi_0|_{\mathcal{D}_0}$. In particular,  $z$ satisfies the \AAA energy-dissipation-balance \EEE   
	\begin{align}\label{maximalslope}
	\frac{1}{2} \int_0^T |z'|_{\mathcal{D}_0}^2(t) \, {\rm d} t + \frac{1}{2} \int_0^T |\partial \phi_0|_{\mathcal{D}_0}^2(z(t)) \, {\rm d} t + \phi_0(z(T)) = \phi_0(z_0) \ \  \ \ \ \forall \,  T>0. 
	\end{align} 
	
\end{theorem} 
\FFF
The statement is proved in \cite[Section 2]{Ortner} for a sequence of functionals and metrics defined on a {\emph{single}} space. The generalization for a sequence of spaces is straightforward and follows from standard adaptions. For a more
detailed discussion of similar statements, we refer for example to   \cite[Section~3.3]{MFMK}.  


\subsection{Curves of maximal slope as limits of continuous solutions}\label{sec: auxi-proofs2}


As before, $(\mathscr{S}_k, \mathcal{D}_k)_k$ and  $(\mathscr{S}_0,\mathcal{D}_0)$ denote complete metric spaces, with \ZZZ corresponding functionals \EEE $(\phi_k)_k$ and  $\phi_0$.  \EEE For the relation of the three- and one-dimensional systems, we will use the following result.

\begin{theorem}\label{thm: sandierserfaty}
	
	
	\EEE Suppose that   \eqref{compatibility}--\eqref{eq: implication} hold. \EEE Moreover, assume that $\vert \partial \phi_n \vert_{\mathcal{D}_n}$, $\vert \partial \phi_0 \vert_{\mathcal{D}_0}$ are strong upper gradients for $\phi_n$, $\phi_0$ with respect to $\mathcal{D}_n$, $\mathcal{D}_0$, respectively.
	Let \FFF $ z_0 \EEE \in \mathscr{S}_0$. For all $n \in \N$, let \FFF $z_n$ \EEE be a curve of maximal slope for $\phi_n$ with respect to $\vert \partial \phi_n \vert_{\mathcal{D}_n}$ such that
	\begin{align*}
	{\rm(i)} \quad & \sup\limits_{n \in \N} \sup \limits_{t \geq 0}  \, \phi_n(\FFF z_n(t) \EEE)  < \infty \\
	{\rm(ii)} \quad & \FFF z_n(0) \EEE \stackrel{\pi\sigma}{\to} \FFF z_0, \EEE\quad \phi_n(\FFF z_n(0) \EEE) \to \phi_0(\FFF z_0 \EEE).
	\end{align*}
	Then, there exists a limiting function $\FFF z \EEE \colon [0,\infty) \to \mathscr{S}_0$ such that up to a subsequence \QQQ  (not relabeled) \EEE
	\begin{align*}
	\FFF z_n(t)\EEE \stackrel{\pi\sigma}{\to} \FFF z(t)\EEE , \quad \phi_n(\FFF z_n(t)\EEE) \to \phi_0 (\FFF z(t)\EEE) \quad \forall t \geq 0
	\end{align*}
	as $n \to \infty$, \AAA and \EEE  $z$ \EEE is a curve of maximal slope for $\phi_0$ with respect to $\vert \partial \phi_0 \vert_{\mathcal{D}_0}$ and satisfies \eqref{maximalslope}.\EEE
\end{theorem}

The result is a variant of \cite{S2} and \EEE is given in \cite[Theorem 3.6]{MFMK}. \FFF Once again, we refer to \cite[Section~3.3]{MFMK} for a brief discussion.  
\section{Geometric rigidity}\label{sec:rigidity}

 As a preliminary step, we  provide the necessary rigidity estimates to adapt the $\Gamma$-convergence result  of \cite{Freddi2013} to our setting with clamped boundary conditions. Recall the definitions of \MMM Subsection~\ref{sec:rescaling}, \EEE especially the definition of \MMM $S = I \times (-1/2,1/2)$ and $\mathscr{S}^{3D}_{h,M}$. \EEE  \MMM Due \EEE to the frame indifference of $W$ and $P$ and \SSS their \EEE coercivity properties, see \eqref{assumptions-W} and \eqref{assumptions-P}, \MMM one \EEE  can only expect estimates up to a rigid motion of $(y^h)_h$, \LLL see \MMM \cite{Freddi2013}. \EEE  The crucial point of the following lemma is \SSS that \EEE the clamped boundary conditions in \eqref{assumption:clampedboundary} imply rigidity estimates for $(y^h)_h$ itself. To this end, we combine the arguments of \cite[Lemma 3.1]{Freddi2013} and \cite[Lemma 13]{lecumberry}. \SSS Additionally, \EEE we derive $L^\infty$-estimates for the scaled gradient $\nabla_h y^h$, \MMM as a consequence of the \EEE the \FFF second-order \EEE perturbation in \eqref{assumption:energyscaling}. (The $L^\infty$-bounds
will be crucial for validity of Lemma~\ref{lemma: metric space-properties} and Lemma~\ref{th: metric space} \MMM below  as any other $L^q$-bound would \EEE not be sufficient to prove those lemmata.)

\begin{lemma}[{Rigidity estimates on thin domains}] \label{lem:rigidity}
	Let $M>0$ and \AAA let \EEE $(y^h)_h$ be a sequence in $\mathscr{S}^{3D}_{h,M}$. Moreover, assume that \AAA \eqref{assumption:energyscalingthickness1}--\eqref{assumption:penalizationscale} hold. \EEE   Then, there exists a sequence $(R^h)_h \subset W^{1,2} (S; SO(3))$ such that we have for $h$ small enough 
	\begin{thmlist}
		\item $ \Vert \nabla_h  y^h - R^h \Vert_{L^2(\Omega)} \leq C \eps_h$,\label{lem:rigidity:energycontrol}
		\item $ \Vert {R^h,}_1 \Vert_{L^2(S)} \leq C \eps_h /\delta_h$, $ \Vert { R^h,}_2 \Vert_{L^2(S)} \leq Ch \eps_h /\delta_h$,\label{lem:rigidity:derivatives} 
		\item $ \Vert  R^h - \Id\Vert_{L^2(S)} \leq C \eps_h /\delta_h$,\label{lem:rigidity:dev:RotID}
		\item $\Vert \nabla_h  y^h - \Id\Vert_{L^2(\Omega)} \leq C \eps_h/\delta_h$,\label{lem:rigidity:dev:GradID}
		\item $\Vert \nabla_h  y^h - \Id \Vert_{L^\infty(\Omega)} \leq C \eps_h^\alpha/ \delta_h^{\alpha}$, \label{lem:rigidity:dev:GradIDInfty}
		\item $\Vert R^h - \Id \Vert_{L^\infty(S)} \leq C \eps_h^\alpha/ \delta_h^{\alpha}$, \label{lem:rigidity:dev:RotIDInfty}\hfill\refstepcounter{equation}\normalfont{(\theequation)}\label{eq:rigidity}
	\end{thmlist}
\LLL where $C$ only depends on $\Omega$, $\Gamma$, \QQQ $M$, and \EEE $p$.
\end{lemma}

\begin{proof}
The proof consists of two steps: \AAA First, \EEE we derive the statement up to a rotation of $\nabla_h y^h$. This has been essentially proven in \cite[Lemma 3.3]{Freddi2012} \MMM and   \cite[Lemma 3.1]{Freddi2013}. \EEE  In our present setting, we derive additional $L^\infty$-estimates, \AAA see \eqref{eq:rigidity}\ref{lem:rigidity:dev:GradIDInfty},\ref{lem:rigidity:dev:RotIDInfty}. \MMM  In \EEE the second step, we employ the boundary conditions in \eqref{assumption:clampedboundary} to show that the estimates also hold for the sequence \LLL $(y^h)_h$ \EEE itself.

\textit{Step 1:}   Due to the coercivity of the energy densitiy $W$ in \eqref{assumptions-W}(iii), by \cite[Lemma 3.1]{Freddi2013}  we can find  \EEE a sequence $(R^h)_h \in W^{1,2}(S; SO(3))$ such that
\ref{lem:rigidity:energycontrol} and \ref{lem:rigidity:derivatives} hold. \SSS Note that the sequence is indeed $SO(3)$-valued by \cite[Lemma 3.3]{Freddi2012}  since \eqref{assumption:energyscalingthickness1} implies $h^{1/2}\eps_h/\delta_h \to 0$. \EEE
The existence of $R^h$ deeply relies on the rigidity estimate by {\sc Friesecke, James, and M\"uller} \cite[Theorem 3.1]{FrieseckeJamesMueller:02}. We briefly describe the heuristics \AAA of \EEE the construction, see \cite[Lemma~3.3]{Freddi2012} for the complete proof: working on the \MMM original \EEE domain $\Omega_h$, one divides $\Omega_h$ into cubes with \FFF side length \EEE $\delta_h$. On each cube, one uses the rigidity estimate to control the $L^2$-norm of the distance \AAA between \EEE $\nabla (y^h \circ \LLL p_h^{-1}) \EEE$ and $SO(3)$ by a positive constant times the $L^2$-deviation of gradient \MMM from  a \EEE \textit{fixed} rotation. Here, $p_h$ is the projection defined in \MMM Subsection~\ref{sec:rescaling}. \EEE As the height of $\Omega_h$ coincides with the \FFF side length \EEE of a cube, the rotations can be chosen independently of $x_3$. Then, after defining piecewise-constant interpolations of the rotations, one obtains $R^h$ by a mollification and a projection onto $SO(3)$. For this argument, the scaling \SSS $h^{1/2}\eps_h/\delta_h \to 0$ \EEE and the smoothness of the manifold $SO(3)$ ensure that the mollification lies in \MMM  in a tubular neighborhood of  $SO(3)$  such that  \LLL the nearest-point   projection onto $SO(3)$, denoted by $\Pi_{SO(3)}$, is smooth.
For later purposes, we recall the \SSS precise \EEE representation of $R^h$ from the proof of \cite[Lemma 3.3]{Freddi2012}  which is given by
\begin{align}\label{repofapproxrot}
R^h(x_1,x_2) := \Pi_{SO(3)}\left[\bar R^h \right], \quad \SSS \text{where} \quad \bar R^h := \EEE \int_S \eta_h(z_1,z_2) \tilde R^h(x_1 - z_1, x_2 - z_2) \,{\rm d}z_1 \, {\rm d}z_2.
\end{align}
\AAA Here, \EEE  $\tilde R^h$ is a piecewise-constant function \AAA (constant on each cube described above), \EEE $SO(3)$-valued, and successively extended to $\R^2$ by reflection, \SSS and \EEE  $(\eta_h)_h$ is a family of standard mollifiers. 

Next, we see by \cite[Lemma 3.1]{Freddi2013} that \ref{lem:rigidity:dev:RotID} and \ref{lem:rigidity:dev:GradID} hold for a sequence \AAA of \EEE rotations $(Q^h)_h$, $Q^h \in SO(3)$, instead of $\Id$, \LLL i.e.,
\begin{align}\label{eq:rigidityuptorigidmotion}
	\Vert R^h - Q^h \Vert_{L^2(S)} \leq C \eps_h/\delta_h, \qquad \text{and} \qquad \Vert\nabla_h y^h - Q^h \Vert_{L^2(\Omega)} \leq C \eps_h/\delta_h,
\end{align}
where the second inequality is an immediate consequence of the first one and \ref{lem:rigidity:energycontrol}. \EEE
 We now prove $Q^h$-versions of \ref{lem:rigidity:dev:GradIDInfty} and \ref{lem:rigidity:dev:RotIDInfty}, inspired by \cite[Lemma 4.2]{MFMK}: by the definition of $\phi_h$ \MMM in \eqref{assumption:energyscaling} \EEE and \eqref{assumptions-P}(iii) we get $\Vert \nabla_h^2 y^h \Vert_{L^p(\Omega)} \leq \MMM  C (\eps_h^2/\zeta_h)^{1/p} \LLL \le  \EEE  C \eps_h^\alpha /\delta_h^{\alpha}$, \MMM where we used \eqref{assumption:penalizationscale}. \EEE Moreover, as $p>3$, Poincaré's inequality yields some $F \in \R^{3 \times 3}$ such that $\Vert\nabla_h y^h - F \Vert_{L^\infty(\Omega)} \leq  C\eps_h^\alpha /\delta_h^{\alpha}$ for a constant $C$ depending only on $\Omega$, $p$, \MMM and $M$. \EEE This   together the triangle inequality and \eqref{eq:rigidityuptorigidmotion} yields
	\begin{align*}
	\vert F-Q^h\vert & \SSS \leq C \EEE \left( \Vert \nabla_h y^h - Q^h \Vert_{L^2(\Omega)} + \Vert \nabla_h  y^h - F \Vert_{L^2(\Omega)} \right)
	\leq  C \left(\eps_h/\delta_h + \eps_h^\alpha /\delta_h^{\alpha}\right)
	\leq  C \eps_h^\alpha /\delta_h^{\alpha}.
	\end{align*}
	\AAA Then, \EEE we have  
	\begin{align}\label{QversionforGradIDInfty}
	\Vert \nabla_h y^h - Q^h \Vert_{L^\infty(\Omega)} \leq C \eps_h^\alpha /\delta_h^{\alpha}
	\end{align} by the triangle inequality. \SSS Finally, to derive the $Q^h$-version of (vi), \EEE recall $\tilde R^h$ from \eqref{repofapproxrot}, let \SSS $q$ be a cube \QQQ with sidelength $\delta_h$ \EEE on which $\hat R^h := \tilde R^h\circ p_h^{-1}$ is constant, and define $z^h:= y^h \circ p_h^{-1}$. 
	Since by  \cite[Theorem~3.1]{FrieseckeJamesMueller:02}   we can assume that $\hat R^h = \Pi_{SO(3)}[\dashint_q \nabla z^h]$, \AAA    \eqref{QversionforGradIDInfty} implies \EEE that 
	\begin{align*}
		|  {\hat R^h}|_{\QQQ q \EEE} - Q^h | &\leq C  \delta_h^{-3} \int_q \vert \nabla (y^h \circ p_h^{-1}(x) ) - Q^h\vert \, {\rm d}x = C  h \delta_h^{-2} \int_{p_h^{-1}(q)} \vert \nabla_h y^h (x) - Q^h\vert \, {\rm d}x 
		\leq C \eps_h^\alpha /\delta_h^{\alpha},
	\end{align*}
	where the constant $C$ does not depend on $q$.
 \EEE Then recalling the definition of $\bar R^h$ in \eqref{repofapproxrot} we also get
	$$\Vert \bar R^h - Q^h \Vert_{\QQQ L^\infty(S) \EEE} \le C \eps_h^\alpha /\delta_h^{\alpha}, $$ \EEE 
	and eventually  we conclude 
		\begin{align}\label{QversionforRotationIDInfty}
		\Vert R^h - Q^h \Vert_{L^\infty(S)} \leq  C \eps_h^\alpha/\delta_h^\alpha.
	\end{align}	\EEE

\textit{Step 2:} We now prove that the rotations $Q^h$ are sufficiently close to the identity $\Id$, \MMM namely \EEE
\begin{align}\label{eq:deviationfromidentity}
Q^h = \Id + \begin{pmatrix}
\mathcal{O}(\eps_h)& \mathcal{O}(\eps_h) & \mathcal{O}(\eps_h) \\
\mathcal{O}(\eps_h)& \MMM \mathcal{O}(\eps_h) \EEE & \mathcal{O}(\eps_h/\delta_h) \\
\mathcal{O}(\eps_h)& \mathcal{O}(\eps_h/\delta_h) & \mathcal{O} (\eps_h/\delta_h)
\end{pmatrix}.
\end{align}
\MMM Then, together with the triangle inequality, \eqref{eq:rigidityuptorigidmotion}, \eqref{QversionforGradIDInfty}, and \eqref{QversionforRotationIDInfty} we get that \SSS all properties \EEE of Lemma~\ref{lem:rigidity} are satisfied for $R^h$ and $y^h$. To show \eqref{eq:deviationfromidentity}, we \EEE follow the strategy of \cite[Lemma 13]{lecumberry}. We define the function $\bar u^h$ similar to \eqref{def:udisplacement} by replacing $y^h$ with $\bar y^h = ({Q^h})^{\top} y^h - c^h$ for some constant vectors $c^h \in \R^3$. By \MMM \cite[Lemma 3.7\,\LLL \emph{(ii)}\EEE]{Freddi2013} \EEE we can choose $c^h$ in such a way \MMM that \EEE
\begin{align}\label{barubounded}
	\Vert \bar u^h \Vert_{W^{1,2}(\Omega)} \leq C
\end{align}
\LLL for a constant $C>0$ depending only on $\Omega$ and \QQQ $M$. \EEE \QQQ Moreove, we can calculate \EEE
	\begin{align}
	Q^h \bar u^h &= Q^h \Big(\frac{\bar y_1^h - x_1}{\eps_h}     , \frac{\bar y_2^h - hx_2}{\eps_h /h}    , \frac{\bar y_3^h - \delta_h x_3}{\eps_h /\delta_h} \Big)^{\top} \nonumber \\
	&= Q^h \Big(\frac{({Q^h}^{\top} y^h)_1 - c^h_1 - x_1}{\eps_h}     , \frac{({Q^h}^{\top} y^h)_2 - c^h_2 - hx_2}{\eps_h /h}    , \frac{({Q^h}^{\top} y^h)_3 - c^h_3- \delta_h x_3}{\eps_h /\delta_h} \Big)^{\top}\nonumber \\
	&= ({u_1^h},{u_2^h}, {u_3^h})^{\top} + (\MMM \Id \EEE - Q^h)(\tfrac{x_1}{\eps_h}, \tfrac{hx_2}{\eps_h/h},\tfrac{\delta_hx_3}{\eps_h/\delta_h})^{\top} - Q^h \big(\tfrac{c^h_1 }{\eps_h}     , \tfrac{ c^h_2}{\eps_h /h}    , \tfrac{  c^h_3 }{\eps_h /\delta_h} \big)^{\top}. \label{eq:closetoidentity}
	\end{align}
	Thus, by passing to the trace in \eqref{eq:closetoidentity} we get with \eqref{assumption:clampedboundary} 
	\begin{align*}
	&\Vert  (\Id - Q^h)(\tfrac{x_1}{\eps_h}, \tfrac{hx_2}{\eps_h/h},\tfrac{\delta_hx_3}{\eps_h/\delta_h})^{\top} - Q^h \big(\tfrac{c^h_1 }{\eps_h}     , \tfrac{ c^h_2}{\eps_h /h}    , \tfrac{  c^h_3 }{\eps_h /\delta_h} \big)^{\top} \Vert_{L^2(\Gamma)}\\
	\leq & \Vert Q^h \bar u^h \Vert_{L^2(\Gamma)} + \Vert ({u_1^h},{u_2^h}, {u_3^h})^{\top} \Vert_{L^2(\Gamma)} \\ \le & \LLL \, C \EEE \Vert Q^h \bar u^h \Vert_{W^{1,2}(\Omega)} + \big\Vert ( \hat \xi_1 - x_2 \hat \xi_2' - x_3 \hat \xi_3' ,\hat \xi_2, \hat \xi_3)^{\top} \big\Vert_{L^2(\Gamma)},
	\end{align*}
	which is uniformly bounded \MMM by \QQQ \eqref{barubounded}.  	\EEE
\MMM In view of $\Gamma = \lbrace -l/2,l/2\rbrace \times \omega$, we find \EEE $	\int_{\Gamma} (x_1, hx_2, \delta_h x_3)\, {\rm d}\mathcal{H}^{2}= 0$. \MMM  By this  and the fact that \EEE  $Q^h$ and $c^h$ are constant, we have by a binomial expansion
	\begin{align*}
	\Vert  (\Id - Q^h)(\tfrac{x_1}{\eps_h}, \tfrac{hx_2}{\eps_h/h},\tfrac{\delta_hx_3}{\eps_h/\delta_h})^{\top}\Vert^2_{L^2(\Gamma)} + \Vert Q^h \big(\tfrac{c^h_1 }{\eps_h}     , \tfrac{ c^h_2}{\eps_h /h}    , \tfrac{  c^h_3 }{\eps_h /\delta_h} \big)^{\top} \Vert^2_{L^2(\Gamma)} \leq C.
	\end{align*}
	Thus, by \cite[Lemma 3.3]{MaNePe} (note that ${\rm aff}(S_0) = \R^3$),  we have 
	\begin{align*}
	\left\vert\left(\frac{(\Id-Q^h)_1}{\eps_h},\frac{(\Id-Q^h)_2}{\eps_h/h^2}, \frac{(\Id-Q^h)_3}{\eps_h/\delta_h^2}\right)\right\vert \leq  C\Big(\int_{\Gamma} \vert (\Id - Q^h)(\tfrac{x_1}{\eps_h}, \tfrac{x_2}{\eps_h/h^2},\tfrac{x_3}{\eps_h/\delta_h^2}) \vert ^2 {\rm d}\mathcal{H}^2 \Big)^{1/2} \leq C,
	\end{align*}
\LLL	where the constant $C>0$ additionally depends on $\Gamma$. \EEE
	As $Q^h$ are rotations, the scaling of each column of $\Id-Q^h$ also holds for the corresponding row. More precisely, \MMM as $\delta_h \le h$, \EEE we have
	\begin{align*}
	Q^h = \Id + \begin{pmatrix}
	\mathcal{O}(\eps_h)& \mathcal{O}(\eps_h) & \mathcal{O}(\eps_h) \\
	\mathcal{O}(\eps_h)& \mathcal{O}(\eps_h/h^2) & \mathcal{O}(\eps_h/h^2) \\
	\mathcal{O}(\eps_h)& \mathcal{O}(\eps_h/h^2) & \mathcal{O} (\eps_h/\delta_h^2)
	\end{pmatrix}.
	\end{align*}	
\MMM This control is not sufficient yet to get \eqref{eq:deviationfromidentity}, \SSS and the estimate needs to be refined.  \QQQ For this purpose, \EEE it will be convenient to analyze the first momentum \MMM in the \QQQ $x_3$-direction, \EEE defined by
\begin{align}\label{eq: zeta def}
\bar \Psi^h(x_1,x_2)  = \int_{-\tfrac{1}{2}}^{\tfrac{1}{2}} x_3 \left[\bar y^h(x_1,x_2,x_3) - \begin{pmatrix}
x_1\\hx_2\\\delta_hx_3
\end{pmatrix}\right] {\rm d}x_3.
\end{align} 
\MMM In a similar fashion, we \EEE  define the first momentum $\Psi^h$ for the \MMM non-rotated \EEE deformation, simply by replacing $\bar y^h$ with $y^h$.
 Analogously to \eqref{eq:closetoidentity}, we derive
\begin{align}\label{eq:closetoidentityzeta}
Q^h \bar \Psi^h =  \MMM \Psi^h + \tfrac{1}{12}( \Id - Q^h ) \delta_h e_3\QQQ, 
\end{align}
\QQQ where $e_3$ denotes the normal vector pointing in the $x_3$-direction. \EEE
\EEE Let us define $\bar R^h:= (Q^h)^{\top} R^h$. The relation
	\begin{align*}
		\tfrac{1}{\delta_h} \bar y^h_{,3} - e_3 = (\nabla_h \bar y^h - \bar R^h) e_3 + (\bar R^h - \Id) e_3
	\end{align*}
	implies together with \ref{lem:rigidity:energycontrol} \MMM that \EEE
	\begin{align*}
		\Vert \tfrac{1}{\delta_h} \bar y^h_{,3} - e_3 -(  \bar R^h - \Id) e_3 \Vert_{L^2(\Omega)} \leq \MMM C \EEE \eps_h.
	\end{align*}
	Then, a Poincar\'e inequality with respect to $x_3$ yields \MMM  $\Vert f \Vert_{L^2(\Omega)}\le C  \eps_h $, where \EEE
	\begin{align*}
f(x) := \frac{1}{\delta_h} \left[\bar y^h -  \begin{pmatrix}
		x_1\\hx_2\\\delta_hx_3
		\end{pmatrix}\right] - \frac{1}{\delta_h} \int_{-\tfrac{1}{2}}^{\tfrac{1}{2}} \left[\bar y^h(x_1,x_2,\MMM s \EEE ) - \begin{pmatrix}
		x_1\\hx_2\\\delta_h \MMM s \EEE
		\end{pmatrix}\right] {\rm d} \MMM s \EEE - x_3 (  \bar R^h - \Id) e_3
	\end{align*} 
 \QQQ for $x\in\Omega$. \EEE By \MMM Jensen's  inequality and the definition in \eqref{eq: zeta def} we thus get \EEE  
	\begin{align*}
		\left\Vert \tfrac{1}{\eps_h} \bar \Psi^h - \tfrac{1}{12} \tfrac{\delta_h}{\eps_h}(\bar R^h - \Id) e_3 \right\Vert^2_{L^2(S)} \MMM = \frac{\delta_h^2}{\eps_h^2} \int_S \Big|  \int_{-1/2}^{1/2}f(x)x_3 \, {\rm d}x_3 \Big|^2 \, {\rm d}x_1 {\rm d}x_2 \EEE \leq \MMM C \EEE \delta_h^2.
	\end{align*}
Then, we find by \QQQ \eqref{eq:rigidityuptorigidmotion} \EEE that
\begin{align}\label{eq: zeta-con1}
\Vert \tfrac{1}{\eps_h} \bar \Psi^h \Vert_{L^2(S)} \le \SSS C \AAA \delta_h^2+ \EEE \EEE C \Vert \tfrac{\delta_h}{\eps_h}(\bar R^h - \Id) e_3 \Vert_{L^2(S)} \le C. 
\end{align}
As $\bar R^h - \Id$ is independent of $x_3$, \MMM for $i=1,2$ it holds that \EEE
\begin{align*}
\frac{\bar \Psi^h_{,i}}{\eps_h h^{i-1}} = \int_{-\tfrac{1}{2}}^{\tfrac{1}{2}} x_3 \tfrac{1}{\eps_h}\left(\nabla_h\bar y^h(x_1,x_2,x_3) - \Id \right)e_i \, {\rm d}x_3 = \int_{-\tfrac{1}{2}}^{\tfrac{1}{2}} x_3 \tfrac{1}{\eps_h}\left(\nabla_h\bar y^h(x_1,x_2,x_3) - \bar R^h \right)e_i  \, {\rm d}x_3.
\end{align*}
\MMM Then, due to \ref{lem:rigidity:energycontrol} and \eqref{eq: zeta-con1}, \LLL we have
\begin{align}\label{eq: zeta-con2}
\Vert \bar \Psi^h \Vert_{W^{1,2}(S)} \le C\eps_h.
\end{align}
Hereby, we also control the $L^2$-trace of $ \bar \Psi^h$ on $\Gamma$. \EEE Next, note that also the $L^2$-trace of $\Psi^h$ scales like $\eps_h$ as the functions $\hat \xi_2 /h$ and $\hat \xi_3 /\delta_h$ vanish, due to the fact that $\int_{-1/2}^{1/2} x_3 {\rm d}x = 0$, see \eqref{assumption:clampedboundary}.  Then, by \MMM using  \eqref{eq:closetoidentityzeta},  \eqref{eq: zeta-con2}, and the trace estimate we get \EEE
\begin{align*}
	\vert (Q^h- \Id)e_3 \vert\leq C\Vert(Q^h- \Id)e_3\Vert_{L^2(\Gamma)} \le \AAA C \EEE \delta_h^{-1}  \big(	\Vert \bar \Psi^h \Vert_{L^2(\Gamma)}  + 	\Vert \Psi^h \Vert_{L^2(\Gamma)} \big)  \leq C\eps_h/\delta_h.
\end{align*}
Using \AAA that \EEE the norm of each column is equal to $1$, we see that the remaining entry \MMM $e_2^{\top}(Q^h- \Id)e_2$ \EEE scales \MMM  like \EEE $\eps_h^2/\delta_h^2$. \MMM This along with \eqref{assumption:energyscalingthickness1} shows  \EEE \eqref{eq:deviationfromidentity} and  concludes the proof. \EEE
\end{proof}
We close the section by noting that all consequences of the rigidity estimate derived in \cite[Lemmata~3.3, 3.4]{Freddi2013}  also hold in our setting for $Q^h$ replaced by $\Id$. 


\section{Properties of the three- and the one-dimensional model}\label{sec: 3d-1d}

 Having presented the abstract theory in Section \ref{sec: auxi-proofs}, let us now apply this approach to our model. More precisely, we need to verify the assumptions of Theorem~\ref{th:abstract convergence 2} and Theorem~\ref{thm: sandierserfaty} in our  application. In this section, we collect the relevant properties of the three- and the one-dimensional \FFF model \EEE and postpone the convergence results to the next section.

\subsection{Properties of the three-dimensional model} \label{sec:3dproperties}

By Lemma~\ref{lem:rigidity}\ref{lem:rigidity:energycontrol} we see that the $SO(3)$-valued maps $R^h$ approximate the scaled gradient \AAA at \EEE the scaling of the energy, see \eqref{assumption:energyscaling}. Due to frame indifference, \MMM the energy is  essentially controlled by the distance of $ (R^h)^{\top} \nabla_h  y^h$ from $\Id$. Therefore, we introduce the quantity \EEE  
\begin{align}\label{def:G^h}
 G^h(y^h) := \frac{(R^h)^{\top} \nabla_h  y^h - \Id	}{\eps_h}.
\end{align}
Recall the definitions of the energies $\phi_h$ and the distances $\mathcal{D}_h$ in \eqref{assumption:energyscaling} and \eqref{assumption:dissipationsclae}, \SSS as well as the quadratic forms in \eqref{eq:quadraticformsnotred}. \EEE  As a consequence of Lemma~\ref{lem:rigidity}, we obtain estimates concerning $\phi_h$ and $\mathcal{D}_h$ and their linearized versions. \EEE
%
%

\begin{lemma}[Dissipation and energy]\label{lemma: metric space-properties}
	Let $h$ \AAA be \EEE sufficiently small \AAA and $M>0$.   Then, there exists $C=C(M)>0$ such that \EEE for all $y$, $y_0$, $y_1 \in \mathscr{S}^{3D}_{h,M}$ and all open subsets $U \subset \Omega$ we have
	\begin{itemize}
		\item[(i)] $\left\vert \int_U D^2(\nabla_h y_0, \nabla_h y_1) \LLL \,{\rm d}x \EEE- \int_U Q_D^3 (\nabla_h(y_1-y_0))\LLL \,{\rm d}x \EEE\right\vert \leq C \eps_h^\alpha / \delta_h^{\alpha}$,
		\item[(ii)] $ \left\vert \mathcal{D}_h(y_0,y_1)^2 - \int_\Omega Q_D^3 (G^h(y_0) - G^h(y_1))\LLL \,{\rm d}x \EEE \right\vert \leq C  \eps_h^\alpha / \delta_h^{\alpha} \Vert G^h(y_0)-G^h(y_1)\Vert^2_{L^2(\Omega)} \leq C \eps_h^\alpha / \delta_h^{\alpha}$,
		\item[(iii)] $\left\vert \triangle(y) \right\vert \leq C  \eps_h^\alpha / \delta_h^{\alpha}$, where $\triangle(y):= \frac{1}{\eps_h^2} \int_\Omega W(\nabla_hy ) \LLL \,{\rm d}x \EEE - \int_\Omega \tfrac{1}{2} Q_W^3(G^h(y)) \LLL \,{\rm d}x \EEE$,
		\item[(iv)]  $\left\vert \triangle(y_0) - \triangle(y_1) \right\vert \leq C  \eps_h^\alpha / \delta_h^{\alpha} \Vert G^h(y_0) - G^h(y_1) \Vert_{L^2(\Omega)} \leq C  \eps_h^\alpha / \delta_h^{\alpha}$.
	\end{itemize}
\end{lemma}
\begin{proof}
\LLL	The lemma is a \AAA variant \EEE of \cite[Lemma 4.3]{MFMKDimension} with the only difference being that suitable scalings are replaced. For the reader's convenience we included a proof in   Appendix~\hyperlink{proof:dissipationandenergy}{A}.
	\end{proof}
 \SSS The \EEE following lemma provides \AAA properties \EEE about the topology and lower semicontinuity. In addition, the lemma \AAA will help \EEE to prove lower \FFF semicontinuity \EEE of the slopes in \LLL Lemma~\ref{lem:stronguppergradient3d} and \EEE Theorem~\ref{theorem: lsc-slope}.

\begin{lemma}[Properties of ($\mathscr{S}^{3D}_{h,M}, \mathcal{D}_h$) and $\phi_h$]\label{th: metric space}
Let $M>0$. \EEE For $h>0$ sufficiently small we have
	\begin{itemize}
		\item[(i)] {\rm Completeness:} $(\mathscr{S}^{3D}_{h,M}, \mathcal{D}_h)$ is a complete metric space.
		\item[(ii)] {\rm Compactness:} If $(y_n)_n \subset \mathscr{S}^{3D}_{h,M}$, then $(y_n)_n$ admits a subsequence converging weakly in $W^{2,p}(\Omega;\R^3)$ and  strongly in  $W^{1,\infty}(\Omega;\R^3)$.
		\item[(iii)] {\rm Topologies:} The \BBB topology \LLL on $\mathscr{S}^{3D}_{h,M}$ \EEE  induced by $\mathcal{D}_h$ coincides with the weak $W^{2,p}(\Omega;\R^3)$ topology.  
		\item[(iv)] {\rm Lower semicontinuity:} \LLL Let $(y_n)_n \subset \mathscr{S}^{3D}_{h,M}$ be a sequence such that $\mathcal{D}_h(y_n,y) \to 0$ for some $y\in \mathscr{S}^{3D}_{h,M}$. Then,   $\liminf_{n \to \infty} \phi_h(y_n) \ge \phi_h(y)$. \EEE
		\item[(v)]  {\rm Local equivalence of metrics:} Define $\overline{\mathcal{D}}_h(y_0,y_1) = \Vert \LLL \nabla_h y_0 - \nabla_h y_1 \Vert_{L^2(\Omega)}$ for $y_0,y_1\in \mathscr{S}^{3D}_{h,M}$. Then, there exists \QQQ $\QQQ \kappa \EEE _h>0$ \EEE such that  for all $y_0,y_1\in \mathscr{S}^{3D}_{h,M}$ with  \SSS $\Vert \nabla_h y_0 - \nabla_h y_1\Vert_{L^{\infty}(\Omega)} \leq \QQQ \kappa \EEE _h$ \EEE  it holds that \EEE
		\begin{align*}
		c_h \LLL \overline{\mathcal{D}}_h(y_0,y_1) \EEE \leq  \mathcal{D}_h(y_0,y_1) \leq C_h \LLL \overline{\mathcal{D}}_h(y_0,y_1) \EEE ,
		\end{align*}
		\MMM where $C_h,c_h >0$ and $\QQQ \kappa \EEE _h$ depend on $h$ \AAA and $M$. \EEE
	\end{itemize}
\end{lemma}

  Here, we remark that the positivity of the metric is quite delicate  as $D$ is only a true distance when restricted to {\emph{symmetric}} and {\emph{positive definite}} matrices.  A similar problem arises in the proof of (v) \SSS due to the nontrivial kernel of the Hessian of $D^2$, see Lemma \ref{lemma: ele}. We follow the strategy devised in \cite{MielkeRoubicek} to circumvent this problem, by using \EEE the following version of Korn's inequality  \AAA (see \cite[Corollary~4.1]{Pompe} and also \cite[Theorem~3.3]{MielkeRoubicek}), \EEE  which itself  is based on \EEE   \cite{Neff}.  
  
\begin{theorem}[Generalized Korn's inequality]\label{pompe}
	Consider an open, bounded, \MMM and connected \EEE set $U \SSS \subset \R^3\EEE$ with Lipschitz boundary. Let $S \subset  \partial U$ be a \MMM \AAA nonempty \EEE open subset. \AAA Moreover, let $\mu >0$ \AAA be \EEE small and $\gamma\in (0,1]$. Then \EEE  there exists a constant $c = c(U,S,\mu,\gamma) >0$ such that  for all $\gamma$-Hölder continuous $F\QQQ \colon \EEE  \overline{U}   \to \R^{3 \times 3}$  with $\inf_{x\in \bar U}\det F(x) \geq \mu$  \AAA and $\Vert F\Vert_{C^\gamma(U)} \le 1/\mu$ we have \EEE
	\begin{align*}
	\int_U \vert   \nabla u(x)^\top F(x)^\top  +  F(x) \nabla u(x) \vert^2 \,{\rm d}x \geq c \int_U \vert \nabla u(x) \vert^2 \,{\rm d}x \ \ \text{for all }  u \in W^{1,2}(U;\R^3) \SSS \text{ with $u=0$ on  \EEE $S$}.
	\end{align*}
\end{theorem}

\begin{proof}[{Proof of Lemma~\ref{th: metric space}}]
We first show (v). 
		Consider $y_0,y_1\in \mathscr{S}^{3D}_{h,M}$ and recall the projection $p_h\colon \Omega \to \Omega_h$, $p_h(x_1,x_2,x_3) = (x_1, hx_2, \delta_h x_3)$. We introduce the variables $z_i = y_i \circ p_h^{-1}$, $i =0,1$, as it is convenient to formulate the problem for the original (not rescaled) functions $z_0$ and $z_1$ defined  \SSS on $\Omega_h$.   Then, by Lemma~\ref{lem:rigidity}\ref{lem:rigidity:dev:GradIDInfty} we can fix $h>0$ such that $\nabla z_0$ is contained in a  small  neighborhood of $\Id$, and thus  Lemma \ref{lemma: ele} is applicable. Further, consider $\QQQ \kappa \EEE >0$ and let $Z := \nabla z_1 - \nabla z_0$ be such that $\Vert Z \Vert_{L^\infty(\Omega_h)} \leq \QQQ \kappa \EEE $. \SSS By a Taylor expansion and 		Lemma \ref{lemma: ele} we find \AAA constants $C,c >0$ \EEE such that
\begin{align*}
D^2( \nabla z_0 + Z , \nabla z_0) \geq  \frac{1}{2}\partial^2_{F_1^2} D^2(\nabla z_0, \nabla z_0)[Z,Z] -C|Z|^3  \ge c|{\rm sym}((\nabla z_0)^\top \QQQ  Z \EEE)|^2	 -C\QQQ \kappa \EEE |Z|^2 \end{align*} pointwise in \QQQ$\Omega_h$. 		 \EEE
%
%
%
%
%
%
%
%
%
%
%
We now use Theorem~\ref{pompe} applied to $F := (\nabla z_0)^{\top}$ and $u:=z_0 - z_1 \in W^{1,2}(\Omega_h;\R^3)$, where the $\gamma$-Hölder continuity of $F$ with $\Vert F \Vert_{C^\gamma(\Omega_h)} \leq 1/\mu$ for some $\gamma \in (0,1]$ and $\mu>0$ follows by a Sobolev embedding as $p>3$. Moreover, for $h>0$ sufficiently small, $\mu>0$ can be chosen such that $\inf_{x\in \bar U}\det F(x) \geq \mu$ due to Lemma~\ref{lem:rigidity}\ref{lem:rigidity:dev:GradIDInfty}. As the result is applied on $\Omega_h$ the corresponding constant is $h$-dependent, denoted by $C_h$. Recalling \eqref{assumption:dissipationsclae}, we get \EEE
	\begin{align*}
	\mathcal{D}_h^2(y_0,y_1) = & \frac{1}{ \eps_h^2 h \delta_h}\int_{\Omega_h} D^2( \nabla z_0 +Z, \nabla z_0) \, {\rm d}x \\ 
	\geq &  \frac{C}{ \eps_h^2h \delta_h}\int_{\Omega_h} \vert \sym ( \SSS (\nabla z_0)^{\top} \EEE Z) \vert^2\, {\rm d}x -   \frac{C\QQQ \kappa \EEE }{ \eps_h^2 h \delta_h}\int_{\Omega_h} |Z|^2 \, {\rm d}x\\
	\geq &  \frac{C_h-C\QQQ \kappa \EEE }{ \eps_h^2h \delta_h}\int_{\Omega_h} \vert Z\vert^2 \,{\rm d}x \geq  \SSS \frac{C_h-C\QQQ \kappa \EEE }{ \eps_h^2}  \int_{\Omega} \vert \nabla_h (y_1-y_0) \vert^2 \,{\rm d}x. \EEE
	\end{align*}
Thus, taking $\QQQ \kappa \EEE >0$ suitably small \MMM depending on $h$, \EEE we obtain the lower bound. \SSS An easier \EEE  calculation  \MMM using that $\nabla z_0$ is uniformly bounded yields   the upper bound.

\AAA Concerning (i), \EEE  the positivity follows by an argument similar to \cite[Lemma 4.4]{MFMKDimension}. \AAA We \EEE
 include the proof in Lemma \ref{lemma: positivity} for convenience of the reader.	The other properties of a metric space follow directly from \eqref{eq: assumptions-D}. \SSS The properties (ii)--(iv) and the completeness of  $(\mathscr{S}^{3D}_{h,M}, \mathcal{D}_h)$ now follow as in \cite[Lemma 4.4]{MFMKDimension}. We omit the proof.     \EEE 
\end{proof}

\begin{lemma}[Properties of $\vert \partial \phi_h \vert_{\mathcal{D}_h}$]\label{lem:stronguppergradient3d} \LLL Let $M>0$ and consider the complete metric space $(\mathscr{S}^{3D}_{h,M}, \mathcal{D}_h)$.
\MMM For  $h>0$ sufficiently small\EEE\LLL , \EEE the local slope $\vert \partial \phi_h \vert_{\mathcal{D}_h}$ is
	\begin{itemize}
		\item[(i)] a strong upper gradient for $\phi_h$,
		\item[(ii)] lower semicontinuous with respect to the weak topology in $W^{2,p}(\Omega;\R^3)$.
	\end{itemize}
	
\end{lemma}

\begin{proof}

 	(i) As the local slope is a weak upper gradient in the sense of Definition \cite[Definition 1.2.2]{AGS} by \cite[Theorem  1.2.5]{AGS}, we only need to show that for an absolutely continuous curve $z\colon(a,b) \to \mathscr{S}^{3D}_{h,M}$  \SSS (with respect to $\mathcal{D}_h$) \EEE satisfying $\vert \partial \phi_h\vert_{\mathcal{D}_h}(z) \vert z '\vert_{\mathcal{D}_h} \in L^1(a,b)$ the curve $\phi_h \circ z$ is absolutely continuous.   It is not restrictive to assume that $(a,b)$ is a bounded interval and \AAA that $z$ \EEE can be extended by continuity to $[a,b]$.
 	We first show that $z$ is absolutely continuous with respect to $\overline{\mathcal{D}}_h$. Let $\QQQ \kappa \EEE _h>0$ be the constant from Lemma~\ref{th: metric space}(v). 
 	  As $z([a,b])$ is compact, there exists $m\in \N$ and $s_i\in [a,b]$ for $i=1,...,m$ such that
 	\begin{align*}
 		 		z\big([a,b]\big) \subset \bigcup_{i = 1}^m \{ y \in \mathscr{S}^{3D}_{h,M}: \Vert y - z(s_i) \Vert_{W^{1,\infty}(\Omega)} <h\QQQ \kappa \EEE _h \}.
 	\end{align*}
 	 This is a cover of open sets with respect to $\mathcal{D}_h$ as the topologies induced by $W^{1,\infty}(\Omega)$ and $\mathcal{D}_h$ coincide,  see Lemma~\ref{th: metric space}(iii).   	Thus, by Lemma~\ref{th: metric space}(v), for all $s,t\in [a,b]$ with $s<t$ there exists a partition $s= \tilde s_0 < \tilde s_1 < ... < \tilde s_m = t$ satisfying  $ \overline{\mathcal{D}}_h(z(\tilde s_i),z(\tilde s_{i-1}))  \leq C_h \mathcal{D}_h(z(\tilde s_i),z(\tilde s_{i-1}))$ which yields
 	\begin{align*}
 		\overline{\mathcal{D}}_h(z(s),z(t)) \leq \sum_{i = 1}^m \overline{\mathcal{D}}_h(z(\tilde s_i),z(\tilde s_{i-1})) \leq C_h \sum_{i = 1}^m \mathcal{D}_h(z(\tilde s_i),z(\tilde s_{i-1})) \leq \int_{s}^{t} C_h \vert z'\vert_{\mathcal{D}_h}(r)\, {\rm d}r < +\infty.
 	\end{align*}
 	In particular, $z$ is absolutely continuous with respect to $\overline{\mathcal{D}}_h$, see \eqref{def:absolutecontinuity}.
	 \QQQ In view of Definition \ref{main def2}(ii) and \eqref{def:metricderivative}, Lemma~\ref{th: metric space}(v) also implies that  $\vert \partial \phi_h\vert_{\overline{\mathcal{D}}_h}(z) \vert z '\vert_{\overline{\mathcal{D}}_h} \in L^1(a,b)$. \SSS It now remains to show that  the local slope $\vert \partial \phi_h \vert_{\overline{\mathcal{D}}_h}$ is a strong upper gradient and that \QQQ $\vert \partial \phi_h \vert_{\overline{\mathcal{D}}_h}\circ z$ is Borel, \EEE as then $\AAA \phi_h \EEE \circ z$ is indeed absolutely continuous by Definition \ref{main def2}(i), as desired.  To this end, \EEE for small $h>0$, we observe by
\cite[Proposition 3.2]{MielkeRoubicek} that $\phi_h$ is $\lambda$-convex with respect to $\overline{\mathcal{D}}_h$ for \MMM some  $\lambda < 0$ \EEE depending only on $M$ and $h$. More precisely, given $y,w \in \mathscr{S}^{3D}_{h,M}$, \MMM for \EEE the convex combination $y_s:= (1-s) y + sw$, $ s \in [0,1]$, \MMM it holds that \EEE 
\begin{align}\label{convexityphih}
	\phi_h(y_s) \leq (1-s) \phi_h(y) + s \phi_h(w) - \tfrac{1}{2} \lambda s(1-s) \overline{\mathcal{D}}_h(y,w)^2.
\end{align}
  Here, we note that this proposition is applicable as Lemma~\ref{lem:rigidity}\ref{lem:rigidity:dev:GradIDInfty} implies that
  \begin{align*}
  	  \Vert \nabla_h y - \nabla_h w \Vert_{L^\infty(\Omega)} \leq C \big(\Vert \nabla_h y - \Id \Vert_{L^\infty(\Omega)} + \Vert \nabla_h w - \Id \Vert_{L^\infty(\Omega)} \big) \leq C \eps_h^\alpha / \delta_h^\alpha.
  \end{align*}
\SSS Now, \eqref{convexityphih}, \AAA Lemma \ref{th: metric space}(iii), \EEE and  \EEE \cite[Corollary 2.4.10]{AGS} ensure that the local slope $\vert \partial \phi_h \vert_{\overline{\mathcal{D}}_h}$ is a strong upper gradient \QQQ and the measurability of $\vert \partial \phi_h \vert_{\overline{\mathcal{D}}_h}\circ z$. \EEE

(ii) We first show that there \MMM exist  constants \EEE $C_h >0$ and $\MMM \QQQ \kappa \EEE _h>0\EEE$ \MMM depending on $h$ \EEE such that for all $y,w \in \mathscr{S}^{3D}_{h,M}$ satisfying $ \Vert  y - w\Vert_{W^{1,\infty}(\Omega)} \leq \MMM \QQQ \kappa \EEE _h\EEE$ \MMM the functions  $y_s:= (1-s) y + sw$, $ s \in [0,1]$, satisfy \EEE
 \begin{align}\label{convexityd}
 	\mathcal{D}_h(y_s,y)^2 \leq s^2 \mathcal{D}_h(w,y)^2 (1 + C_h \Vert \nabla_h w - \nabla_h y \Vert_{L^\infty(\Omega)}).
 \end{align}
Indeed, due to the uniform bounds on $\nabla_h w$ and $\nabla_h y$ in \eqref{eq:rigidity}\ref{lem:rigidity:dev:GradIDInfty}, we obtain by a Taylor expansion
\begin{align*}
\int_\Omega \tfrac{1}{2}\partial^2_{F_1^2}D^2({\nabla_h y},{\nabla_h y})[\nabla_h (w -  y),\nabla_h (w -   y) ] \LLL \,{\rm d}x \EEE  \leq 	\int_\Omega D^2(\nabla_h y, \nabla_h w) \LLL \,{\rm d}x \EEE  + C \Vert \nabla_h (w-   y) \Vert^3_{L^3(\Omega)}.
\end{align*}
Similarly, we \AAA get \EEE
\begin{align*}
\int_\Omega D^2(\nabla_h y, \nabla_h y_s) \LLL \,{\rm d}x \EEE  \leq s^2 \int_\Omega \tfrac{1}{2}\partial^2_{F_1^2}D^2&({\nabla_h y},{\nabla_h y})[\nabla_h (w -  y),\nabla_h (w -   y) ] \LLL \,{\rm d}x \EEE  + Cs^3 \Vert \nabla_h (w-   y) \Vert^3_{L^3(\Omega)}  .
\end{align*}
\SSS Combination of  \EEE the two estimates \MMM with \LLL the lower bound in \EEE Lemma~\ref{th: metric space}(v) yields \eqref{convexityd}, \SSS provided that $\QQQ \kappa \EEE _h$ is chosen sufficiently small. \EEE     Whenever \EEE  $y \neq w$ and $  \phi_h(y) - \phi_h(w)  + \tfrac{1}{2} \MMM \lambda \SSS  \overline{\mathcal{D}}_h(y,w)^2 \EEE >0$ we obtain by \eqref{convexityphih} 
\begin{align*}
	\frac{\phi_h(y) - \phi_h(y_s)}{\mathcal{D}_h(y,y_s)} \geq \frac{ \phi_h(y) - \phi_h(w)  + \tfrac{1}{2} \lambda (1-s)  \overline{\mathcal{D}}_h(y,w)^2 \EEE}{\mathcal{D}_h(y,w)}\frac{s\mathcal{D}_h(y,w)}{\mathcal{D}_h(y,y_s)}.
\end{align*}
With \eqref{convexityd}, \QQQ for $ \Vert  y - w\Vert_{W^{1,\infty}(\Omega)} \leq \MMM \QQQ \kappa \EEE _h\EEE$, \EEE this implies
\begin{align*}
\frac{\phi_h(y) - \phi_h(y_s)}{\mathcal{D}_h(y,y_s)} \geq \frac{ \phi_h(y) - \phi_h(w)  + \tfrac{1}{2} \lambda (1-s)\LLL \overline{\mathcal{D}}_h(y,w)^2 \EEE} {\mathcal{D}_h(y,w) (1 + C_h \Vert \nabla_h w - \nabla_h y \Vert_{L^\infty(\Omega)})^{1/2} }   .
\end{align*}
Thus, by taking the limit \MMM $s \to 0 $ \EEE and the supremum, we obtain for any $\QQQ \kappa \EEE ' \leq  \MMM \QQQ \kappa \EEE _h\EEE$ 
\begin{align}\label{slopelowerbound}
\vert \partial \phi_h \vert_{\mathcal{D}_0}(y) \geq \sup\limits_{\substack{y\neq w\\\Vert y - w\Vert_{W^{\MMM 1,\infty}(\Omega)} \MMM \le \EEE \QQQ \kappa \EEE '}} \frac{ \left(\phi_h(y) - \phi_h(w)  + \tfrac{1}{2} \MMM \lambda \EEE \LLL \overline{\mathcal{D}}_h(y,w)^2 \EEE \right)^+}{\mathcal{D}_h(y,w) (1 + C_h \Vert \nabla_h w - \nabla_h y \Vert_{L^\infty(\Omega)})^{1/2} }.
\end{align}
We also \AAA get \EEE the reverse inequality
\begin{align}\label{slopeupperbound}
\vert \partial \phi_h \vert_{\AAA \mathcal{D}_h}(y) =& \limsup\limits_{w \to y}  \frac{ \left(\phi_h(y) - \phi_h(w)  + \tfrac{1}{2} \MMM \lambda \EEE \LLL \overline{\mathcal{D}}_h(y,w)^2 \EEE \right)^+}{\mathcal{D}_h(y,w) (1 + C_h \Vert \nabla_h w - \nabla_h y \Vert_{L^\infty(\Omega)})^{1/2} }  \nonumber \\
\leq& \sup\limits_{\substack{y\neq w\\\Vert y - w\Vert_{W^{\MMM 1,\infty}(\Omega)} \MMM \le \EEE \QQQ \kappa \EEE '}} \frac{ \left( \phi_h(y) - \phi_h(w)  + \tfrac{1}{2} \MMM \lambda \EEE \LLL \overline{\mathcal{D}}_h(y,w)^2 \EEE \right)^+}{\mathcal{D}_h(y,w) (1 + C_h \Vert \nabla_h w - \nabla_h y \Vert_{L^\infty(\Omega)})^{1/2} }  ,
\end{align}
where we used \AAA Lemma \ref{th: metric space}(v) and the fact \EEE that $w \to y$ with respect to $\mathcal{D}_h$ implies $\Vert y - w \Vert_{W^{1,\infty}(\Omega)} \to 0$, see \MMM Lemma~\ref{th: metric space}(ii),(iii). \EEE  We are now ready to confirm
the lower semicontinuity. Consider a sequence $(y_n)_n$, $y_n \in \mathscr{S}^{3D}_{h,M}$, converging weakly to $y \in \mathscr{S}^{3D}_{h,M}$ in $W^{2,p}(\Omega)$ (or equivalently with respect to $\mathcal{D}_h$, see \MMM  Lemma~\ref{th: metric space}(iii)). \EEE
For $w \neq y$ such that $\Vert w- y \Vert_{W^{\MMM 1,\infty}(\Omega)} \leq \MMM\QQQ \kappa \EEE _h/2 \EEE $, \MMM we have \EEE $w\neq y_n$  and  $\Vert w- y_n \Vert_{W^{\MMM 1,\infty}(\Omega)} \leq \MMM \QQQ \kappa \EEE _h \EEE $ \LLL for $n$ large enough, \EEE and thus
\begin{align*}
	\liminf\limits_{n\to \infty} \vert \partial \phi_h\vert_{\mathcal{D}_h}(y_n) \MMM & \ge \EEE  \liminf\limits_{n\to \infty} \frac{ \left( \phi_h(y_n) - \phi_h(w)  + \tfrac{1}{2} \MMM \lambda\LLL \overline{\mathcal{D}}_h(y_n,w)^2 \EEE \right)^+}{\mathcal{D}_h(y_n,w) (1 + C_h \Vert \nabla_h w - \nabla_h y_n \Vert_{L^\infty(\Omega)})^{1/2} } \\  
	&\geq \frac{ \left( \phi_h(y) - \phi_h(w)  + \tfrac{1}{2} \MMM \lambda \EEE \LLL \overline{\mathcal{D}}_h(y,w)^2 \EEE \right)^+}{\mathcal{D}_h(y,w) (1 + C_h \Vert \nabla_h w - \nabla_h y \Vert_{L^\infty(\Omega)})^{1/2} },
\end{align*}
where we used \AAA Lemma~\ref{th: metric space}(ii)--(iv), \EEE and \LLL \eqref{slopelowerbound} for $\QQQ \kappa \EEE ' = \QQQ \kappa \EEE _h$.  \EEE By taking the supremum with respect to $w$, \MMM and using \LLL \eqref{slopeupperbound} for $\QQQ \kappa \EEE ' = \QQQ \kappa \EEE _h/2$, \EEE the
lower semicontinuity follows.

\end{proof}

\subsection{Properties of the one-dimensional model}\label{sec:1dproperties}

We now \SSS briefly \EEE collect some facts of the one-dimensional \SSS model.  Recall the definitions of $\mathscr{S}^{1D}$, $\phi_0$, and $\mathcal{D}_0$ in \eqref{eq: s1D}, \eqref{def:enegeryphi-1D}, and \eqref{def:metric}.

\begin{lemma}[Properties of $(\mathscr{S}^{1D},\mathcal{D}_0)$, $\phi_0$, \MMM and \EEE $\vert \partial{\phi}_0\vert_{\mathcal{D}_0}$]\label{lem:complete1d}\label{Thm:representationlocalslope1}
	
	Let $M>0$, $\Phi^1(t) := \sqrt{t^2 + C  t^3 + C  t^4}$ and $\Phi_{M}^2(t):= C\sqrt{ M} t^2 + C t^3 +  Ct^4$ for any $C>0$ large enough. Then \LLL it holds that \EEE
	\begin{itemize}
		\item[(i)] {\rm Completeness:} $(\mathscr{S}^{1D},\mathcal{D}_0)$ is a complete metric space.
		\item[(ii)] {\rm Generalized convexity:}
		\MMM For \EEE all $(u_0, \theta_0) \in \mathscr{S}^{1D}$ satisfying $\phi_0(u_0,\theta_0) \leq M$ and all $(u_1, \theta_1) \in \mathscr{S}^{1D}$ we have
		\begin{itemize}[leftmargin = 2cm]
			\item[{\rm (ii')}] $\mathcal{D}_0((u_0,\theta_0),(u_s,\theta_s)) \le s \Phi^1\big(\mathcal{D}_0((u_0,\theta_0),(u_1,\theta_1))\big)$
			\item[{\rm (ii'')}] $\phi_0(u_s,\theta_s) \leq (1-s) \phi_0(u_0,\theta_0) + s \phi_0(u_1,\theta_1) + s \Phi_M^2\big(\mathcal{D}_0((u_0,\theta_0),(u_1,\theta_1))\big)$
		\end{itemize}
		for $(u_s,\theta_s):= (1-s) (u_0,\theta_0) + s (u_1,\theta_1)$ and $s \in [0,1]$. 
		\item[(iii)] {\rm Characterization of the slope:}  The local slope for the energy $\phi_0$ admits the representation
		\begin{align*}
		\vert \partial \phi_0 \vert_{\mathcal{D}_0}(u,\theta) := \sup\limits_{(u,\theta)\neq(\tilde u,\tilde \theta)\in \mathscr{S}^{1D}} \frac{\big(\phi_0(u,\theta) - \phi_0 (\tilde u,\tilde \theta) - \Phi_M^2(\mathcal{D}_0((u,\theta),(\tilde u,\tilde \theta))\EEE ) \EEE\big)^+}{\Phi^1\big(\mathcal{D}_0((u,\theta),(\tilde u,\tilde \theta) )\big)\EEE}
		\end{align*}
		for all $(u,\theta)\in \mathscr{S}^{1D}$  satisfying $\phi_0(u,\theta) \leq M$.
		\item[(iv)] {\rm Strong upper gradient:} 		The local slope $\vert \partial \phi_0 \vert_{\mathcal{D}_0}$ is a strong upper gradient for $\phi_0$. 
		\item[(v)] {\rm Topology:} Representing $(u,\theta)$ as $(\xi_1,\xi_2,\xi_3,\theta)$ by \eqref{def:Bernoulli-Navier}, the topology \LLL on $\mathscr{S}^{1D}$  induced by $\mathcal{D}_0$ \EEE  coincides with the topology induced by $W^{1,2}(I) \times W^{2,2}(I)  \times W^{2,2}(I) \times W^{1,2}(I)$. 
%
%
				
	\end{itemize}
	
\end{lemma}

\begin{proof}
	The proof for the case $r=1$ is \MMM given \EEE in \cite{MFLMDimension2D1D}, \SSS and the arguments for general $r>0$ remain unchanged. \EEE For items (i),(v), \EEE see \cite[Lemma 4.2]{MFLMDimension2D1D}, where $w$ coincides with $\xi_3$ and $y = (\xi_1 - x_2\xi_2',\xi_2)$,   and for  (ii) we refer to \cite[Lemma A.1]{MFLMDimension2D1D}. \MMM Finally, \EEE (iii) and (iv) are addressed in \cite[Lemma 4.3]{MFLMDimension2D1D}.
	The case $r=0$ is even simpler. In fact, $\mathcal{D}_0$ \AAA becomes a norm and $\phi_0$ is \EEE $\lambda$-convex for some $\lambda \geq0 $, and the theorem can be replaced by the prototypical theory \cite[Section~2.4]{AGS}.
\end{proof}

\section{Passage to the one-dimensional limit}\label{sec: gammastatic}

In this \FFF section, \EEE we relate the three-dimensional model \AAA to \EEE its limit. We first give the \AAA proof \EEE of the compactness result in Proposition~\ref{lemma:compactness}. Then, we identify the limiting strain in Lemma~\ref{lem:identificationofstrainlimit} \AAA in order \EEE to state the $\Gamma$-convergence in Theorem~\ref{th: Gamma}. Eventually, \EEE we provide the   properties needed in the present evolutionary setting (see the assumptions of Theorem~\ref{th:abstract convergence 2}). Similarly to the \AAA $\Gamma$-convergence \EEE analysis, we derive the lower semicontinuity of the dissipation distances in Theorem~\ref{th: lscD}. \AAA Afterwards, \EEE we show the major part of our contribution: the lower semicontinuity of the local slopes (Theorem~\ref{theorem: lsc-slope}).

Recall the definitions of $u^h$ and $\theta^h$ in \eqref{def:udisplacement} and \eqref{def:theta}. We now prove the compactness result stated in Proposition~\ref{lemma:compactness}. The key argument is to prove a $W^{1,2}(\Omega;\R^3)$-bound on $u^h$ with a Korn-Poincaré inequality. \SSS Whereas in  \cite{Freddi2013} \EEE such an inequality is used for functions satisfying $\int_\Omega {\rm skew} (\nabla_h y^h) \,  {\rm d}x = 0$ and $\int_\Omega y^h \, {\rm d}x = 0$, \SSS we use a suitable version for \EEE boundary conditions, see \SSS \cite[Proposition 1]{hierarchy}. \EEE
\begin{proof}[Proof of Proposition \ref{lemma:compactness}]
	By definition we have
	\begin{align*}
	y_1^h = x_1 + \eps_h u_1^h, \quad y_2^h = hx_2 + \eps_h \tfrac{u_2^h}{h}, \quad y_3^h = \delta_h x_3 + \eps_h \tfrac{u_3^h}{\delta_h}.
	\end{align*}
	Thus, we can \AAA write \EEE
	\begin{align}\label{eq:ucharacterizationlimit}
	\frac{\nabla_h y^h -\Id}{\eps_h} = \begin{pmatrix}
	u_{1,1}^h & u_{1,2}^h/h & u_{1,3}^h/\delta_h\\
	u_{2,1}^h/h & u_{2,2}^h/h^2& u_{2,3}^h / (h\delta_h)\\
	u_{3,1}^h/\delta_h & u_{3,2}^h /(h\delta_h)& u_{3,3}^h / \delta_h^2
	\end{pmatrix}.
	\end{align}
	Recall the boundary conditions \eqref{assumption:clampedboundary} 
	\begin{align*}
\AAA u^h = 	\hat U:=   \EEE 	(
 \hat \xi_1 - x_2 \hat \xi_2' - x_3 \hat \xi_3', \hat \xi_2 , \hat \xi_3 )^\top \quad \text{on $\Gamma$.}
	\end{align*} 
\AAA Then, \EEE by Korn's inequality \SSS  (see \cite[Proposition 1, equation (81)]{hierarchy}) we find \EEE
	\begin{align*}
	\Vert u^h - \hat U \Vert_{W^{1,2}(\Omega)} \leq C\Vert \sym (\nabla u^h - \nabla  \hat U) \Vert_{L^2(\Omega)}.
	\end{align*}
	\MMM Therefore, we deduce \EEE
	\begin{align*}
	\Vert u^h \Vert_{W^{1,2}(\Omega)} \leq C \left(\Vert \sym (\nabla u^h) \Vert_{L^2(\Omega)} + \MMM \Vert \hat U\Vert_{W^{1,2}(\Omega)}\right). \EEE
	\end{align*}
	As part (iii) of \cite[Lemma 3.4]{Freddi2013}  ensures that $\sym ( \tfrac{\nabla_h y^h - \Id}{\eps_h})$ is bounded in $L^2(\Omega;\R^{3 \times 3})$, we \MMM get \LLL by \eqref{eq:ucharacterizationlimit} \EEE that $\sym (\nabla u^h)$  is bounded in $L^2(\Omega;\R^{3 \times 3})$. Thus, we \EEE can extract a subsequence such that $u^h \rightharpoonup u$ in $W^{1,2}(\Omega;\R^3)$ for some $u \in W^{1,2}(\Omega;\R^3)$. 	To observe that the convergence of $u_3^h$ is strong, we multiply both sides of \eqref{eq:ucharacterizationlimit} by $\delta_h$ and obtain
	\begin{align*}
		\Vert u^h_{3,1} - u_{3,1} \Vert_{L^2(\Omega)} =  \Vert \SSS \tfrac{1}{\eps_h/\delta_h} \EEE (\nabla_h y^h- \Id)_{31} - u_{3,1}\Vert_{L^2(\Omega)}.
	\end{align*}
	An analogous observation and the fact that $\delta_h/h \to 0$ yield for $j=2,3$
		\begin{align}\label{eq: NNNNN}
	\Vert u^h_{3,j} \Vert_{L^2(\Omega)} \leq  C \tfrac{h}{\eps_h/\delta_h} \Vert (\nabla_h y^h- \Id)_{3j}\Vert_{L^2(\Omega)}.
	\end{align}
	Then, as   Lemma~\ref{eq:rigidity}(i)--(iii) imply that $\frac{1}{\eps_h/\delta_h} (\nabla_h y^h - \Id) $ converges strongly in $L^2(\Omega;\R^{3\times 3})$ (see \cite[Lemma 3.4(i)]{Freddi2013} for a detailed calculation),   \AAA the \EEE previous  estimates  imply the strong convergence of $u^h_3$.

\AAA To conclude the proof of (i), we need to  \EEE characterize the limit $u$. As $\sym (\nabla u^h)$ converges weakly in $L^2(\Omega;\R^{3\times 3})$, we can infer from \eqref{eq:ucharacterizationlimit} that $u$ satisfies the identities $u_{i,j} + u_{j,i} = 0$ for all $i = 1,2,3$ and $j = 2,3$. \AAA Moreover, from \eqref{eq: NNNNN} we deduce $u_{3,2} = 0$ and thus also $u_{2,3} = 0$. \EEE   By arguing analogously to \cite[Theorem 4.1]{ABJMV}, we see that $u$ is a Bernoulli-Navier function, i.e., $u \in \mathcal{A}^{BN}_{\hat \xi_1, \hat \xi_2, \hat \xi_3}$, see \eqref{def:Bernoulli-Navier} for the definition.  Here, our choice of
	the boundary values in \eqref{assumption:clampedboundary} becomes apparent since it guarantees that the limit of the
	displacements \AAA lies in \EEE $\mathcal{A}^{BN}_{\hat \xi_1, \hat \xi_2, \hat \xi_3}$. Indeed, we have $u^h = \hat U$ on $\Gamma$, \MMM and therefore also the limit satisfies \EEE $u_1 =\hat \xi_1 - x_2 \hat \xi_2' - x_3 \hat \xi_3'$, $u_2 = \hat \xi_2$, and $u_3 = \hat \xi_3$ on \MMM $\Gamma$. 	
		
	This \SSS concludes the proof of \EEE (i). \EEE The compactness for $\theta_h$ is a consequence of \cite[Lemma 3.7]{Freddi2013}. One only needs to observe that $\theta_h$ satisfies zero boundary conditions on $\Gamma$ due to \eqref{assumption:clampedboundary}. Therefore, $\theta \in W^{1,2}_0(\MMM I \EEE )$.
\end{proof}

%
 
As observed in the discussion above equation \eqref{def:G^h}, the energy $\phi_h(y^h)$ is essentially controlled by $G^h(y^h)$, which converges weakly to  some $G_y \in L^2(\Omega;\R^{3\times 3})$, \QQQ due to Lemma~\ref{lem:rigidity}(i). \EEE As we want to represent the $\Gamma$-limit by the variables $u$ and $\theta$, we therefore need to characterize the limiting strain \AAA $G_y$ \EEE in terms of $u$ and $\theta$. This is addressed in the next lemma, \SSS along with an analysis \FFF of \EEE the limit behavior of the rotations $R^h$.\EEE

\begin{lemma}[Identification of the scaled limiting strain]\label{lem:identificationofstrainlimit}
	Consider a sequence $(y^h)_h$ with $y^h\in \mathscr{S}^{3D}_{h,M}$ for $M>0$. \EEE Let $(u,\theta) \in \mathscr{S}^{1D}$ be the limit \SSS given by \EEE Proposition~\ref{lemma:compactness}. Then, there exists $G_y \in L^2(\Omega;\R^{3 \times 3})$ and $A_{u,\theta} \in W^{1,2}(\Omega;\R^{3 \times 3})$ such that, up to subsequences,
	\begin{itemize}
		\item[(i)] $\SSS G^h(y^h) \EEE \rightharpoonup G_y$ in $L^2(\Omega;\R^{3 \times 3})$ and we have 
		$${(G_y)}_{11} (x) = \xi_1'(x_1) - x_2 \xi_2''(x_1) - x_3 \xi_3''(x_1) + \tfrac{r}{2} (\xi_3'(x_1))^2 \quad \text{\QQQ for a.e.\ } x\in \Omega,$$ \MMM as well as \EEE $$ \frac{1}{2} \big({(G_y)}_{12} + {(G_y)}_{21}\big)(x) = -x_3 \theta'(x_1) + \tilde g(x_1,x_2) \quad \text{\QQQ for a.e.\ } x\in \Omega$$
		 for a suitable function $\tilde g \in L^2(\MMM S \EEE )$, where $(\xi_1,\xi_2,\xi_3)$ \SSS are related to $u$ by  \EEE \eqref{def:Bernoulli-Navier}. \EEE
		\item[(ii)] $\frac{R^h - \Id}{\eps_h/\delta_h} \rightharpoonup A_{u,\theta}$  in $W^{1,2}(S;\R^{3\times3}) \quad$  and $ \quad \sym \left( \frac{R^h - \Id}{\eps_h} \right) \to \tfrac{r}{2} A_{u,\theta}^2$ in $L^2(S;\R^{3\times3}),$\linebreak where $A_{u,\theta} = e_3 \otimes p - p \otimes e_3$ for $p = (\xi_3',\theta,0)$. 
			\end{itemize}
\end{lemma}
For later purposes, we note that $A_{u,\theta}^2$ is given by
\begin{align}\label{rep:Asquared}
	A_{u,\theta}^2 = - \begin{pmatrix}
	(\xi_3')^2& \xi_3'\theta&0\\
	\xi_3'\theta&\theta^2& 0 \\
	0&0& (\xi_3')^2 + \theta^2
	\end{pmatrix}.
\end{align}

\begin{proof}
\MMM We omit \QQQ a detailed \EEE proof as the statement is essentially proven in \cite[Lemma 3.3 and 3.7]{Freddi2013}. The adaptions are straightforward: \EEE  the existence of \AAA $G_y$ \EEE follows immediately by Lemma~\ref{lem:rigidity}\ref{lem:rigidity:energycontrol}. By using the rigidity estimates from Lemma~\ref{lem:rigidity}, one can show part (i) analogously to the proof of \cite[Lemma~3.7]{Freddi2013}. Whereas the convergence of $\frac{R^h - \Id}{\eps_h/\delta_h}$ and $ \sym \left( \frac{R^h - \Id}{\eps_h} \right)$  is addressed in \cite[Lemma~3.3]{Freddi2013}, \SSS for a skew-symmetric valued limiting tensor field $A_{u,\theta}$ \EEE \MMM (use also \eqref{assumption:energyscalingthickness1}), \EEE the characterization of  $A_{u,\theta}$ is provided \MMM in the statement and part \textit{ii} of \EEE the proof of \cite[Lemma 3.7]{Freddi2013}. \SSS (Below \cite[\QQQ equation \EEE (37)]{Freddi2013} we find $A_{12} = 0$ and $A_{13} = u_{1,3} = - \xi_3'$, where we also use \eqref{def:Bernoulli-Navier}.) \EEE
\end{proof}
We now state \MMM  the \EEE  $\Gamma$-convergence result, based on \cite[Theorems 3.11 and 3.14]{Freddi2013}. \EEE To this end, recall the definition of the energies and the metrics in \eqref{assumption:energyscaling}, \eqref{assumption:dissipationsclae} and \eqref{def:enegeryphi-1D}, \eqref{def:metric},   and \AAA the \EEE topology of convergence in Proposition~\ref{lemma:compactness}, which is denoted with the symbols $\pi\sigma$ and $\pi\rho$, see Subsection~\ref{sec:mainconvergenceresult}.\EEE

\begin{theorem}[$\Gamma$-convergence of energies]\label{th: Gamma}
	$\phi_\eps$ converges to $\phi_0$ in the sense of   $\Gamma$-convergence. More precisely:	
	\noindent {\rm (i) (Lower bound)} For all $(u,\theta) \in {\mathscr{S}^{1D}}$ and all sequences $(y^h)_h$ such that  $y^h \stackrel{\pi\sigma}{\to} (u,\theta)$ we find
	$$\liminf_{h \to 0}\phi_h(y^h) \ge {\phi}_0(u,\theta). $$
	
	\noindent {\rm (ii) (Optimality of lower bound)} For all $(u,\theta) \in {\mathscr{S}^{1D}}$ there exists a sequence $(y^h)_h$, $y^h \in \mathscr{S}_h^{3D}$ \BBB for all $h > 0$, \EEE   such that  $y^h \stackrel{\pi\rho}{\to} (u,\theta)$ and  
	$$\lim_{h \to 0}\phi_h(y^h) = {\phi}_0(u,\theta).$$
\end{theorem}

\SSS By minor adaptions to our setting,  the proof follows from the one  in \cite{Freddi2013}. \EEE The main differences \MMM are \EEE that the \FFF second-order \MMM penalization \EEE should vanish in the limit, \MMM and that \EEE the recovery sequence should be compatible with the boundary conditions in \eqref{assumption:clampedboundary}. In our contribution, the proof of (i) is even simpler as the \FFF second-order \EEE \FFF perturbation \EEE provides \FFF a \EEE quantitative estimate \QQQ concerning \EEE the deviation \QQQ between \SSS the nonlinear and linearized energy,  \EEE see Lemma~\ref{lemma: metric space-properties}(iii). \EEE We include the proof of (i) \SSS for convenience. \EEE For (ii), we refer to its ansatz in \eqref{Ansatz} \FFF and Remark~\ref{rem:recoverysequence} \EEE as it can be viewed as a corollary of the proof of Theorem~\ref{theorem: lsc-slope}.
\begin{proof}[Proof of Theorem \ref{th: Gamma}\,{\rm (i)}]
	We may suppose that $\liminf_{h \to 0 } \phi_h(y^h) < +\infty$ as otherwise there is nothing to prove. Thus, we can assume that $y^h \in \mathscr{S}^{3D}_{h,M}$ for $M>0$ large enough.
	By Lemma~\ref{lemma: metric space-properties}(iii) we have
	\begin{align*}
		\liminf_{h \to 0} \tfrac{1}{\eps_h^2} \int_\Omega W(\nabla_h y^h)\LLL \,{\rm d}x \EEE  = 		\liminf_{h \to 0}  \int_\Omega \tfrac{1}{2} Q_W^3(G^h(y^h))\LLL \,{\rm d}x \EEE  .
	\end{align*}
	Then, the convexity of the quadratic form $Q_W^3$, \eqref{def:quadraticforms}, and Lemma~\ref{lem:identificationofstrainlimit}(i) imply
	\begin{align}\label{integrationx_2}
		\liminf_{h \to 0}  \int_\Omega \tfrac{1}{2} Q_W^3(G^h(y^h)) \LLL \,{\rm d}x \EEE \geq \int_\Omega \tfrac{1}{2} Q_W^1( \xi_1' - x_2 \xi_2'' - x_3 \xi_3'' + \tfrac{r}{2} \xi_3'^2   ,   -x_3 \theta'+ \tilde g)\LLL \,{\rm d}x \EEE 
	\end{align}
for a suitable function $\tilde g \in L^2(\MMM S\EEE )$. By \QQQ a \EEE binomial expansion, Fubini's theorem, \SSS and again by  \eqref{def:quadraticforms}, \EEE we derive
\begin{align}\label{integrationx_3}
\int_\Omega \tfrac{1}{2} Q_W^1\big( \xi_1' - x_2 \xi_2'' - x_3 \xi_3'' + \tfrac{r}{2} \xi_3'^2   ,   -x_3 \theta'+ \tilde g\big) \LLL \,{\rm d}x \EEE \geq \phi_0(u,\theta).
\end{align}
The lower bound follows as the \FFF second-order \EEE perturbation is a nonnegative term, see \eqref{assumptions-P}. 
\end{proof}

\FFF
The force terms in \eqref{assumption:energyscaling} and \eqref{def:enegeryphi-1D} represent continuous perturbations of the energies. More precisely, \eqref{def:udisplacement}, strong $L^2(\Omega)$-convergence of $u_3^h$, and weak $L^2(I)$-convergence of $\tfrac{1}{\eps_h\delta_h}  f^{3D}_h$, see \eqref{eq: forces}, yield
\begin{align*}
\lim\limits_{h\to 0}\tfrac{1}{\eps_h^2} \int_\Omega f^{3D}_h(x_1) \, y_3(x) \, {\rm d}x = \lim\limits_{h\to 0}\int_\Omega \tfrac{1}{\eps_h\delta_h}  f^{3D}_h(x_1) \, u_3^h(x) \, {\rm d}x = \int_I f^{1D} \, u_3 \, {\rm d}x_1.
\end{align*}
\EEE Before commenting on the ansatz of (ii), we also collect the corresponding result for the dissipation distances.
\begin{theorem}[Lower semicontinuity of dissipation distances]\label{th: lscD}
	Let $M>0$. Then, for sequences $(y^h)_h$ and  $(\tilde y^h)_h$, $y^h,\tilde y^h \in \mathscr{S}^{3D}_{h,M}$, with  $y^h \stackrel{\pi\sigma}{\to} (u,\theta)$ and $\tilde y^h \stackrel{\pi\sigma}{\to} (\tilde u,\tilde \theta)$  we have
	$$\liminf_{h \to 0} \mathcal{D}_h(y^h,\tilde y^h ) \ge {\mathcal{D}_0} \big ((u,\theta),(\tilde u, \tilde \theta) \big).$$
\end{theorem}

\begin{proof}
	One can follow the lines of the proof of Theorem \ref{th: Gamma}(i), with the only difference being that we use the quadratic forms $Q_D^3$ instead of $Q_W^3$ and Lemma~\ref{lemma: metric space-properties}(ii) instead of Lemma~\ref{lemma: metric space-properties}(iii).
\end{proof}

We now discuss the ansatz for the recovery sequence, based on \cite{Freddi2013}. By a general approximation argument for $\Gamma$-convergence, it suffices to prove the upper bound on a dense subset. To this end, we use the following lemma.

\begin{lemma}[Density of smooth functions with same boundary conditions]\label{lemma: density}
	For each $\MMM u \EEE \in \mathcal{A}^{BN}_{\hat \xi_1, \hat \xi_2, \hat \xi_3} $ \MMM represented by $(\xi_1,\xi_2,\xi_3)$ \EEE we find sequences $(\xi_{1,\Lambda})_\Lambda \subset W^{2,p}(I)$ and $(\xi_{i,\Lambda})_\Lambda \subset W^{3,p}(I)$, $i=2,3$, such that 
\begin{thmlist}
	\item $\xi_{1,\Lambda} = \BBB \hat{\xi}_1,\quad$ $ \xi_{i,\Lambda} = \BBB \hat{\xi}_i,\quad$   $\xi_{i,\Lambda}' = \BBB \hat{\xi}_i'\quad$ \, on $ \partial I\quad\quad \,$ $\quad$ for \MMM $ i = 2,3$, \EEE
	\item $\xi_{1,\Lambda} \to \xi_1 $ in  $W^{1,2}(I),\quad$  $\xi_{i,\Lambda} \to \xi_i \ \ \ \ $  in  $W^{2,2}(I)$ $\quad$ for \MMM $i = 2,3$ \LLL$\quad$ as $\Lambda \to 0$. \EEE
\end{thmlist}
\end{lemma} 
\begin{proof}
	The proof is standard. We approximate $\xi_i-\hat{\xi_i}$, $i =1,2,3$, by smooth functions with compact support in $I$ and add $\hat \xi_1 \in W^{2,p}(I)$, $\hat \xi_i \in W^{3,p}(I)$, $i =2,3$, respectively. 
\end{proof}


	Consider $(u,\theta) \in \mathscr{S}^{1D}$ with $\theta \in C_c^\infty(I)$ and $u$ having  the regularity in Lemma~\ref{lemma: density}, \MMM i.e., $u$ is represented by $(\xi_1,\xi_2,\xi_3)$  with $\xi_1 \in W^{2,p}(I)$ and $\xi_i \in W^{3,p}(I)$, $i=2,3$. As $\phi_0$ is continuous with respect to the convergence in Lemma \ref{lemma: density} and \FFF$W^{1,2}(I)$-convergence \EEE for $\theta$, \AAA by \EEE a diagonal argument, it suffices to define recovery sequences for such functions.
	 We proceed \EEE as follows. First, we find approximations    $ \xi_{3,h}' \in C_c^\infty(I)$ such that   $ \xi_{3,h}' \to \xi_3'$ in \SSS $L^4(I)$ \QQQ as $h \to 0$. \EEE  Moreover, by passing to a subsequence we can control the speed of convergence of the derivatives, and thus we find
	\begin{align}\label{def:approxrecovery}
\sup_h \big(\Vert  \xi_{3,h}' \Vert_{L^{\infty}(I)}   + \Vert h^{1/2} \xi_{3,h}' \Vert_{W^{2,\infty}(I)} \big)< +\infty.
	\end{align}
	The ansatz is given by 
	\begin{align}\tag{\textbf{A}} \label{Ansatz}
	y^{h} = \begin{pmatrix}
	x_1 \\ h x_2 \\ \delta_h x_3 \\
	\end{pmatrix} + \eps_h \begin{pmatrix}
	u_1 - hx_2x_3  \theta' \\  u_2/h - x_3  \theta \\  u_3/ \delta_h + (h/\delta_h) x_2  \theta\\
	\end{pmatrix} + \eps_h \beta^h(x),
	\end{align}
	where
	\begin{align*}
	\begin{matrix*}[l]
	\beta_1^h &:=  - h  \MMM  x_2r\EEE \xi_{3,h}' \theta , \\
	\beta_2^h&:= -h x_2 \frac{r}{2} \theta^2 , \\
	\beta_3^h&:= - \delta_h  x_3 \frac{r}{2} ((\xi_{3,h}')^2 + \theta^2) ,
	\end{matrix*}
	\end{align*}
\QQQ	and $r$ is given by \eqref{assumption:energyscalingthickness1}. \EEE
	By construction the function $\beta^h$ attains zero boundary values. \AAA Therefore, \EEE the choice of the sequence is in compliance with the boundary conditions in \eqref{assumption:clampedboundary} as \SSS $\theta \in C_c^\infty(I)$. \EEE   Thus we have $y^h \in \mathscr{S}_h^{3D}$.   Then, for later purposes, we note that the scaled gradient is given by
	\begin{align}\label{eq:scaledgradientrecovery}
	\nabla_h y^h = \MMM \Id \EEE + \eps_h (M^h + \tfrac{1}{\delta_h} \LLL A_{u,\theta} \EEE +\nabla_h \beta^h),
	\end{align}
	where
	\begin{align*}
	M^h:= \begin{pmatrix}
	\partial_1 u_1 - hx_2x_3  \theta'' &  - \xi_2' /h  - x_3 \theta'& - hx_2 \theta'/\delta_h \\
	\xi_2'/h - x_3 \theta' & 0 & 0 \\
	h x_2 \theta' /\delta_h &0 &0 \\
	\end{pmatrix}, \quad   \quad \LLL A_{u,\theta} = \EEE \begin{pmatrix}
	0 & 0& - \xi_3' \\
	0 & 0 & - \theta \\
	\xi_3' & \theta & 0 \\
	\end{pmatrix},
	\end{align*}
	and 
	\begin{align*}
	\nabla_h \beta^h =\begin{pmatrix}
	-h2x_2 \frac{r}{2} (\xi_{3,h}''\theta +  \xi_{3,h}' \theta')   &   -r \xi_{3,h}'  \theta  & 0\\
	-hx_2 \MMM r \EEE  \theta  \theta'	&  - \frac{r}{2}  \theta^2 & 0\\
	-\delta_h x_3 \MMM r \EEE ( \xi_{3,h}' \xi_{3,h}''+ \theta\theta') &0 &   - \frac{r}{2} ((\xi_{3,h}')^2 + \theta^2)\\
	\end{pmatrix} .
	\end{align*}
At this point, we remark that the perturbation $\beta^h$ is essential to handle the involved relaxation of the quadratic forms, see \eqref{def:quadraticforms}.
	We will not prove Theorem~\ref{th: Gamma}(ii) as it can be seen as corollary of the proof of Theorem~\ref{theorem: lsc-slope}, see Remark~\ref{rem:recoverysequence}.
	We only show that the \FFF second-order \EEE \FFF perturbation \EEE vanishes in the limit. For this purpose, we recall that the scaled Hessian $\nabla_h^2 y^h$ is defined by $(\nabla_h^2y)_{ijk} := h^{-\delta_{2j}-\delta_{2k}}\delta_h^{-\delta_{3j} - \delta_{3k}} (\nabla^2 y_i)_{jk}$ for $i,j,k \in \{1,2,3\}$. Now, a computation yields that $ \Vert \nabla_h^2 \beta^h \Vert_{L^p(\Omega)} \SSS \le C (\Vert {\xi}_{3,h} \Vert_{W^{3,\infty}(\Omega)} + \Vert \theta\Vert_{W^{2,\infty}(\Omega)})^{2} \EEE $. \MMM Thus, by   \eqref{def:approxrecovery}, by the fact that $\delta_h/h \to 0$, \MMM and a similar computation for the matrix $M^h$ \EEE we see
	\begin{align}\label{eq:boundedsecondgradient}
		\Vert \nabla_h^2 y^h \Vert_{L^p(\Omega)} \leq C \eps_h/\delta_h.
	\end{align}
	With \eqref{assumptions-P}  this implies 
	\begin{align}\label{eq:secondordervanishes}
		\tfrac{\zeta_h}{\eps_h^2} \int_\Omega P(\nabla^2_h y^h) \LLL \,{\rm d}x \EEE \leq C \SSS \tfrac{\zeta_h}{\eps_h^2} \EEE \int_\Omega \vert \nabla^2_h y^h \vert^p \LLL \,{\rm d}x \EEE  \leq C \SSS \tfrac{\zeta_h}{\eps_h^2} \EEE (\eps_h/\delta_h)^{p},
	\end{align}
and the \FFF second-order \EEE \FFF perturbation \EEE converges to zero \SSS by \eqref{assumption:penalizationscale}. \EEE
The remaining \SSS part of the  \EEE section is dedicated to the proof of the \MMM weak \EEE lower semicontinuity of the local slopes.

\begin{theorem}[Lower semicontinuity of slopes]\label{theorem: lsc-slope}
	For each sequence $(y^h)_h$ with $y^h \in  \mathscr{S}^{3D}_{h,M}$ such that $y^h \stackrel{\pi\sigma}{\to}  (u,\theta)$ we have 
	$$\liminf_{n \to \infty}|\partial {\phi}_{h}|_{{\mathcal{D}}_{h}}(y^h) \ge |\partial  {\phi}_0|_{ {\mathcal{D}}_0}(u,\theta).$$ 
\end{theorem}

The proof is  technical and the main difficulty of this paper. As a first intuition, we observe that it does \textit{not} suffice to use the $\Gamma$-convergence results as in Lemma \ref{lem:complete1d}(iii) metrics appear in the denominator and energies are subtracted in the enumerator. Therefore, we construct a suitable competitor sequence $(y^h+z_s^h)_h$, \FFF perturbing \EEE the deformations $(y^h)_h$ with sufficiently small strains $(z_s^h)_h$, such that the difference of the quantities $G^h(y^h)$ and $G^h(y^h + z^h_{\MMM s})$ converges \textit{strongly} in $L^2(\Omega)$ with linear control in $s$ (weak convergence is already guaranteed). This property is \LLL ensured \EEE by the following lemma.

\begin{lemma}[Strong convergence of \BBB strain differences\EEE]\label{lemma: strong convergence}
	Let $M>0$. Let $(y^h)_h$ be a sequence with $y^h \in \mathscr{S}^{3D}_{h,M}$ and let  $(z^h_s)_{s,h} \subset W^{2,p}(\Omega;\R^3)$, $h>0$, $s \in (0,1)$,  be functions with \SSS $z^h_s = 0 $ on $\Gamma$ \EEE  such that 
	\begin{align}\label{eq: strong convergence assumptions}
	{\rm (i)} & \ \  \Vert \nabla_h z^h_s \Vert_{L^{\infty}(\Omega)} + \Vert \nabla^2_h z^h_s \Vert_{\LLL L^p(\Omega)} \le   M   s \eps_h/\delta_h,\\
	{\rm (ii)} & \ \ \Vert {\rm sym}(\nabla_h z^h_s) \Vert_{L^{2}(\Omega)} \le   M   s \eps_h,\notag\\
	{\rm (iii)} & \ \ \BBB \big|  {\rm skew} (\nabla_h z^h_s)(x_1,x_2,x_3)  -  \int_{-1/2}^{1/2} {\rm skew}(\nabla_h z^h_s)(x_1,x_2,t) \, \MMM {\rm d}t \EEE \big|   \le Ms\eps_h h^{1/2} \EEE \ \ \ \text{for a.e.\ $x \in \Omega$},\notag\\
	{\rm (iv)} & \ \ \PPP \text{there exist $E^s, F^s \in L^2(\Omega;\R^{3 \times 3}) $ for $s\in(0,1)$, and $\eta(h) \to 0$ as $h \to 0$  such that  } \notag \\
	& \ \  \Vert \eps_h^{-1}{\rm sym}(\nabla_h z^h_s) - E^s \Vert_{L^2(\Omega)}  +  \Vert \delta_h\eps_h^{-1}{\rm skew}(\nabla_h z^h_s) - F^s \Vert_{L^2(\Omega)}  \le s\eta(h).\notag
	\end{align}
Then the following holds for a subsequence of $(y^h)_h$ (not relabeled):
	
\noindent	{\rm (a)} For all $h$ sufficiently small,  $w^h_s := y^h + z^h_s$ lies in $\mathscr{S}^{3D}_{h,M'}$ for some $M' = M'(M)>0$.  
	
\noindent	{\rm 	(b)} Let $(G^h(y^h))_h$, $(G^h(w^h_s))_h$ be the sequences in Lemma \ref{lem:identificationofstrainlimit}   and let $G_y$, $G^s_w$ be their limits.   \BBB Then, there  \MMM exist \EEE $C=C(M)>0$ and $\rho(h)$ with $\rho(h)\to 0$ as $h \to 0$ such that   
	\begin{align*}
	{\rm (i)} & \ \ \big\|  \big(G^h(y^h) - G^h(w^h_s)\big) -   \big(G_y - G^s_w  \big)\big\| _{L^2(\Omega)} \le s\rho(h),\notag \\
	{\rm (ii)} & \ \  \Vert G^h(y^h) - G^h(w^h_s)\Vert_{L^2(\Omega)} \le Cs.
	\end{align*}
	
\noindent	{\rm 	(c)}   Let $A_{\FFF \hat{u}_s,\hat{\theta}_s \EEE} $ and $ A_{u,\theta}$ be the limits as given in Lemma~\ref{lem:identificationofstrainlimit}{\rm (ii)}, where
$(u,\theta)$ and $(\FFF \hat{u}_s,\hat{\theta}_s \EEE)$ are the limits \LLL in the sense of Proposition~\ref{lemma:compactness} corresponding to $y^h$ and $w^h_s$, respectively.  \EEE Then, \MMM it holds that \EEE
	\begin{align}\label{eq:one-dimensionalstraindifference}
	{\rm sym}(  G_y - G^s_w ) 
	=  \tfrac{r}{2} (A_{\FFF \hat{u}_s,\hat{\theta}_s \EEE})^2 - \tfrac{r}{2} (A_{u,\theta})^2 -    E^s  \quad \text{	a.e.\ in $\Omega$.}
	\end{align}
\end{lemma}

\begin{proof}
	\BBB Let $R(y^h)$ be the $SO(3)$-valued mappings given by   Lemma \ref{lem:rigidity}. For brevity, we introduce notations for the symmetric and skew-symmetric part of $\nabla_h z^h_s$ by
	$${E(z^h_s) = {\rm sym}(\nabla_h z^h_s), \ \ \ \ \  F(z^h_s) =  {\rm skew}(\nabla_h z^h_s), \ \ \ \ \  {\thickbar F}(z^h_s) = \MMM \int_{-1/2}^{1/2} \EEE F(z^h_s) \, {\rm d}x_3.}$$
	The crucial point is to find a suitable $SO(3)$-valued mapping $R(w^h_s)$ associated to  $w^h_s = y^h + z^h_s$ depending only on $x_1$ and $x_2$ and satisfying the properties stated in Lemma \ref{lem:rigidity} (Step 1). Once $R(w^h_s)$ has been defined, we can prove properties (a)--(c) (Step 2).

	\emph{Step 1: Definition of $R(w^h_s)$.} We first define 
	$$\AAA \tilde{R}^h \EEE  =   R(y^h) \big(\Id + {\thickbar F}(z^h_s) - \tfrac{1}{2} {\thickbar F}(z^h_s)^\top {\thickbar F}(z^h_s)\big) $$ 
	on $S$. By Lemma \ref{lem:rigidity}(ii), \LLL $p>2$, and \EEE \eqref{eq: strong convergence assumptions}(i), we can check that $\AAA \tilde{R}^h \EEE$ is in a small tubular neighborhood of $SO(3)$  and satisfies \SSS $\Vert  \tilde{R}^h_{,1} \Vert_{L^2(S)} \le C\eps_h/\delta_h$ and $\Vert  \tilde{R}^h_{,2} \Vert_{L^2(S)} \le C\eps_hh/\delta_h$.    We let $R(w^h_s)$ be the map obtained from $\tilde{R}^h$ by  nearest-point projection \LLL onto \EEE $SO(3)$, see the proof of \cite[Theorem 3.2 and Lemma 3.3]{Freddi2012} \LLL and \eqref{repofapproxrot} \EEE for a similar argument.  We now check that $R(w^h_s)$ satisfies the properties stated in Lemma  \ref{lem:rigidity}. \LLL By the regularity of the projection, \EEE it holds that $R(w^h_s) \in W^{1,2}(S;SO(3))$. By \eqref{eq: strong convergence assumptions}(i) and  ${\thickbar F}(z^h_s)(x_1,x_2) \in \R^{3\times 3}_{\rm skew}$ for all $(x_1,x_2) \in S$, it follows \MMM that \EEE  
	$$
	\big\| R(w^h_s) -  R(y^h) \big(\Id + {\thickbar F}(z^h_s) - \tfrac{1}{2} {\thickbar F}(z^h_s)^\top {\thickbar F}(z^h_s)\big)   \big\|_{L^\infty(S)}  \le C\Vert {\thickbar F}(z^h_s)  \Vert^3_{L^\infty(S)} \le  Cs^3(\eps_h /\delta_h)^3.
	$$
	Indeed, this follows from  the fact that $|((\Id + A -\frac{1}{2}A^\top A)^\top (\Id + A -\frac{1}{2}A^\top A))^{\PPP 1/2 \EEE } - \Id| \le C|A|^3$ for all $A \in \R^{3 \times 3}_{\rm skew}$ \MMM and a Polar decomposition. Along with \eqref{eq: strong convergence assumptions}(iii) we thus get
	\begin{align}\label{eq: w-rotat}
	\big\| R(w^h_s) -   R(y^h) \big(\Id + F(z^h_s) - \tfrac{1}{2} F(z^h_s)^\top F(z^h_s)\big)   \big\|_{L^\infty(\Omega)}  \le Cs\eps_h \gamma_h,
	\end{align}
	for $\gamma_h:= \max\{ h^{1/2} , \MMM \eps_h^2/\delta_h^3 \EEE  \} $, which is a null sequence  due to \eqref{assumption:energyscalingthickness1}.
	First, $\Vert  {R(w_s^h)}_{,1} \Vert^2_{L^2(S)} \le C\eps_h/\delta_h$ and $\Vert  {R(w_s^h)}_{,2} \Vert^2_{L^2(S)} \le C\eps_hh/\delta_h$ by the chain rule as \LLL the gradient of the projection is uniformly bounded on a compact neighbourhood of $SO(3)$. \EEE Moreover, \eqref{eq: strong convergence assumptions}(i), \eqref{eq: w-rotat},    and the fact that $R(y^h)$ satisfies \eqref{eq:rigidity}\ref{lem:rigidity:dev:RotID} shows that $\Vert R(w^h_s) -\Id  \Vert_{L^2(S)}\le C \eps_h/\delta_h$. \BBB In a similar fashion, \eqref{eq:rigidity}\ref{lem:rigidity:dev:RotIDInfty} \PPP and $\alpha <1$ yield \BBB  $\Vert R(w^h_s) -\Id  \Vert_{L^\infty(S)}\le C(\eps_h /\delta_h)^{\alpha}$. It thus remains to check \PPP \eqref{eq:rigidity}(i), i.e., that \BBB  
	\begin{align}\label{eq: last thing to check}
	\Vert R(w^h_s)^\top \nabla_h w^h_s - \Id \Vert^2_{L^2(\Omega)}\le C\eps_h^2
	\end{align}
	holds. For notational convenience, we denote by  $\omega^h_i \in L^2(\Omega;\R^{3 \times 3})$, $i \in \N$,  (generic) matrix valued functions whose $L^2$-norm is controlled in terms of a constant independent of $h$ and $s$.  By  \eqref{eq:rigidity}\ref{lem:rigidity:dev:GradIDInfty} \PPP (applied for $y^h$), \BBB \eqref{eq: strong convergence assumptions}(i), and \eqref{eq: w-rotat} we find 
	\begin{align}\label{eq: new-strain}
	R(w^h_s)^\top \nabla_h w^h_s  = \big(  \Id + F(z^h_s) - \tfrac{1}{2} F(z^h_s)^\top F(z^h_s) \big)^\top R(y^h)^\top (\nabla_h y^h + \nabla_h z^h_s) + s\eps_h \gamma_h\omega_1^h. 
	\end{align}
	We now consider the asymptotic expansion of $R(y^h)^\top (\nabla_h y^h + \nabla_h z^h_s)$  in terms of $h$: in view of \eqref{eq: strong convergence assumptions}(ii), $\nabla_h z^h_s  = E(z^h_s) +  F(z^h_s)$,    and $\Vert R(y^h) - \Id \Vert_{L^\infty(\Omega)} \le C(\eps_h/\delta_h)^\alpha$ (see \eqref{eq:rigidity}\ref{lem:rigidity:dev:RotIDInfty}) we find  
	$$R(y^h)^\top (\nabla_h y^h + \nabla_h z^h_s) = R(y^h)^\top (\nabla_h y^h +  F(z^h_s)) + E(z^h_s) + s \eps_h (\eps_h/\delta_h)^\alpha \omega^h_2.$$	
	In a similar fashion, by  \PPP using \eqref{eq:rigidity}(i),\ref{lem:rigidity:dev:RotID} (for $y^h$)  and \eqref{eq: strong convergence assumptions}(i),(ii), \BBB  we compute 
	$$R(y^h)^\top (\nabla_h y^h + \nabla_h z^h_s) = \Id + R(y^h)^\top  F(z^h_s) + \eps_h \omega^h_3 =  \Id +  F(z^h_s) + (\eps_h/\delta_h)^2 \omega^h_4\QQQ,\EEE$$
	as well as 
	$$ R(y^h)^\top (\nabla_h y^h + \nabla_h z^h_s)   =  \Id  +  ( \eps_h/\delta_h) \omega^h_5.$$ 
	By inserting these three estimates into \eqref{eq: new-strain} and using \eqref{eq: strong convergence assumptions}(i) we find
	\begin{align}\label{eq: new-strain2}
	R(w^h_s)^\top \nabla_h w^h_s &  = R(y^h)^\top (\nabla_h y^h + F(z^h_s)) + E(z^h_s) +  F(z^h_s)^\top     + F(z^h_s)^\top F(z^h_s)\\
	&   \ \ \ - \tfrac{1}{2} F(z^h_s)^\top F(z^h_s) +  s(\eps_h \gamma_h \MMM +  \eps_h (\eps_h/\delta_h)^\alpha +  \eps_h^3/\delta_h^3 \EEE ) \,\omega^h_6\notag \\
	& = R(y^h)^\top \nabla_h y^h + (R(y^h) - \Id)^\top F(z^h_s)  + E(z^h_s) + \tfrac{1}{2} F(z^h_s)^\top F(z^h_s) + s\notag \omega^h_6 o(\eps_h), 
	\end{align}
	where in the last step we used $ F(z^h_s)^\top +  F(z^h_s) = 0$ and \eqref{assumption:energyscalingthickness1}. We now check that \eqref{eq: last thing to check} holds. Indeed, it suffices to use \MMM  \eqref{assumption:energyscalingthickness1}, \EEE \eqref{eq: new-strain2}, \eqref{eq: strong convergence assumptions}(i),(ii),  $\omega_6^h \in L^2(\Omega;\R^{3\times 3})$,    and the fact that \eqref{eq:rigidity}(i),\ref{lem:rigidity:dev:RotID}  \AAA hold \EEE for $y^h$. In conclusion, this implies that the mapping $R(w^h_s)$ satisfies the properties stated in Lemma \ref{lem:rigidity}.

	\emph{Step 2: Proof of the statement.} We are now in a position to prove the statement. \EEE

	(a) \PPP By   \eqref{eq:rigidity}\ref{lem:rigidity:dev:GradIDInfty}  and \eqref{eq: strong convergence assumptions}(i), $\nabla_h w^h_s$ is in a neighborhood of $\Id$. 	Thus, by \EEE \eqref{assumptions-W}(iii) it holds that  $\int_\Omega W(\nabla_h w^h_s) \LLL \, {\rm d}x \EEE \le C\int_\Omega \dist^2(\nabla_h w^h_s,SO(3)) \LLL \, {\rm d}x \EEE $, 	and then by \eqref{eq: last thing to check} we get    $\int_\Omega W(\nabla_h w^h_s)\LLL \, {\rm d}x \EEE \le  M' \eps_h^2$ for some $M'$ sufficiently large   (depending on $M$). Moreover, from \eqref{assumptions-P}(iii) and the triangle inequality we get $P(\nabla^2_h w^h_s) \le CP(\nabla^2_h y^h) + CP(\nabla^2_h z^h_s)$. Therefore, by possibly passing to a larger $M'$, we obtain  $\LLL \zeta_h/\eps_h^{-2} \EEE \int_\Omega P(\nabla^2_h w^h_s) \LLL \, {\rm d}x \EEE \le M'$ by \LLL \eqref{assumption:penalizationscale}, \EEE \eqref{assumptions-P}(iii), \eqref{eq: strong convergence assumptions}(i),   and the fact that  $y^h \in \mathscr{S}_{h,M}^{3D}$. Summarizing, since \SSS $z^h_s  \in W^{2,p}(\Omega;\R^3)$ with $z^h_s = 0$ on $\Gamma$ \EEE and thus $w^h_s$ also satisfies the boundary conditions \eqref{assumption:clampedboundary}, we have shown that $w^h_s \in \MMM \mathscr{S}_{h,M'}^{3D}\EEE$  for some $M'=M'(M)>0$    independent of $h$. In particular, this implies that the statement of Lemma \ref{lem:rigidity} holds for $w^h_s$ with $R(w^h_s)$ as defined in Step~1.

	(b) \BBB   By  $(u,\theta)$ we denote the limit corresponding to $y^h$ as given in \MMM Proposition~\ref{lemma:compactness}. \EEE  Recalling the definition $G^h(y^h) = \eps_h^{-1}(R(y^h)^\top \nabla_h y^h -\Id)$ we find  by Lemma \ref{lem:identificationofstrainlimit}(ii), \MMM \eqref{eq: strong convergence assumptions}(i),(iv), \EEE \eqref{assumption:energyscalingthickness1}, \EEE and \eqref{eq: new-strain2} that 
	\begin{align}\label{eq: c2}
	\big\|\big(G^h(w^h_s) - G^h(y^h)\big) - \big( r A_{u,\theta}^\top F^s + E^s +\tfrac{r}{2} (F^s)^\top F^s \big)\big\|_{L^2(\Omega)} \le   s{\rho}(h),
	\end{align}
	where ${\rho}(h) \to 0$ as $\rho \to 0$. 
	By $G^s_w$ and $G_y$ we denote the weak $L^2$-limits of $G^h(w^h_s)$ and $G^h(y^h)$, respectively, which exist by  Lemma \ref{lem:identificationofstrainlimit}. Then  \eqref{eq: c2}  implies 
	\begin{align}\label{eq: c3}
	G^s_w - G_y = r A_{u,\theta}^\top F^s + E^s +\tfrac{r}{2} (F^s)^\top F^s,
	\end{align}
	and the first part of (b) holds. The second part of (b) is a consequence of \eqref{eq:rigidity}\FFF (iii) \EEE \PPP (for $y^h$), \EEE \eqref{eq: strong convergence assumptions}(i),(ii),  \eqref{eq: new-strain2}, \FFF \eqref{assumption:energyscalingthickness1},  \EEE and the fact that $\omega^h_6 \in L^2(\Omega;\R^{3 \times 3})$.

	(c)  By  $(\FFF \hat{u}_s,\hat{\theta}_s \EEE)$ we denote the limit corresponding to $w^h_s$, \LLL resulting from Proposition~\ref{lemma:compactness}. \EEE By Lemma~\ref{lem:identificationofstrainlimit}(ii) (for $w^h_s$ and $y^h$, respectively), \eqref{eq: strong convergence assumptions}(i),(iv), and  \eqref{eq: w-rotat}  we observe that pointwise a.e.\ in $\Omega$ \MMM it holds that \EEE 
	$$A_{\FFF \hat{u}_s,\hat{\theta}_s \EEE} = \lim_{h \to 0} \frac{1}{(\eps_h/\delta_h)}(R(w^h_s) - \Id) =  \lim_{h \to 0} \Big(  \frac{\delta_h}{\eps_h}(R(y^h) - \Id)(\Id + F(z^h_s)) + \frac{\delta_h}{\eps_h} F(z^h_s) \Big) = A_{u,\theta} + F^s.   $$
	Then by  \eqref{eq: c3} and an expansion we get 
	\begin{align*}
	{\rm sym}(G^s_w - G_y) & = E^s + r\,{\rm sym} (A_{u,\theta}^\top F^s)  +\tfrac{r}{2} (F^s)^\top F^s  \\
	&= E^s + \tfrac{r}{2} (A_{\FFF \hat{u}_s,\hat{\theta}_s \EEE})^\top A_{\FFF \hat{u}_s,\hat{\theta}_s \EEE} -  \tfrac{r}{2} (A_{u,\theta})^\top A_{u,\theta} = E^s - \tfrac{r}{2} (A_{\FFF \hat{u}_s,\hat{\theta}_s \EEE})^2 +  \tfrac{r}{2} (A_{u,\theta})^2,
	\end{align*}
	where in the last step we used that $A_{u,\theta} \in \R^{3 \times 3}_{\rm skew}$ pointwise a.e.\ in $\Omega$ and thus $A_{u,\theta}^\top A_{u,\theta} = -(A_{u,\theta})^2$. \EEE
%
	%
	%
\end{proof}

Finally, we are ready to prove \QQQ Theorem~\ref{theorem: lsc-slope}. \EEE

\begin{proof}[Proof of Theorem~\ref{theorem: lsc-slope}]  The proof is divided into several steps. We first define approximations of $(\xi_1,\xi_2,\xi_3,\theta)$ which allow us to work with more regular functions (Step 1). We then construct  \emph{competitor sequences} $(w^h_s)_{h,s}$ \EEE for the local slope in the 3D setting  satisfying $w^h_s \to y^h$ as $s \to 0$    (Step 2). Afterwards, \PPP we identify the limiting strain of the sequences $(w^h_s)_h$ (Step 3), \EEE and  we \EEE prove the lower semicontinuity (Step 4). Some technical estimates are contained in Steps 5--7.
	
	\emph{Step 1: Approximation.} By Lemma \ref{lemma: density}, for $\Lambda>0$ we can fix $\xi_{1,\Lambda} \in  W_{\hat \xi_1}^{1,2}(I) \cap W^{2,p}(I)$, $\xi_{2,\Lambda} \in  W_{\hat \xi_2}^{2,2}(I)\cap W^{3,p}(I)$, $\xi_{3,\Lambda} \in  W_{\hat \xi_3}^{2,2}(I) \cap W^{3,p}(I)$ and $\theta_\Lambda \in  C_c^\infty(I)$  with 
	\begin{align}\label{eq: eps approx}
	\Vert \xi_{1,\Lambda}  - \xi_1 \Vert_{W^{1,2}(I)} + \Vert\xi_{2,\Lambda}  - \xi_2 \Vert_{W^{2,2}(I)}+ \Vert\xi_{3,\Lambda}  - \xi_3 \Vert_{W^{2,2}(I)} + \Vert \theta_\Lambda  - \theta \Vert_{W^{1,2}(I)} \le \Lambda .
	\end{align}
	This approximation will be necessary to construct sufficiently regular  competitor sequences \EEE for the local slope of the 3D setting. We denote the approximation by $(u_\Lambda,\theta_\Lambda)\in \mathscr{S}^{1D}$. We further fix $ (\tilde u, \tilde \theta) \in \mathscr{S}^{1D}$, $(\tilde{u},\tilde \theta )\neq (u,\theta),(u_\Lambda,\theta_\Lambda)$, satisfying the same regularity as $(u_\Lambda,\theta_\Lambda)$.  The tuple $(\tilde{u},\tilde \theta )$ will represent the  competitor in the local slope of the 1D setting, see Lemma \ref{Thm:representationlocalslope1}(iii). Below in \eqref{eq: reg is enough}, we will see that by approximation  it is enough to work with functions of this regularity.   The convex combinations 
	\begin{align}\label{eq: convexi}
	(\tilde u_s,\tilde \theta_s) := (1-s)(u_\Lambda,\theta_\Lambda) + s  (\tilde u, \tilde \theta), \ \  s \in [0,1],
	\end{align}
	will be the starting point for the construction of competitor sequences \PPP $(w^h_s)_{h,s}$ \EEE for the 3D setting.  In the following, $\tilde{C}, C_\Lambda$ denote generic constants which may vary from line to line, where $\tilde{C}$ may depend on $\tilde{u},u,\MMM\tilde{\theta},\theta\EEE$, and $C_\Lambda$ additionally on $\Lambda$. \EEE

	\emph{Step 2: Construction of competitor  sequences $(w^h_s)_{h,s}$.\EEE} Define \MMM $y^h_\Lambda$ and  $\tilde{y}^h_s$   as in \eqref{Ansatz} by using the functions $(u_\Lambda,\theta_\Lambda)$ and $(\tilde u_s,\tilde \theta_s)$, respectively. \PPP In view of \eqref{def:approxrecovery} and \eqref{eq: convexi}, \EEE this can be done in such a way that it holds
	\begin{align}\label{eq: uniform}
\MMM  \Vert \tilde{\xi}'_{3,s,h} -  {\xi}'_{3,\Lambda,h} \Vert_{L^{\infty}(I)} +   h^{1/2} \Vert \tilde{\xi}'_{3,s,h} -  {\xi}'_{3,\Lambda,h} \Vert_{W^{2,\infty}(I)} \le C_\Lambda s, \EEE
	\end{align}
  where \LLL  $\tilde{\xi}'_{3,s,h}$ and ${\xi}'_{3,\Lambda,h}$ \EEE denote the approximations in \eqref{def:approxrecovery}. \EEE   	We observe that   $y^h_\Lambda, \tilde{y}^h_s$ satisfy the boundary conditions, i.e., ${z}^h_s := \tilde{y}_s^h -y_\Lambda^h \in W^{2,p}(\Omega;\R^3)$ with ${z}^h_s  = 0$ on $\Gamma$. \PPP For $h>0$ small and $s \in [0,1]$, \EEE we define  
	\begin{align}\label{eq: ulitmate w-def}
	w^h_s := y^h + z^h_s= y^h - y_\Lambda^h + \tilde{y}_s^h.
	\end{align}
 By the fact that $(\tilde{u}_s-u_\Lambda, \tilde{\theta}_s-\theta_\Lambda )=  s(\tilde{u}-u_\Lambda, \tilde{\theta}-\theta_\Lambda )$, \MMM \eqref{eq:scaledgradientrecovery}--\eqref{eq:boundedsecondgradient}, \EEE and that $\delta_h/h \to 0$, we see
	\begin{align}\label{eq: w1}
	{\rm (i)} & \ \  \Vert \nabla_h z^h_s\Vert_{L^{\infty}(\Omega)}  +  \Vert \nabla^2_h z^h_s\Vert_{\LLL L^p(\Omega)}  \le C_\Lambda s \eps_h/\delta_h, \quad \quad \quad \Vert {\rm sym}(\nabla_hz^h_s) \Vert_{L^2(\Omega)} \le C_\Lambda s\eps_h ,\notag\\
	{\rm (ii)} & \ \ \big|  {\rm skew} (\nabla_h z^h_s)(x)  -  \int_{-1/2}^{1/2} {\rm skew}(\nabla_h z^h_s)(x_1,x_2,t) \, \MMM  {\rm d}t \EEE \big|   \le C_\Lambda s\eps_h \MMM \sqrt{h} \EEE \ \ \ \text{for a.e.\ $x \in \Omega$}.  
	\end{align}
 This shows that the assumptions   \eqref{eq: strong convergence assumptions}(i)-(iii) are satisfied for $(z^h_s)_{s,h}$ (for a constant $M=M(C_\Lambda)$). From \eqref{eq:scaledgradientrecovery} and \eqref{eq: uniform} we also get that \eqref{eq: strong convergence assumptions}(iv) holds for suitable $E^s$ and $F^s$. 
	In particular, by definition we observe $	 E^s \,  = \EEE \lim_{h\to 0} \frac{1}{\eps_h}{\rm sym}(\nabla_hz^h_s) \,$

	\begin{align}\label{eq:representationEs}
\small=  \begin{pmatrix}
	 \partial_1 (  {\tilde u_{s,1}} - {u_{\Lambda,1}})& \tfrac{r}{2} ( {\xi_{3,\Lambda}'}\theta_\Lambda - {\tilde \xi_{3,s}'}\tilde\theta_s   ) - x_3( \tilde \theta_s' - \theta_\Lambda') & 0 \\
\tfrac{r}{2} ( {\xi_{3,\Lambda}'}\theta_\Lambda - {\tilde \xi_{3,s}'}\tilde\theta_s   ) - x_3( \tilde \theta_s' - \theta_\Lambda')& \tfrac{r}{2} (   \theta_\Lambda^2 -  \tilde \theta_s^2 ) &0\\
	 0&0& \tfrac{r}{2} \big(  ({\xi_{3,\Lambda}'})^2 + \theta_\Lambda^2    -    ({\tilde \xi_{3,s}'})^2 -  \tilde \theta_s^2\big)
	 \end{pmatrix}
	\end{align}
	a.e.\ in $\Omega$.
Then  Lemma \ref{lemma: strong convergence}(a) implies $w^h_s \in \mathscr{S}^{3D}_{h,M'}$ for a constant $M'>0$  sufficiently large depending on $\Lambda$,  but independent of $s,h$. By Lemma \ref{th: metric space}(ii), there exists a subsequence of $(w^h_s)_s$ converging weakly in $W^{2,p}(\Omega)$. But by \eqref{eq: convexi}    and \eqref{Ansatz} \EEE   we see that $z_s^h \to 0$ strongly in $L^2(\Omega)$ as $s \to 0$ and thus by \eqref{eq: ulitmate w-def} we get that $w_s^h\to y^h$ strongly in $L^2(\Omega)$ as $s \to 0$. As the limits are unique, we obtain

	\begin{align}\label{eq: sconv}
	w_s^h \rightharpoonup y^h \ \  \ \text{in} \ \ \ W^{2,p}(\Omega;\R^3) \  \text{ as $s \to 0$}.
	\end{align}

	\PPP \emph{Step 3: Identification of limiting strains.} \EEE 
	Since the ansatz for $(y_\Lambda^h)_h$ and $(\tilde{y}_s^h)_h$  is compatible with  the convergence results in Proposition~\ref{lemma:compactness} and Lemma \ref{lem:identificationofstrainlimit}, the scaled  displacement fields corresponding to $(y_\Lambda^h)_h$ and $(\tilde{y}_s^h)_h$ converge to $(u_\Lambda,\theta_\Lambda)$ and $(\tilde{u}_{s}, \tilde{\theta}_{s})$, respectively. Thus, \PPP in view of \eqref{eq: ulitmate w-def}, \EEE the scaled  displacement fields corresponding to $w^h_s$ converge to $ (\hat{u}_{s}^\Lambda, \hat{\theta}_{s}^\Lambda):= (u - u_\Lambda + \tilde{u}_{s}, \theta- \theta_\Lambda+ \tilde{\theta}_{s})$. By \eqref{eq: convexi}, this can be rewritten as
	\begin{align}\label{eq: convexi2}
	(\hat{u}_{s}^\Lambda, \hat{\theta}_{s}^\Lambda) \MMM = \EEE  (u,\theta) + s(\tilde{u}-u_\Lambda,\tilde{\theta}-\theta_\Lambda),  \ \ \ s \in [0,1].
	\end{align} 
The limits of the mappings $G^h(y^h)$ and $G^h(w^h_s)$ given by Lemma \ref{lem:identificationofstrainlimit} are denoted by $G_y$ and  $G^s_w$,   i.e, we have (up to a subsequence)
	\begin{align}\label{eq: weakcovi}
	G^h(y^h)  \rightharpoonup G_y, \ \ \ \ \ \ \ G^h(w^h_s) \rightharpoonup G^s_w  \ \ \ \ \ \text{weakly in $L^2(\Omega;\R^{3 \times 3})$}.
	\end{align}
%
%
	Above we have checked that the assumptions  \eqref{eq: strong convergence assumptions} hold for $(z^h_s)_{s,h}$. We can therefore \QQQ use \EEE Lemma~\ref{lemma: strong convergence}(b) and obtain 
	\begin{align}\label{eq: slope-lsc-1}
	{\rm (i)}& \ \ \Vert  (G^h(y^h) - G^h(w^h_s)) -   \big(G_y - G^s_w  \big)\Vert_{L^2(\Omega)} \le s\rho_\Lambda(h), \notag\\
	{\rm (ii)} & \  \  \Vert  G^h(y^h) - G^h(w^h_s) \Vert_{L^2(\Omega)} \le C_\Lambda s,
	\end{align}
	where $\rho_\Lambda(h)$ depends on $\Lambda$ and satisfies $\rho_\Lambda(h) \to 0$ as $h \to 0$. 
	The \QQQ representations in \EEE \eqref{eq:representationEs} and \eqref{eq:one-dimensionalstraindifference} \EEE fully characterize $\sym(G_w^s- G_y)$. Especially, in view of \FFF \eqref{rep:Asquared}\EEE,  \eqref{eq: convexi}, \eqref{eq:representationEs}, and \eqref{eq: convexi2},  elementary computations lead to
	\begin{align}\label{eq:symdifference11}
	\sym(G_y - G_w^s)_{11} 
	= \tfrac{r}{2} \big( (\xi_3')^2 - {(({\hat\xi^\Lambda_{3,s}})')}^2 \big) + \partial_1 (u_1 - \hat u^\Lambda_{1,s}  )
	\end{align}
	and, \SSS after a long and tedious, but elementary computation, \EEE
	\begin{align}\label{eq:symdifference12}
	\sym(G_y - G_w^s)_{12} =& - s \SSS \tfrac{r}{2} \EEE \left( (\tilde \xi_3' - \MMM \xi_{3,\Lambda}' \EEE ) (\theta-\theta_\Lambda) + (\tilde \theta - \theta_\Lambda) (\xi_3'-\MMM \xi_{3,\Lambda}' \EEE )\right) - x_3( \theta' - (\hat\theta^\Lambda_s)'),
	\end{align}
where the variables were identified via \eqref{def:Bernoulli-Navier}. \FFF In similar fashion, we find by \eqref{rep:Asquared}, \eqref{eq:one-dimensionalstraindifference}, and \eqref{eq:representationEs}
\begin{align}
	\sym(G_y - G_w^s)_{22} &= \tfrac{r}{2} \big(  \theta^2 -(\hat\theta^\Lambda_s)^2  -\theta_\Lambda^2 + \tilde \theta_s^2   \big),\qquad
	\sym(G_y - G_w^s)_{13} = 0,\qquad \sym(G_y - G_w^s)_{23} = 0, \notag \\
	\sym(G_y - G_w^s)_{33} &=  \tfrac{r}{2} \big( (\xi_3')^2 + \theta^2 - ((\hat\xi_{3,s}^\Lambda)' )^2- (\hat \theta_s^\Lambda)^2 -  ({\xi_{3,\Lambda}'})^2 - \theta_\Lambda^2    +    ({\tilde \xi_{3,s}'})^2 +  \tilde \theta_s^2\big) . \label{tediouscomputation2}
\end{align}
\EEE Moreover,  by recalling \eqref{eq: convexi2}  \EEE we compute
	\begin{align*}
	\vert(\xi_{3,\Lambda}',\theta_\Lambda)|^2 - |(\tilde \xi_{3,s}',\tilde \theta_s)|^2 -|(\xi_3',\theta)|^2 + |(({\hat \xi^\Lambda_{3,s}})',\hat \theta_s^\Lambda )|^2 = 2s\langle (\xi_{3,\Lambda}',\theta_\Lambda) -(\xi_3',\theta), (\xi_{3,\Lambda}',\theta_\Lambda) - (\tilde \xi_3',\tilde \theta)  \rangle,
	\end{align*}
	\LLL where the  brackets denote the Euclidean scalar product. \EEE
Thus, we find by Hölder's inequality, \SSS \eqref{eq: eps approx}, \EEE and \eqref{tediouscomputation2}  
	\begin{align}\label{estimaterelaxation}
	\Vert (\sym (G_y - G_w^s) )_{ij} \Vert_{L^2(\Omega)} \leq   \tilde C s \Lambda  \EEE
	\end{align}
	for	$i,j = 2$ and $i = 3$, \LLL $j=1,2,3$.\EEE    

%
%
	
	\emph{Step 4: Lower semicontinuity of slopes.}  We will show that there exist \PPP a \EEE continuous function $\eta_\Lambda\colon [0,\infty) \to  [0,\infty)$ with $\eta_\Lambda(0) = 0$ and a constant $\tilde{C}$ depending on $u,\theta,\tilde{u},\tilde{\theta}$ such that for all $s \in [0,1]$ it holds that
	\begin{align}\label{eq: three properties}
	{\rm (i)} & \ \  {\mathcal{D}}_{h}(y^h,w^h_s)  \le{\mathcal{D}}_0 \big((u,\theta),(\hat{u}_{s}^\Lambda,\hat{\theta}_{s}^\Lambda) \big) + s \eta_\Lambda(h) + \tilde{C}s\Lambda, \notag \\
	{\rm (ii)} & \ \  \eps_h^{-2}\int_\Omega  \big( W(\nabla_h y^h) - W(\nabla_h w^h_s) \big) \LLL \, {\rm d}x \EEE \ge {\phi}_0(u, \theta) - {\phi}_0(\hat{u}_{s}^\Lambda, \hat{\theta}_{s}^\Lambda)   - s \eta_\Lambda(h) - \tilde{C}s\Lambda, \notag\\
	{\rm  (iii)} & \ \  \LLL \zeta_h\eps_h^{-2} \EEE \int_\Omega \big( P(\nabla^2_h y^h) - P(\nabla^2_h w^h_s) \big)  \LLL \, {\rm d}x \EEE\ge - s \eta_\Lambda(h).  
	\end{align}
We defer the proof of \eqref{eq: three properties} to Steps 5--7 below and now prove the lower semicontinuity. \BBB Recall the definition of $\phi_h$ in \eqref{assumption:energyscaling}. \EEE By combining the three estimates in \eqref{eq: three properties}  we obtain for all $s \in [0,1]$ 
	\begin{align*}
	\frac{({\phi}_{h}(y^h) - {\phi}_{h}(w^h_s))^+}{{\mathcal{D}}_{h}(y^h,w^h_s)}&  \ge  \frac{\big( {\phi}_0(u,\theta) - {\phi}_0(\hat{u}_{s}^\Lambda,\hat{\theta}_{s}^\Lambda) - 2s \eta_\Lambda(h) - s\tilde{C}\Lambda\big)^+ }{{\mathcal{D}}_{0}((u,\theta),(\hat{u}_{s}^\Lambda,\hat{\theta}_{s}^\Lambda)) + s\eta_\Lambda(h) + s\tilde{C}\Lambda}.
	\end{align*}
	\BBB Recall that $y^h \in \mathscr{S}^{3D}_{h,M}$ and Theorem \ref{th: Gamma}(i) imply \QQQ that \EEE $\phi_0(u,\theta) \le M$. \EEE By applying Lemma~\ref{lem:complete1d}(ii) with $(u_0,\theta_0) = (u,\theta)$ and $(u_1,\theta_1) = (\hat{u}_{1}^\Lambda,\hat{\theta}_{1}^\Lambda)$  we get
	\begin{align*}
	\frac{({\phi}_{h}(y^h) - {\phi}_{h}(w^h_s))^+}{{\mathcal{D}}_{h}(y^h,w^h_s)}  &\ge \frac{s\big({\phi}_0(u,\theta) - {\phi}_0(\hat{u}_{1}^\Lambda,\hat{\theta}_{1}^\Lambda) - \Phi^2_M\big({\mathcal{D}}_0((u,\theta),(\hat{u}_{1}^\Lambda,\hat{\theta}_{1}^\Lambda))\big)    - 2 \eta_\Lambda(h) - \tilde{C}\Lambda \big)^+ }{s\Phi^1\big({\mathcal{D}}_0((u,\theta),(\hat{u}_{1}^\Lambda,\hat{\theta}_{1}^\Lambda))\big) +  \PPP s \EEE \eta_\Lambda(h) + \PPP s \EEE \tilde{C}\Lambda}.
	\end{align*}
    Thus, in view of \eqref{eq: sconv}, Lemma \ref{th: metric space}(iii), and  Definition \ref{main def2}, we find by letting $s \to 0$  
	$$|\partial {\phi}_{h}|_{{\mathcal{D}}_{h}}(y^h)  \ge \frac{\big( {\phi}_0(u,\theta) - {\phi}_0(\hat{u}_{1}^\Lambda,\hat{\theta}_{1}^\Lambda) - \Phi^2_M\big({\mathcal{D}}_0((u,\theta),(\hat{u}_{1}^\Lambda,\hat{\theta}_{1}^\Lambda))\big)  - 2 \eta_\Lambda(h) - \tilde{C}\Lambda \big)^+ }{\Phi^1\big({\mathcal{D}}_0((u,\theta),(\hat{u}_{1}^\Lambda,\hat{\theta}_{1}^\Lambda))\big) +  \eta_\Lambda(h) + \tilde{C}\Lambda}.$$
	\PPP Letting \EEE $h \to 0$ we then derive
	$$ \liminf_{h \to 0} |\partial {\phi}_{h}|_{{\mathcal{D}}_{h}}(y^h) \ge \frac{\big( {\phi}_0(u,\theta) - {\phi}_0(\hat{u}_{1}^\Lambda, \hat{\theta}_{1}^\Lambda) - \Phi^2_M\big({\mathcal{D}}_0((u,\theta),(\hat{u}_{1}^\Lambda, \hat{\theta}_{1}^\Lambda))\big) - \tilde{C}\Lambda \big)^+ }{\Phi^1\big({\mathcal{D}}_0((u,\theta),(\hat{u}_{1}^\Lambda, \hat{\theta}_{1}^\Lambda))\big) + \tilde{C}\Lambda }. 
	$$
	\QQQ We observe that the limits of $\hat{u}_{1}^\Lambda$ and  $\hat{\theta}_{1}^\Lambda $ are given by $\tilde{u}\EEE$ and $\MMM \tilde{\theta}\EEE$, respectively, as $\Lambda \to 0$, see \eqref{eq: eps approx} and \eqref{eq: convexi2}. Thus, letting $\Lambda \to 0$ and using the \MMM  continuity of  $\phi_0$ and $\mathcal{D}_0$ with respect to the convergence in \eqref{eq: eps approx}, \EEE and then taking   the supremum with respect to $(\tilde{u},\tilde{\theta})$ we get
	\begin{align}\label{eq: reg is enough}
	\liminf_{h \to 0} |\partial {\phi}_{h}|_{{\mathcal{D}}_{h}}(y^h)   &\ge \sup\limits_{\substack{(\tilde{u},\tilde{\theta})\neq (u,\theta)\\ (\tilde{u},\tilde{\theta}) \in \QQQ{\mathscr{S}^{1D}_{\rm reg}} } }   \frac{\Big({\phi}_0(u,\theta) - {\phi}_0(\tilde{u},\tilde{\theta}) - \Phi^2_M\big({\mathcal{D}}_0((u,\theta),(\tilde{u},\tilde{\theta})) \big) \Big)^+}{\Phi^1\big({\mathcal{D}}_0((u,\theta),(\tilde{u},\tilde{\theta}))\big)},
	\end{align}
	where $ \QQQ{\mathscr{S}^{1D}_{\rm reg}}\EEE \subset  \mathscr{S}^{1D}$ denotes the subset consisting of functions $(u,\theta)$ identified \QQQ with $(\xi_1,\xi_2,\xi_3)$ via \eqref{def:Bernoulli-Navier} \EEE with regularity \QQQ $W^{2,p}(I) \times W^{3,p}(I;\R^2) \times C_c^\infty(I)$.  \EEE Since each  $(\tilde{u},\tilde{\theta}) \in \mathscr{S}^{1D}$ can be approximated  \FFF in $W^{1,2}(\Omega) \times W^{2,2}(I;\R^2) \times W^{1,2}(I)$ \EEE  by a sequence of functions in \QQQ ${\mathscr{S}^{1D}_{\rm reg}}$ \EEE (see Lemma \ref{lemma: density} \FFF and \eqref{def:Bernoulli-Navier}\EEE) and   the right-hand side is continuous with respect to that convergence, the previous inequality also holds for  $\mathscr{S}^{1D}$ instead of ${\mathscr{S}^{1D}_{\rm reg}}$.   The representation given in  Lemma~\ref{lem:complete1d}(iii) then implies
	$$\liminf_{h \to 0} |\partial {\phi}_{h}|_{{\mathcal{D}}_{h}}(y^h)  \ge |\partial  {\phi}_0|_{ {\mathcal{D}}_0}(u,\theta).$$
	To conclude the proof, it therefore remains to show  \eqref{eq: three properties}.  
	
\emph{Step 5: Proof of \eqref{eq: three properties}{\rm (i)}.} By using Lemma \ref{lemma: metric space-properties}(ii)   and \eqref{eq: slope-lsc-1} we get 
	\begin{align*} 
	{\mathcal{D}}_{h}(y^h,w^h_s)^2 &\le \int_\Omega  Q_D^3\big(G^h(y^h) - G^h(w^h_s)\big)  \LLL \, {\rm d}x \EEE + \SSS C_\Lambda \EEE (\eps_h/\delta_h)^{\alpha}\Vert G^h(y^h) - G^h(w^h_s) \Vert^2_{L^2(\Omega)} \notag\\
	& \le \int_\Omega  Q_D^3(G_y - G^s_w)  \LLL \, {\rm d}x \EEE + s^2\big(C_\Lambda (\eps_h/\delta_h)^{\alpha} + C(\rho_\Lambda(h))^2\big)\notag \\
	& = \int_\Omega  Q_D^3({\rm sym}(G_y - G^s_w))  \LLL \, {\rm d}x \EEE+ s^2\big(C_\Lambda (\eps_h/\delta_h)^{\alpha} + C(\rho_\Lambda(h))^2\big).  
	\end{align*}
	Here, the last step follows from the fact that $Q_D^3(F) = Q_D^3({\rm sym}(F))$ for $F \in \R^{3 \times 3}$, see \SSS Lemma \ref{lemma: ele}. \FFF Now, the compatibility condition \eqref{quadraticforms} comes into play. It implies together with \eqref{estimaterelaxation} that \EEE
	\begin{align} \label{eq: firstD}
	{\mathcal{D}}_{h}(y^h,w^h_s)^2  \le& \hspace{-0.1cm}  \int_\Omega  Q_D^1\big(\sym(G_y-G_w^s)_{11}, \sym(G_y-G_w^s)_{12}\big)  \LLL \, {\rm d}x     +s^2\big(\tilde C^2 \Lambda^2 \EEE + C_\Lambda (\eps_h/\delta_h)^{\alpha} + C(\rho_\Lambda(h))^2\big).
	\end{align}
	\SSS In view of \eqref{eq:symdifference11},  the representation in \eqref{def:Bernoulli-Navier}, and  quadratic expansions we get 
(omitting ${\rm d}x $ for convenience)
\begin{align*}
\int_\Omega |\sym(G_y-G_w^s)_{11}|^2   & =  \int_\Omega   \Big| \tfrac{r}{2} \big( (\xi_3')^2 - {((\hat\xi^\Lambda_{3,s})')}^2 \big) + \xi_1' - ({\hat \xi_{1,s}^\Lambda})'-x_2(\xi_2''-({\hat\xi_{2,s}^\Lambda})'')-x_3(\xi_3''-({\hat\xi_{3,s}^\Lambda})'') \Big|^2  \\
& = \int_{\QQQ I \EEE}  \Big| \tfrac{r}{2} \big( (\xi_3')^2 - {((\hat\xi^\Lambda_{3,s})')}^2 \big) + \xi_1' - ({\hat \xi_{1,s}^\Lambda})' \Big|^2 +   \tfrac{1}{12}  \QQQ \big( \EEE  | \xi_2''-({\hat\xi_{2,s}^\Lambda})'' |^2 + |\xi_3''-({\hat\xi_{3,s}^\Lambda})''|^2     \QQQ \big) \EEE
\end{align*}
In a similar fashion,  \eqref{eq:symdifference12},  quadratic expansions, Hölder's inequality, and \eqref{eq: eps approx} yield
\begin{align*}
\int_\Omega |\sym(G_y-G_w^s)_{12}|^2   & =  \int_\Omega   \Big|    s \SSS \tfrac{r}{2} \EEE \left( (\LLL \tilde \xi_3' - \MMM \xi_{3,\Lambda}' \EEE ) (\theta-\theta_\Lambda) + (\tilde \theta - \theta_\Lambda) (\LLL \xi_3'-\MMM \xi_{3,\Lambda}' \EEE )\right) + x_3( \theta' - (\hat\theta^\Lambda_s)')  \Big|^2   \QQQ \, {\rm d}x \EEE \\
& \le  \int_{\QQQ I \EEE}  \tfrac{1}{12}  |  \theta' - (\hat\theta^\Lambda_s)'|^2  \QQQ \, {\rm d}x_1 \EEE    + (\tilde{C}s)^2 \AAA \Lambda^2. \EEE
\end{align*}
Then, by \FFF \eqref{quadraticforms}  \EEE and the definition of the metric in \eqref{def:metric}  we get 
\begin{align*}
\int_\Omega  Q_D^1\big(\sym(G_y-G_w^s)_{11}, &\sym(G_y-G_w^s)_{12}\big)   \, {\rm d}x =\int_I  Q_D^0\big(\tfrac{r}{2} \big( (\xi_3')^2 - ({(\hat\xi^\Lambda_{3,s})'})^2 \big) + \xi_1' - ({\hat \xi_{1,s}^\Lambda})' \big) \QQQ \, {\rm d}x_1 \EEE\notag\\
&\hspace{-1.3cm}+\tfrac{1}{12}\int_I  Q_D^1\big((\xi_3''-({\hat\xi_{3,s}^\Lambda})'') ), ( \theta' - (\hat\theta^\Lambda_s)')\big)  \LLL \, {\rm d}x_1 \EEE +\tfrac{1}{12}\int_I Q_D^0(\xi_2''-({\hat\xi_{2,s}^\Lambda})'')  \LLL \, {\rm d}x_1 \EEE  + (\tilde{C}s )^2 \AAA \Lambda^2 \EEE  \notag\\
 & \ \ \ \ \ \ \ \ \ \ \ \ \ \ \ = {\mathcal{D}}_0 \big((u,\theta), (\hat{u}^\Lambda_s, \hat{\theta}^\Lambda_s) \big)^2 + (\tilde{C}s )^2 \AAA \Lambda^2. \EEE
	\end{align*}
  This combined with \eqref{eq: firstD} concludes the proof of  \EEE \eqref{eq: three properties}(i).

	\emph{Step 6: Proof of \eqref{eq: three properties}{\rm (ii)}.}
\FFF	In this step, we will frequently use the elementary expansion
	\begin{align}
	Q(a)-Q(b) = Q(a-b) + 2 \C[a-b,b]\label{expansion}
	\end{align}
	for all $a,b \AAA \in \R^d$, where \AAA $Q\colon \R^d\to \R$ \EEE \EEE is \EEE a quadratic form with associated bilinear form $\C$.
	\EEE
 First, by \QQQ the first inequality in Lemma \ref{lemma: metric space-properties}(iv)  and \eqref{eq: slope-lsc-1}(ii) we get 
	\begin{align}\label{eq: seco1}
	\frac{2}{\eps_h^2}\int_\Omega  \Big( W(y^h) - W(w^h_s) \Big)  \LLL \, {\rm d}x \EEE  \ge  \int_\Omega  \Big(Q_W^3(G^h( y^h)) - Q_W^3(G^h(w^h_s))\Big) \LLL \, {\rm d}x \EEE - C_\Lambda (\eps_h/\delta_h)^{\alpha}s.
	\end{align} \EEE
	Recall the definition of $\C_W^3$ in  Subsection~\ref{sec:quadraticforms}.
\FFF	Then, the triangle inequality, the weak convergence $G^h(y^h) \rightharpoonup G_y$ in $L^2(\Omega;\R^{3 \times 3})$, see \eqref{eq: weakcovi}, with $s^{-1}(G_w^s-G_y)$ as a test function, and \eqref{eq: slope-lsc-1}(i) combined with Hölder's inequality yield
	\begin{align}\label{convergence1}
	\left\vert	\int_\Omega 2\C_W^3[G^h(y^h), G^h(w^h_s) - G^h(y^h)]\Big)  \LLL \, {\rm d}x - 	\int_\Omega 2\C_W^3[G_y, G_w^s - G_y]\Big)  \LLL \, {\rm d}x \right\vert \leq s \tilde{\rho}_\Lambda(h)
	\end{align}
	for some $\tilde{\rho}_\Lambda(h)$, satisfying $\tilde{\rho}_\Lambda(h) \to 0$ as $h \to 0$. 
	Similarly, \eqref{expansion} and \eqref{eq: slope-lsc-1}(i) provide the estimate
	\begin{align}\label{convergence2}
		\int_\Omega  Q_W^3\big(G^h(w^h_s)  -G^h( y^h)\big) - Q_W^3\big(G_w^s  -G_y\big) \,{\rm d}x 
	\leq s \tilde\rho_\Lambda(h).
	\end{align}
	Thus, in view of \eqref{expansion}, \eqref{convergence1}, and \eqref{convergence2} we get \EEE
	\begin{align}\label{eq: seco2}
&	\int_\Omega  \Big(Q_W^3(G^h( y^h)) - Q_W^3(G^h (w^h_s))\Big)  \LLL \, {\rm d}x \EEE\notag \\
	=&   \int_\Omega \QQQ - \EEE \Big(Q_W^3\big(G^h(w^h_s)  -G^h( y^h)\big) + 2\C_W^3[G^h(y^h), G^h(w^h_s) - G^h(y^h)]\Big)  \LLL \, {\rm d}x \EEE\notag   \\
	\ge&  \int_\Omega \QQQ - \EEE \Big(Q_W^3 (G_w^s  -G_y ) + 2\C_W^3[G_y,G_w^s  -G_y ]\Big)   \LLL \, {\rm d}x \EEE- \AAA 2 \EEE s \tilde\rho_\Lambda(h).
	\end{align}
%
	Using \eqref{estimaterelaxation}, \FFF \eqref{quadraticforms} \EEE and the fact that $Q_W^3(F) = Q_W^3({\rm sym}(F))$ yields
\begin{align}\label{eq: firstW}
	&  - \int_\Omega  \Big(Q_W^3 (G_w^s  -G_y ) + 2\C_W^3[G_y,G_w^s  -G_y ]\Big) \LLL \, {\rm d}x \EEE    \notag \\
	 \ge&   - \int_\Omega  \Big(Q_W^1 ( \sym(G_w^s  -G_y)_{11} ,  \sym(G_w^s  -G_y)_{12} ) \notag \\&\qquad+ 2\C_W^1[(\sym(G_y)_{11} ,  \sym(G_y)_{12}) , (\sym(G_w^s  -G_y)_{11} ,  \sym(G_w^s  -G_y)_{12} ) ]\Big) \LLL \, {\rm d}x \EEE    - \tilde C s\Lambda \notag \\
	 	 \FFF = \EEE &  \AAA \int_\Omega  \Big(  \EEE  Q_W^1\big((\sym G_y)_{11}, (\sym G_y)_{12}\big) - Q_W^1 \big((\sym G_w^s)_{11}, (\sym G_w^s)_{12} \big) \Big)  \LLL \, {\rm d}x \EEE    - \tilde{C}s\Lambda,
\end{align}		
\FFF where the last equality follows from \eqref{expansion}. \EEE
Moreover, by Lemma~\ref{lem:identificationofstrainlimit}(i), \eqref{eq:one-dimensionalstraindifference}, \eqref{eq:symdifference11}, and \eqref{eq:symdifference12} it holds that
	\begin{align*}
\SSS (\sym G_y)_{11} = \partial_1 u_1   + \tfrac{r}{2} (\xi_3'(x_1))^2, \quad \quad \EEE	(\sym G_w^s)_{11} =	\partial_1  \hat u^\Lambda_{1,s}   + \tfrac{r}{2}  {(\hat\xi^\Lambda_{3,s})'}^2 
	\end{align*}
	a.e.\ in $\Omega$
	and
	\begin{align*}
	(\sym G_w^s)_{12} =-x_3 (\hat\theta^\Lambda_s)' + \tilde g(x_1,x_2) +s   \AAA \tfrac{r}{2} \EEE\left( (\LLL \tilde \xi_3' - \MMM \xi_{3,\Lambda}' \EEE ) (\theta-\theta_\Lambda) + (\tilde \theta - \theta_\Lambda) (\LLL \xi_3'-\MMM \xi_{3,\Lambda}' \EEE )\right)  
	\end{align*}
	a.e.\ in $\Omega$, where we note that $\tilde g \in L^2(\MMM S \EEE )$ is the same function as in \AAA identity \EEE
	\begin{align*}
	(\sym G_y)_{12} = -x_3 \theta'(x_1) + \tilde g(x_1,x_2) \quad \quad {\rm a.e.\ in} \ \Omega.
	\end{align*}
\SSS By expansions similar to Step 5, in particular \FFF by \EEE exploiting the structure of the quadratic form in \FFF \eqref{quadraticforms}\QQQ, \EEE we \FFF derive the estimates \EEE
\begin{align}\label{someotherestimate}
\int_\Omega  \frac{1}{2}  Q_W^1\big((\sym G_y)_{11}, (\sym G_y)_{12}\big) \, {\rm d}x & = {\phi}_0(u, \theta) +  \frac{1}{2} C_W^* \int_\Omega |\tilde g|^2 \, {\rm d}x\qquad \FFF \text{and} \EEE \nonumber \\ 
\int_\Omega  \frac{1}{2}  Q_W^1 \big((\sym G_w^s)_{11}, (\sym G_w^s)_{12} \big)\, {\rm d}x  & \le  {\phi}_0(\hat{u}_{s}^\Lambda, \hat{\theta}_{s}^\Lambda) +  \frac{1}{2}  C_W^* \int_\Omega |\tilde g|^2 \, {\rm d}x + \tilde{C}s\Lambda,
\end{align}	
	where for the second estimate we again used \eqref{eq: eps approx}. \FFF Eventually, \eqref{eq: seco1}, \eqref{eq: seco2}, \eqref{eq: firstW}, and \eqref{someotherestimate} lead to  \eqref{eq: three properties}(ii).
	
	\EEE

	\emph{Step 7: Proof of \eqref{eq: three properties}{\rm (iii)}.} By convexity of $P$ and the definition  $w^h_s = y^h - y^h_\Lambda+ \tilde{y}^h_s$, see \eqref{eq: ulitmate w-def}, we find
	\begin{align}\label{eq: Pconvi}
\LLL \zeta_h /\eps_h^{-2} \EEE \int_\Omega \big( P(\nabla_h^2 y^h) - P(\nabla_h^2 w^h_s) \big)  \LLL \, {\rm d}x \EEE\ge \LLL \zeta_h /\eps_h^{-2} \EEE  \int_\Omega \partial_Z P(\nabla_h^2 w^h_s) : (\nabla_h^2 y^h_\Lambda- \nabla_h^2 \tilde{y}_s^h) \LLL \, {\rm d}x \EEE.
	\end{align}
By  H\"older's inequality and  \eqref{assumptions-P}(iii) we get 
	\begin{align*}
	\int_\Omega &|\partial_{Z} P(\nabla_h^2 w^h_s) : (\nabla_h^2  y^h_\Lambda- \nabla_h^2 \tilde{y}_s^h)| \LLL \, {\rm d}x \EEE \le \Vert  \partial_Z P(\nabla_h^2 w^h_s) \Vert_{L^{p/(p-1)}(\Omega)} \Vert  \nabla_h^2 \tilde{y}^h_s- \nabla_h^2 y_\Lambda^h \Vert_{L^{p}(\Omega)}\\
	& \le C\Big( \int_\Omega P(\nabla_h^2 w^h_s) \LLL \, {\rm d}x \EEE \Big)^{\frac{p-1}{p}}   \Vert  \nabla_h^2 \tilde{y}^h_s- \nabla_h^2 y_\Lambda^h \Vert_{L^{p}(\Omega)} .
	\end{align*}
	Using $\phi_h(w^h_s) \le    M'$ since $w^h_s \in \mathscr{S}^{M'}_h$   (see Lemma \ref{lemma: strong convergence}(a))   and   \eqref{eq: w1}(i) we then derive 
	\begin{align*}
	&\int_\Omega |\partial_{Z} P(\nabla_h^2 w^h_s) : (\nabla_h^2  y^h_\Lambda- \nabla_h^2 \tilde{y}_s^h)|  \LLL \, {\rm d}x \EEE \le \FFF C_\Lambda \EEE s(\eps_h/\delta_h)\Big( \int_\Omega P(\nabla_h^2 \FFF w_s^h)  \LLL \, {\rm d}x \EEE\Big)^{\frac{p-1}{p}} \hspace{-0.2cm} \\ 
	& \le C_\Lambda s(\eps_h/\delta_h)  \SSS (\eps_h^2 \zeta_h^{-1})^{\frac{p-1}{p}}  \FFF \leq \EEE    C_\Lambda s \big( \zeta_h /\eps_h^{-2}  \big)^{-1} \Big( \eps_h^p\delta_h^{-p} \zeta_h \eps_h^{-2} \Big)^{1/p},
	\end{align*}
	where $C_\Lambda$ depends also on $M'$. By \eqref{eq: Pconvi} \SSS and  \eqref{assumption:penalizationscale} \EEE   we finally get that \eqref{eq: three properties}(iii) holds. 
\end{proof}

\begin{rem}\label{rem:recoverysequence}
{\normalfont
	Step 6 of the proof of Theorem~\ref{theorem: lsc-slope} ensures the existence of a recovery sequence in Theorem~\ref{th: Gamma}(ii). Consider $(\tilde u, \tilde \theta) \in \mathscr{S}^{1D}$ and set $s=1$, \MMM  $y^h = y_\Lambda^h = \FFF (x_1,h x_2, \delta_h x_3)^T$, and $(u,\theta) = 0$ \EEE in the proof above, \MMM where it suffices that boundary conditions are satisfied for the recovery sequence $(w^h_1)_h$  but not for the other functions.  Then, \eqref{eq: three properties}(ii),(iii) are even simpler and we obtain
	\begin{align*}
		\limsup\limits_{h \to 0} \phi_h(w_s^h)\leq \phi_0(\tilde u, \tilde \theta).
	\end{align*}}
\end{rem}

\begin{rem}[Vanishing dissipation effect]\label{rem: VDE}
	{\normalfont  
		A potential extension for thin materials with general Poisson ratio is to consider a vanishing dissipation effect in the \QQQ $x_2$- and $x_3$-direction. \EEE A possible \SSS   modeling choice is the  \EEE dissipation potential 
		$R(F,\dot F) = \tfrac{1}{2} \vert A_h (F^{\top} \dot F + \dot F^{\top} F) \vert ^2$, where $A_h^2$ a positive definite, symmetric matrix satisfying $F: A_h^2 F = a_{11} F_{11}^2 + a_{12} F_{12}^2 + \QQQ o(1) \EEE$. \SSS Taking  \EEE  \eqref{intro:R} into account, \SSS an \EEE associated dissipation distance satisfying \eqref{eq: assumptions-D} is given by $D(F_1,F_2) := \vert A_h(F_1^{\top}F_1) - A_h(F_2^{\top}F_2) \vert $. 		From \AAA a \EEE mathematical point of view, this allows us to relax the restrictions for the elastic potential, \SSS more precisely, \EEE the conditions on $Q_W^3$ in \eqref{quadraticforms}. The crucial part is to ensure that the competitor $w_s^h$ in the proof Theorem~\ref{theorem: lsc-slope} still satisfies \eqref{eq: three properties}. 
		By construction of the quadratic forms in \eqref{eq:quadraticformsnotred}, there holds $Q_R^3(F) = \vert A_h \sym(F)\vert^2$ and thus $\lim_{h \to 0}  Q_R^3(F)  = Q_{R}^1(F_{11},F_{12} )$. In particular, it is simpler to verify \eqref{eq: three properties}(i). On the other hand, its harder to ensure \eqref{eq: three properties}(ii). For this purpose, one would need  to employ the (slightly more general) recovery sequence from the purely static setting in \cite{Freddi2013}. \SSS Note that we still need to require the compatibility condition   $\argmin_{z \in \R} Q_S^1(q_{11},z) = 0 $. We do not provide rigorous proofs for this setting, but refer to \cite[Subsection 2.3]{MFLMDimension2D1D} where such a model has  been discussed in a related \QQQ two-dimensional \EEE setting, and proofs have been provided (see \cite[Lemma~5.3, Remark 5.5(b)]{MFLMDimension2D1D}).
\EEE } 
\end{rem}

\EEE

\section{Proof of the main theorems}\label{sec: mianmain}
In this section, we prove the main results.

\begin{proof}[Proof of Proposition~\ref{maintheorem1}]
	(i) is guaranteed by the direct method of the Calculus of Variations. \QQQ Indeed, let $Y_{h,\tau_k}^{n-1}$, $n \geq1$ be given. Then, the coercivity of $\Phi_h(\tau_k, Y_{h,\tau_k}^{n-1},\cdot)$ with respect to the weak $W^{2,p}(\Omega;\R^3)$-topology is a consequence of \eqref{assumptions-P}(iii), Lemma~\ref{lem:rigidity}(v), \eqref{assumption:clampedboundary}, and Poincaré's inequality. Note that, if forces are included, one needs to employ Lemma~\ref{lem:compactnessfornonzeroforces}. As $p>3$, the weak lower semicontinuity follows by the compact embedding $W^{1,\infty}(\Omega)\subset\subset W^{2,p}(\Omega)$, the dominated convergence theorem, and the convexity of $P$, see \eqref{assumptions-P}(ii).  \EEE To see (ii), we note that the assumptions of Theorem~\ref{th:abstract convergence 2} (using $\phi_h$ in place of  $\phi_k$ and $\phi_0$) are satisfied for the fixed metric space $(\mathscr{S}^{3D}_{h,M},\mathcal{D}_h)$,  due to Lemma~\ref{th: metric space} and Lemma~\ref{lem:stronguppergradient3d}, where we use the weak $W^{2,p}(\Omega;\R^3)$-topology for the topology of convergence.  
\end{proof}\LLL Now we give the proof of Theorem~\ref{maintheorem3}, addressing the relation of the three-dimensional and the one-dimensional model.\EEE

\begin{proof}[Proof of Theorem~\ref{maintheorem3}]
\AAA First, \EEE	(i) corresponds to the construction of a recovery sequence in the static setting and is addressed in Theorem~\ref{th: Gamma}(ii).
	To prove (ii) and (iii), we need to check that the assumptions of Theorem~\ref{th:abstract convergence 2} and Theorem~\ref{thm: sandierserfaty} are satisfied.
The spaces $(\mathscr{S}^{3D}_{h,M},\mathcal{D}_{h})$ and $(\mathscr{S}^{1D},\mathcal{D}_{0})$ are complete metric spaces due to Lemma~\ref{th: metric space}(i) and
	Lemma~\ref{lem:complete1d}(i).
	  \EEE   Moreover, \LLL Proposition~\ref{lemma:compactness} \EEE yields \eqref{basic assumptions2}   and Theorems~\ref{th: Gamma}, \ref{th: lscD}, and \ref{theorem: lsc-slope} give \eqref{compatibility} and \eqref{eq: implication}. \EEE Furthermore, the slopes are strong upper gradients due to Lemma~\ref{lem:stronguppergradient3d}(i) and Lemma~\ref{Thm:representationlocalslope1}(iv). We note that the energies of the curves of maximal  \AAA slope \EEE are uniformly bounded depending only on the initial data, \ZZZ see~\eqref{maximalslope}. Concluding, we can apply Theorem~\ref{th:abstract convergence 2} and Theorem~\ref{thm: sandierserfaty}, respectively. This \EEE yields the existence of a curve of maximal slope $(u,\theta)$ for $\phi_0$ with respect to $\vert \partial \phi_0 \vert_{\mathcal{D}_0}$ and the convergence \AAA in \eqref{eq:ABC1} and \eqref{eq:ABC2} hold \EEE for $t\geq 0$, up \EEE to a subsequence.

	  	It remains to show \AAA strong convergence. Without restriction, we give the argument for time-continuous evolutions \EEE $u^h(t)$.  From now on, we drop the $t$ dependence, for the reader's convenience. First, we observe that Theorem~\ref{th:abstract convergence 2} and Theorem~\ref{thm: sandierserfaty} also guarantee convergence of energies, i.e., $ \phi_h(y^h) \to {\phi}_0(u,\theta)$. By Lemma~\ref{lemma: metric space-properties}(iii), we obtain  $\int_\Omega \tfrac{1}{2} Q_W^3(G^h(y^h)) \to \phi_0(u,\theta)$.  As $G \mapsto \int_\Omega \tfrac{1}{2}Q_W^3(G)$ is lower semicontinuous with respect to the weak $L^2(\Omega)$-topology, we see that
	$\int_\Omega \tfrac{1}{2} Q_W^3(G_y) \leq \phi_0(u,\theta)$. The reverse inequality is a consequence of the proof of Theorem~\ref{th: Gamma}(i), \SSS see \eqref{integrationx_2}--\eqref{integrationx_3}. \EEE This implies that $\int_\Omega \tfrac{1}{2} Q_W^3(G^h(y^h)) \to \int_\Omega \tfrac{1}{2} Q_W^3(G_y) $. Integrating the expansion
	\begin{align*}
	Q_W^3(G^h(y^h))-Q_W^3(G_y) - 2 \C_W^3[G^h(y^h)-G_y,G_y] = Q_W^3(G^h(y^h)-G_y),
	\end{align*}
 taking the limits, \AAA using \EEE the weak $L^2$-convergence of $G^h(y^h)$, and the positive definiteness of $Q_W^3$ on $\R^{3\times 3}_{\sym}$ \AAA we get that \EEE the convergence of $\big(\sym(G^h(y^h))\big)_h$ also holds strongly in $L^2(\Omega; \AAA \R^{3 \times 3} \EEE )$. In particular, as $\Vert R^h - \Id \Vert_{L^\infty(S)} \to 0$ by Lemma~\ref{lem:rigidity}\ref{lem:rigidity:dev:RotIDInfty}, \SSS Lemma \ref{lem:rigidity}\ref{lem:rigidity:energycontrol} \AAA yields \EEE that $\sym(R^h G^h(y^h)) \to \sym(G_y)$ in $L^2(\Omega;\R^{3 \times 3} )$. Moreover,  by Korn's inequality, \eqref{eq:ucharacterizationlimit}, \SSS and the structure of $u$ (see \eqref{def:Bernoulli-Navier}) \EEE we get \EEE  
	\begin{align*}
	\Vert u^h - u \Vert_{W^{1,2}(\Omega)} \leq &C \Vert \MMM \sym  (\nabla u^h)  -\partial_1 u_1 e_1 \otimes e_1 \EEE \Vert_{L^2(\Omega)} \\
	\leq & C  \left\Vert \sym \left( \tfrac{\nabla_h y^h - \Id}{\eps_h} \right)_{11} - \partial_1 u_1  \right\Vert_{L^2(\Omega)}+C h \left\Vert \sym \left( \tfrac{\nabla_h y^h - \Id}{\eps_h} \right)  \right\Vert_{L^2(\Omega)}
	\end{align*}
\AAA where we used the definition of $\nabla_h$. \EEE The right-hand side of the \MMM identity \EEE $
	\sym \left( \frac{\nabla_h y^h - \Id}{\eps_h}\right) = 		 \MMM \sym ( R^h G^h) \EEE  + \sym \left( \frac{R^h - \Id}{\eps_h}\right)$ converges   strongly to $\sym (G_y + \tfrac{r}{2} A_{u,\theta}^2 )$ in $L^2(\Omega;\R^{3 \times 3} )$, 	due to Lemma~\ref{lem:identificationofstrainlimit}(ii). Additionally, Lemma~\ref{lem:identificationofstrainlimit}(i) \AAA and \eqref{rep:Asquared} \EEE provide the characterization $(G_y + \tfrac{r}{2} A_{u,\theta}^2)_{11} = \partial_1 u_1$, which concludes the proof.
\end{proof}

%
%
%
\EEE
\appendix
\SSS 

\section{Some elementary lemmata}\label{sec:Appendix}

In this section we collect elementary lemmata on the energies and the  dissipation distances. \AAA We also provide a compactness result in the presence of forces. \EEE We start with the proof of Lemma~\ref{lemma: metric space-properties}. To this end, \EEE  we set for shorthand $H_Y := \frac{1}{2}\partial^2_{F_1^2} D^2(Y,Y) = \frac{1}{2}\partial^2_{F_2^2} D^2(Y,Y)$  for $Y \in \R^{3 \times 3}$  in a neighborhood of $SO(3)$. Given a deformation $y \in \mathscr{S}^{3D}_{h,M}$, we also introduce the mapping $H_{\nabla_h y}\colon \Omega \to \R^{3 \times 3 \times 3 \times 3}$   by $H_{\nabla_h y}(x) = H_{\nabla_h y(x)}$ for $x \in \Omega$. \BBB Note that this is well defined for $h$ sufficiently small by \eqref{eq:rigidity}\ref{lem:rigidity:dev:GradIDInfty} and \eqref{eq: assumptions-D}(iv). Viewing $H_Y$ as a bounded operator mapping from $\R^{3\times 3} \to \R^{3 \times 3}$ with norm  $\Vert H_Y \Vert_\infty := \sup_{0 \neq A \in \R^{3\times 3} } \vert H_Y A\vert / \vert A \vert$, we see that $Y \to H_Y$ is Lipschitz near $SO(3)$ satisfying 
\begin{align}\label{eq: Lip property}
\Vert H_{Y_1} - H_{Y_2}\Vert_\infty  \leq C \SSS | Y_1 - Y_2 | \EEE.
\end{align}
\begin{proof}[\hypertarget{proof:dissipationandenergy}{Proof of Lemma~\ref{lemma: metric space-properties}}]
	As a preparation, we observe that by the uniform bound on $\nabla_h y_0$, $\nabla_h y_1$ (see \eqref{eq:rigidity}\ref{lem:rigidity:dev:GradIDInfty}) and  a Taylor expansion \BBB at $(\nabla_h y_0,\nabla_h y_0)$ \EEE we obtain for all open subsets $U \subset \Omega$
	\begin{align}\label{eq: Taylor1}
	\Big|\int_U D^2(\nabla_h y_0, \nabla_h y_1) \QQQ \, {\rm d}x \EEE - \int_U H_{\nabla_h y_0}[\nabla_h (y_1 -  y_0),\nabla_h (y_1 -   y_0) ] \QQQ \, {\rm d}x \EEE \Big| \le C \Vert \nabla_h (y_1-   y_0) \Vert^3_{L^3(U)}.
	\end{align}
	We define $G(y_i) \BBB := \EEE \eps_hG^h(y_i) \BBB = R(y_i)^\top \nabla_h y_i - \Id\EEE$, $i=0,1$, for convenience. Using the separate frame indifference  \eqref{eq: assumptions-D}(v) we have 
	$$\int_\Omega D^2(\nabla_h y_0, \nabla_h y_1) \,{\rm d}x = \int_\Omega D^2\big(R(y_0)^\top\nabla_h y_0, R(y_1)^\top \nabla_h y_1\big) \,{\rm d}x.$$
	Thus, by \BBB $\eps_h^2 \mathcal{D}_h(y_0,y_1)^{2} =  \int_\Omega D^2(\nabla_h y_0, \nabla_h y_1) $ and  \EEE  again by Taylor expansion we also get
	\begin{align}\label{eq: Taylor2}
	\Big|\eps_h^2 \QQQ \mathcal{D}_h^2 \EEE (y_0,y_1) - \int_\Omega H_{R(y_0)^\top\nabla_h y_0} [G(y_1) - G(y_0), G(y_1) - G(y_0)] \QQQ \, {\rm d}x \EEE \Big|\le C \Vert   G(y_1) - G(y_0) \Vert^3_{L^3(\Omega)}. 
	\end{align} 
	
	We now show (i). By \SSS \eqref{eq: Lip property} \EEE and \eqref{eq:rigidity}\ref{lem:rigidity:dev:GradIDInfty}  we get $\Vert H_{\nabla_h y_0} -   \C^3_D   \Vert_\infty \le C(\eps_h/\delta_h)^\alpha$, where  $\C^3_D$ is the \QQQ fourth-order \EEE tensor associated to the quadratic form $Q^3_D $. Therefore, we obtain
	$$\Big|\int_U H_{\nabla_h y_0}[\nabla_h (y_1 -  y_0),\nabla_h (y_1 -   y_0) ] \QQQ \, {\rm d}x \EEE- \int_U Q^3_D(\nabla_h y_1 -  \nabla_h y_0) \QQQ \, {\rm d}x \EEE \Big| \le C(\eps_h/\delta_h)^\alpha \Vert \nabla_h y_1 - \nabla_h y_0 \Vert^2_{L^2(U)}$$
	\BBB for all open $U \subset \Omega$. \EEE By using \eqref{eq: Taylor1} and again \eqref{eq:rigidity}\ref{lem:rigidity:dev:GradIDInfty} we  get (i).
	
	To see (ii), we observe \BBB $\Vert H_{R(y_0)^\top\nabla_h y_0} - \C^3_D\Vert_\infty\le C\Vert \SSS R(y_0)^\top\nabla_h y_0 - \Id \EEE \Vert_\infty \le C(\eps_h/\delta_h)^{\alpha}$ by \SSS \eqref{eq: Lip property}, \QQQ \eqref{eq:rigidity}\ref{lem:rigidity:dev:RotIDInfty}, and  \eqref{eq:rigidity}\ref{lem:rigidity:dev:GradIDInfty}. \EEE  Thus,  we get  
	\begin{align}\label{eq: Taylor3}
	\Big|\int_\Omega  H_{R(y_0)^\top\nabla_h y_0}[G(y_1) - G(y_0),G(y_1) - G(y_0) ]\QQQ \, {\rm d}x \EEE & - \int_\Omega Q^3_D \big(G(y_1) - G(y_0) \big) \QQQ \, {\rm d}x \EEE\Big| \notag\\& \le C(\eps_h/\delta_h)^\alpha \Vert G(y_1) - G(y_0) \Vert_{L^2(\Omega)}^{2}.
	\end{align}
	In a similar fashion, \QQQ \eqref{eq:rigidity}\ref{lem:rigidity:dev:RotIDInfty} and  \eqref{eq:rigidity}\ref{lem:rigidity:dev:GradIDInfty} \EEE also imply $\Vert G(y_i) \Vert_{L^\infty(\Omega)} \le C(\eps_h/\delta_h)^\alpha$ for $i=0,1$ and thus 
	\begin{align*}
	\Vert   G(y_1) - G(y_0) \Vert^3_{L^3(\Omega)}  \le C (\eps_h/\delta_h)^\alpha \Vert G(y_1) - G(y_0) \Vert_{L^2(\Omega)}^{2}.
	\end{align*}
	This together with \eqref{eq: Taylor2} \QQQ and \EEE \eqref{eq: Taylor3} (divided by $\eps_h^2$), and $G^h(y_i) = \eps_h^{-1}G(y_i)$ for $i=0,1$ yields 
	$$
	\Big|\mathcal{D}_h(y_0,y_1)^2 -  \int_\Omega  Q^3_D \big(G^h(y_0) - G^h(y_1)\big)  \QQQ \, {\rm d}x \EEE\Big| \le C (\eps_h/\delta_h)^\alpha \Vert G^h(y_0) - G^h(y_1) \Vert^2_{L^2(\Omega)}.
	$$
	This  shows the first inequality of (ii). To see the second inequality, we use \eqref{eq:rigidity}\ref{lem:rigidity:energycontrol}.
	
	We now show (iii) and (iv).  We use the frame indifference of $W$ and  \cite[Lemma 4.1(iii)]{MFMKDimension} (with $F_i = R(y_i)^\top\nabla_h y_i = \Id + G(y_i)$ for $i=0,1$)   to obtain  
	\begin{align*}
	|\Delta(y_1) - \Delta(y_0)|  & \le C\eps_h^{-2} \sum\nolimits_{k=1}^3  \int_\Omega |R(y_0)^\top\nabla_h y_0-\Id|^{3-k}|R(y_1)^\top\nabla_h y_1 - R(y_0)^\top\nabla_h y_0|^k \QQQ \, {\rm d}x \EEE \\
	&  = C\eps_h^{-2}  \sum\nolimits_{k=1}^3  \int_\Omega|G(y_0)|^{3-k}|G(y_1) - G(y_0)|^k \QQQ \, {\rm d}x \EEE \\
	& \le C\eps_h^{-2} \int_\Omega (|G(y_1)| + |G(y_0)|)^2|G(y_1) - G(y_0)| \QQQ \, {\rm d}x \EEE,
	\end{align*}
	where  $\Delta(y_0)$ and $\Delta(y_1)$ \EEE are defined in the statement of the lemma. The fact that $\Vert G(y_i) \Vert_{L^\infty(\Omega)} \le C(\eps_h/\delta_h)^\alpha$ for $i=0,1$ and H\"older's inequality yield
	$$|\Delta(y_1) -  \Delta(y_0)| \le C \eps_h^{-2} (\eps_h/\delta_h)^\alpha \big(\Vert   G(y_0) \Vert_{L^2(\Omega)} + \Vert G(y_1)   \Vert_{L^2(\Omega)} \big) \, \Vert G(y_1) - G(y_0) \Vert_{L^2(\Omega)}. $$
	Using $G^h(y_i) = \eps_h^{-1}G(y_i)$ for $i=0,1$ and \eqref{eq:rigidity}(i) we obtain the first inequality of (iv). The second inequality follows again by \eqref{eq:rigidity}(i). Finally, to see (iii), we apply (iv)   for $y_0 = y$ and $y_1 = \id$, where we use $\Delta(y_1) = 0$. \EEE 
\end{proof}

\SSS

\begin{lemma}[Hessian of $D^2$]\label{lemma: ele}
For all $F \in \R^{3 \times 3}$ in a neighborhood of \AAA $SO(3)$ \EEE and all $G \in \R^{3 \times 3}$  it holds that    ${\rm Ker} (\partial^2_{F_1^2} D^2(F, F) ) = F^{-\top} \R_{\rm skew}^{3 \times 3}$ and  
\begin{align}\label{eq:quadraticlowerbound-new}
\partial^2_{F_1^2} D^2(F, F) [G,G]  \ge c|{\rm sym}(F^\top G)|^2
\end{align}
\AAA for some universal $c>0$. \EEE  In particular, the potential $R(F,\dot F)$ introduced in   \eqref{intro:R}  depends only on  $C:= F^{\top}F$ and $\dot C= \dot{F}^{\top}F+F^{\top}\dot F$,   it is quadratic in $\dot C$, and satisfies \eqref{eq:quadraticlowerbound}.
\end{lemma}  
\begin{proof}
We choose $F\QQQ\in \R^{3\times 3}\EEE$ near \AAA $SO(3)$ \EEE and let $G \in \R^{3 \times 3}$. \QQQ Suppose \EEE $\eps >0$. By Polar decomposition, we can write $U = QF$ for $U = \sqrt{F^\top F}$. 
\QQQ Then, exploiting \EEE the frame indifference \eqref{eq: assumptions-D}(v) for  $Q_1 = Q$ and $Q^\eps_2 = Q \exp(\eps A)$, where $A = \frac{1}{2}(F^{-\top} G^\top - G F^{-1}) \in \R_{\rm skew}^{3\times 3}$, \QQQ yields \EEE
\begin{align}\label{eq:quadraticlowerbound-0}
\frac{1}{2\eps^2}D^2(F + \eps G, F) &= \frac{1}{2\eps^2}D^2(Q^\eps_2 (F + \eps G), Q_1 F )  = \frac{1}{2\eps^2}D^2(Q^\eps_2 (F + \eps G), U).
\end{align}
By using \QQQ a Taylor expansion, \EEE the definition of \QQQ  $A$, \EEE and $Q = UF^{-1}$ we calculate
\begin{align*}
\QQQ Q^\eps_2 \EEE (F + \eps G) & = Q (\Id + \eps A + \mathcal{O}(\eps^2)) (F + \eps G)  = U + \eps (QAF + QG) + \mathcal{O}(\eps^2) \\
&    = U  + \eps U^{-1}  U^2(F^{-1}AF + F^{-1}G) + \mathcal{O}(\eps^2) \\ 
& = U  + \eps U^{-1}  F^\top F   \frac{F^{-1}F^{-\top}G^\top F +   F^{-1}G }{2} +  \mathcal{O}(\eps^2) \\& = U  + \eps U^{-1}  {\rm sym}(F^\top G) +  \mathcal{O}(\eps^2).
\end{align*}
Plugging this into  \eqref{eq:quadraticlowerbound-0} and using a Taylor expansion we deduce 
\begin{align*}
\frac{1}{\eps^2}D^2(F + \eps G, F) 
& = \frac{1}{2}\partial^2_{F_1^2} D^2(U,U) [U^{-1}  {\rm sym}(F^\top G),U^{-1}  {\rm sym}(F^\top G)] + \mathcal{O}(\eps)  \notag \\
& =  \frac{1}{2}    (U^{-\top} \partial^2_{F_1^2} D^2(U,U) U^{-1})  [ {\rm sym}(F^\top G), {\rm sym}(F^\top G)]  + \mathcal{O}(\eps).
\end{align*}  
As $\partial^2_{F_1^2} D^2(\Id,\Id)$ is positive definite on $\R_{\rm sym}^{3\times 3}$ by \AAA \eqref{eq: assumptions-D}(vi), \EEE \QQQ it follows by a standard  continuity argument \EEE that also $U^{-\top} \partial^2_{F_1^2} D^2(U,U) U^{-1}$ is positive definite on $\R_{\rm sym}^{3\times 3}$ for all $F$ sufficiently close to \AAA $SO(3)$. \EEE By sending $\eps \to 0$ this shows \eqref{eq:quadraticlowerbound-new}. Using $G = \dot F$ in the above formula, we find that $R(F,\dot F)$ depends only on $U = \sqrt{C}$ and ${\rm sym}(F^\top G) = \frac{1}{2}\dot C$, and that it is quadratic in $\dot C$. Finally, \eqref{eq:quadraticlowerbound}  follows from  \eqref{eq:quadraticlowerbound-new}.
\end{proof} 
 \EEE

\begin{lemma}[Positivity of $\mathcal{D}_h$]\label{lemma: positivity}
Let $M>0$ and let $h$ sufficiently small.   Let $y_0, y_1 \in \mathscr{S}^{3D}_{h,M}$ with $\mathcal{D}_h(y_0,y_1) = 0$. Then $y_0 = y_1$.
\end{lemma}

\begin{proof} \QQQ Instead of working on $\Omega$, we introduce the rescaled variables $z_i = y_i \circ p_h^{-1}$, $i =0,1$ \AAA with the projection \EEE $p_h\colon \Omega \to \Omega_h$.
Consider the family of pairwise disjoint cubes, defined by
	\begin{align*}
	Q_h(i,j):= (-\tfrac{l}{2} + (i-1)\delta_h, \AAA -  \tfrac{l}{2}+i \delta_h) \EEE \times (-\tfrac{h}{2} + (j-1) \delta_h, -\tfrac{h}{2} + j \delta_h) \times (-\tfrac{\delta_h}{2} ,\tfrac{\delta_h}{2})
	\end{align*}
	for $i =1,..., N_1^h:= \SSS \lfloor l/\delta_h\rfloor\EEE $ and $j = 1,..., N_2^h:= \lfloor h/\delta_h \rfloor$. Then we have $\Omega_h= \bigcup_{i = 1}^{N_1^h}\bigcup_{j = 1}^{N_2^h} \SSS Q_h(i,j) \EEE \cup \AAA \tilde Q_h \EEE$ for  a small set $\tilde Q_h$. Denote the family of cubes by $\SSS \mathcal{Q}_h\EEE$. We now first show that $z_0 = z_1$ on each $Q_h(1,j)$ for each $j = 1,..., N_2^h$. \AAA To this end, fix $j$ and denote $Q_h(1,j)$ by $Q$ for convenience. \EEE Assuming $\int_{\Omega_h} D^2(\nabla z_0, \nabla z_1) =0$, we get by a Taylor expansion and \eqref{intro:R}
	\begin{align}\label{eq: positivity1pompe}
	 \int_{Q} R(\nabla z_0,\nabla z_1 -  \nabla z_0)\LLL \,{\rm d}x \EEE \le C \Vert \nabla  z_1-  \nabla  z_0 \Vert_{L^\infty(Q)} \Vert \nabla  z_1-  \nabla  z_0 \Vert^2_{L^2(Q)}. 
	\end{align}
	Since $z_1 = z_0$ on $\partial I \times (-\frac{h}{2}, \frac{h}{2})\times (-\frac{\delta_h}{2}, \frac{\delta_h}{2})$, we get that $z_1 = z_0$ on at least one face of $\partial Q$. Then Theorem~\ref{pompe} and \eqref{eq:quadraticlowerbound} imply
	\begin{align}\label{eq: positivity2pompe}
	\Vert \nabla z_1 - \nabla z_0 \Vert^2_{L^2(Q)} \le C \Vert {\rm sym}\big({\nabla z_0}^{\top}(\nabla z_1 - \nabla z_0) \big)\Vert^2_{L^2(Q)} \leq C 	 \int_{Q} R(\nabla z_0,\nabla z_1 -  \nabla z_0)\LLL \,{\rm d}x \EEE  .
	\end{align}
	As $\Vert \nabla  z_1-  \nabla  z_0 \Vert_{L^\infty(Q)} \to 0$ for $h \to 0$ by Lemma~\ref{lem:rigidity}\ref{lem:rigidity:dev:GradIDInfty}, \eqref{eq: positivity1pompe} and \eqref{eq: positivity2pompe} show $\nabla z_1 = \nabla z_0$ a.e.\ on $Q$, for $h$ sufficiently small. Since  $z_1 = z_0$ on at least one face of $\partial Q$, this also gives $z_1 = z_0$ a.e.\ on $Q$, as desired. 
	We now proceed \BBB iteratively \EEE to show that $z_1 = z_0$ on each $Q \in \mathcal{Q}_h$: suppose that the property has already been shown for all $Q \in\bigcup_{i = 1}^{m}\bigcup_{j = 1}^{N_2^h} Q_h(i,j)$. Then $z_1 = z_0$ on each $Q \in \bigcup_{j = 1}^{N_2^h} Q_h(m+1,j)$ follows from the above arguments noting that $z_1 = z_0$ on at least one face of  $\partial Q$ since $z_1 = z_0$ on all cubes  $Q \in \bigcup_{i = 1}^{m}\bigcup_{j = 1}^{N_2^h} Q_h(i,j)$. 
	This shows $z_1 = z_0$ on $\Omega_h \setminus \tilde Q_h$. Our arrangement of cubes ensures that one corner of $\Omega_h$ is covered. Rearranging the cubes by starting in the remaining corners and repeating the arguments from above yield that $z_1 = z_0$ on whole $\Omega_h$. 
	\end{proof}
	
	\AAA Finally, for \EEE the reader's convenience, we briefly give the most important adaption if nonzero forces are considered.
	  
	\begin{lemma}[Compactness if $f_h^{3D}\neq 0$]\label{lem:compactnessfornonzeroforces}
		Suppose that $f_h^{3D} \neq 0$ satisfies \eqref{eq: forces} and let $(y^h)_h$ be a sequence with $y^h \in \mathscr{S}^{3D}_{h,M}$ for some $M>0$.
	\begin{itemize}
		\item[(i)] 	Then, there exists a constant $C=C(M)$ such that 
		$$I_h(y^h) := \tfrac{1}{\eps_h^2} \int_\Omega W(\nabla_h y^h(x)) \, \MMM  {\rm d}x+\tfrac{\zeta_h}{\eps_h^2} \int_\Omega P(\nabla^2_h y^h(x)) \, \MMM  {\rm d}x \EEE \leq C.$$ 
		\item[(ii)] There exists a constant $c = c(M) \in \R$ such that $\phi_h(y^h)\geq c > -\infty$.
%
	\end{itemize}	
	\end{lemma}
	In particular, \AAA uniform bounds on $\phi_h$ imply uniform bounds on $I_h$, \EEE and all proofs concerning compactness properties are applicable. 
%
	\begin{proof}[Proof of Lemma~\ref{lem:compactnessfornonzeroforces}]
		Let $M>0$ and \AAA let \EEE $(y^h)_h$ be a sequence with $y^h \in \mathscr{S}^{3D}_{h,M}$.
		Note that by Hölder's inequality  we have
		\begin{align}\label{compactnessforce1}
		\tfrac{1}{\eps_h^2} \int_{\Omega}  f^{3D}_h(x_1)  y_3^h(x) \, \MMM  {\rm d}x \leq \Vert \tfrac{1}{\eps_h\delta_h}  f^{3D}_h \Vert_{L^2(I)} \tfrac{\delta_h}{\eps_h} \Vert  (y_3^h-\delta_hx_3) \Vert_{L^2(\Omega)} .
		\end{align}
\AAA Thus, \EEE Lemma~\ref{lem:rigidity}(iv) yields
		\begin{align*}
		\Vert \nabla_h y^h - \Id\Vert_{L^2(\Omega)} \leq C \eps_h/\delta_h \sqrt{ I_h(y^h) },
		\end{align*}	
		\AAA where we also refer to \cite[Equation (68)]{Freddi2013} for the scaling of  $I_h(y^h)$ on the right-hand side. \EEE In particular, this estimate and Poincar\'e's inequality imply
		\begin{align}\label{compactnessforce3}
		&\Vert y_3^h - \delta_h x_3 \Vert_{L^2(\Omega)} \leq \Vert y_3^h - \delta_h x_3 - \eps_h/\delta_h \hat{\xi}_3 \Vert_{L^2(\Omega)} + C \eps_h/\delta_h \leq C \Vert \nabla_h y^h - \Id \Vert_{L^2(\Omega)}  + C \eps_h/\delta_h \notag\\
		&\qquad\leq C \eps_h/\delta_h \sqrt{I_h(y^h)} + C \eps_h/\delta_h ,
		\end{align}
		where we have used that $y_3^h = \delta_h x_3 + \eps_h/\delta_h \hat{\xi}_3$ on $\Gamma$, see \eqref{assumption:clampedboundary}. 
		Combining \eqref{compactnessforce1} and \eqref{compactnessforce3} yields
		\begin{align}\label{compactnessforce4}
		I_h(y^h) =& \phi_h(y^h) +\tfrac{1}{\eps_h^2} \int_{\Omega}  f^{3D}_h(x_1)  y_3^h(x) \, \MMM  {\rm d}x \notag\\
		\leq & M + \Vert \tfrac{1}{\eps_h\delta_h}f_h^{3D}\Vert_{L^2(I)}C  \sqrt{I_h(y^h)} + C \Vert \tfrac{1}{\eps_h\delta_h}f_h^{3D}\Vert_{L^2(I)}.
		\end{align}
\AAA Thus, \EEE (i) holds as $\Vert \tfrac{1}{\eps_h\delta_h}f_h^{3D}\Vert_{L^2(I)}$ is bounded, due to \eqref{eq: forces}. Similarly, (ii) follows by using the equality in \eqref{compactnessforce4} together with \eqref{compactnessforce1} and \eqref{compactnessforce3}.
	\end{proof}
\begin{rem}[Horizontal forces]\label{rem:horizontalforces}
{\normalfont
		Denoting by $(f^{3D}_{1,h},f^{3D}_{2,h})$ a vector of horizontal forces one \AAA would have \EEE to assume that $\Vert \tfrac{1}{\eps_h\delta_h}f_{i,h}^{3D}\Vert_{L^2(I)}$, $i=1,2$, are uniformly bounded in order to guarantee compactness, see \eqref{compactnessforce4} for the analogous computation. However, this implies that $f_{1,h}^{3D}/ \eps_h$ and $f_{2,h}^{3D}/ (\eps_h h)$ necessarily converge to $0$. In particular, these forces would not affect the $\Gamma$-limit, due to \eqref{def:udisplacement} and Proposition~\ref{lemma:compactness}.}
\end{rem}

\section*{Acknowledgements} 
This work was funded by  the DFG project FR 4083/5-1 and  by the Deutsche Forschungsgemeinschaft (DFG, German Research Foundation) under Germany's Excellence Strategy EXC 2044 -390685587, Mathematics M\"unster: Dynamics--Geometry--Structure.

 \typeout{References}

\end{document}